\documentclass[10pt]{article}

\usepackage{url}
\usepackage{mathtools}
\usepackage{amssymb}
\usepackage{amsthm}
\usepackage{empheq}
\usepackage{latexsym}
\usepackage{enumitem}
\usepackage{eurosym}
\usepackage{dsfont}
\usepackage{appendix}
\usepackage{color} 
\usepackage[unicode]{hyperref}
\usepackage{frcursive}
\usepackage[utf8]{inputenc}
\usepackage[T1]{fontenc}
\usepackage{geometry}
\usepackage{multirow}
\usepackage{todonotes}
\usepackage{lmodern}
\usepackage{anyfontsize}
\usepackage{stmaryrd}
\usepackage{natbib}
\usepackage{cleveref}

\usepackage{amsbsy}

\setcitestyle{numbers,open={[},close={]}}

\definecolor{red}{rgb}{0.7,0.15,0.15}
\definecolor{green}{rgb}{0,0.5,0}
\definecolor{blue}{rgb}{0,0,0.7}
\hypersetup{colorlinks, linkcolor={red},citecolor={green}, urlcolor={blue}}
			
\makeatletter \@addtoreset{equation}{section}

\newtheorem{theorem}{Theorem}[section]
\newtheorem{assumption}[theorem]{Assumption}
\newtheorem{corollary}[theorem]{Corollary}

\newtheorem{lemma}[theorem]{Lemma}
\newtheorem{proposition}[theorem]{Proposition}

\newtheorem{definition}[theorem]{Definition}
\newtheorem{remark}[theorem]{Remark}

\DeclareUnicodeCharacter{014D}{\=o}
\def \E{\mathbb{E}}
\def \F{\mathbb{F}}
\def \G{\mathbb{G}}
\def \H{\mathbb{H}}

\def \M{\mathbb{M}}
\def \N{\mathbb{N}}

\def \P{\mathbb{P}}
\def \Q{\mathbb{Q}}
\def \R{\mathbb{R}}
\def \S{\mathbb{S}}

\def \Z{\mathbb{Z}}

\def \Pr{\mathrm{P}}

\def\Ac{{\cal A}}
\def\Bc{{\cal B}}
\def\Cc{{\cal C}}

\def\Fc{{\cal F}}
\def\Gc{{\cal G}}
\def\Hc{{\cal H}}

\def\Kc{{\cal K}}
\def\Lc{{\cal L}}

\def\Pc{{\cal P}}

\def\Rc{{\cal R}}
\def\Sc{{\cal S}}

\def\Vc{{\cal V}}
\def\Wc{{\cal W}}

\def\Fb{{\bar F}}
\def\Gb{{\overline \G}}
\def\Gcb{\overline \Gc}

\def\x{\times}

\def\Om{\Omega}
\def\Omt{\widetilde{\Omega}}
\def\Omh{\widehat{\Omega}}
\def\om{\omega}

\def\Omb{\overline{\Om}}

\def\Fb{\overline{\F}}
\def\Gcb{\overline{\Gc}}
\def\Fcb{\overline{\Fc}}
\def\Bt{\widetilde{B}}

\def\Gct{\widetilde{\Gc}}
\def\Fct{\widetilde{\Fc}}
\def\Ft{\widetilde{\F}}
\def\Gt{\widetilde{\G}}

\def\Xt{\widetilde{X}}

\def\0{\mathbf{0}}

\def \bb{\mathbf{b}}

\def \qb{\mathbf{q}}
\def \xb{\mathbf{x}}

\def \mub{\overline{\mu}}

\def \muh{\widehat{\mu}}

\def \mut{\widetilde{\mu}}
\def \nub{\bar{\nu}}

\def\normeL2#1{\left\|{#1}\right\|_{L^2}}

\def\Pcb{\overline \Pc}

\def \Prt{\widetilde{\Pr}}

\def\Bh{\widehat B}
\def\Mh{\widehat M}

\def\Xh{\widehat X}
\def\Wh{\widehat W}
\def\Zh{\widehat Z}

\def\gammah{\widehat \gamma}
\def\Phih{\widehat \Phi}

\def\Yt{\widetilde{Y}}

\def \Lim{\displaystyle\lim}
\def \Liminf{\displaystyle\liminf}
\def \Limsup{\displaystyle\limsup}

\def \xf{\mathsf{x}}

\def \alphab {\boldsymbol{\alpha}}

\def \Xbb{\mathbf{X}}

\def \Ybb{\mathbf{Y}}

\def \Qr{\mathrm{Q}}

 \title{Large population games with interactions through controls and common noise: convergence results and equivalence between $open$--$loop$ and  $closed$--$loop$ controls}

\author{
    Mao Fabrice {\sc Djete}\footnote{Ecole Polytechnique Paris, Centre de Math\'ematiques Appliqu\'ees, mao-fabrice.djete@polytechnique.edu.}
    }
             \date{\today}

\begin{document}

\maketitle
 
\begin{abstract}
    In the presence of a common noise, we study the convergence problems in mean field game (MFG) and mean field control (MFC) problem where the cost function and the state dynamics depend upon the joint conditional distribution of the controlled state and the control process. In the first part, we consider the MFG setting. We start by recalling the notions of $measure$--$valued$ MFG equilibria and of approximate $closed$--$loop$ Nash equilibria associated to the corresponding $N$--player game. Then, we show that all convergent sequences of approximate $closed$--$loop$ Nash equilibria, when $N \to \infty,$ converge to $measure$--$valued$ MFG equilibria. And conversely, any $measure$--$valued$ MFG equilibrium is the limit of a sequence of approximate $closed$--$loop$ Nash equilibria. In other words, $measure$--$valued$ MFG equilibria are the accumulation points of the approximate $closed$--$loop$ Nash equilibria. Previous work has shown that $measure$--$valued$ MFG equilibria are the accumulation points of the approximate $open$--$loop$ Nash equilibria. Therefore, we obtain that the limits of approximate  $closed$--$loop$ Nash equilibria and approximate $open$--$loop$ Nash equilibria are the same. In the second part, we deal with the MFC setting. After recalling the $closed$--$loop$ and $open$--$loop$ formulations of the MFC problem, we prove that they are equivalent. We also provide some convergence results related to approximate $closed$--$loop$ Pareto equilibria.
\end{abstract}
\section{Introduction}

\medskip
Our primarily goal in this paper is to discuss the convergence problem of $closed$--$loop$ Nash equilibria in the setting of mean field game of controls (MFGC) or extended mean field game. Let us briefly explained the mathematical framework that we consider. The full details explanation are given in \Cref{sec:finite-player-game}. 
We consider that $N$ players have private state processes $\Xbb^N:=(X^1,\dots,X^N)$ given by the stochastic differential equations (SDEs) system
\begin{align} \label{intro:Nplayer}
        \mathrm{d}X^{i}_t
        &= b\big(t,X^{i}_{t},\overline{\varphi}^{N}_{t}[\alphab^N] ,\alpha^i(t,\Xbb^N_t) \big) \mathrm{d}t 
        +
        \sigma \big(t,X^{i}_{t} \big) \mathrm{d}W^i_t
        +
        \sigma_0 \mathrm{d}B_t,\;t \in [0,T],
        \\
        \alphab^N&:=(\alpha^1,\cdots,\alpha^N),\;
        \overline{\varphi}^{N}_{t}[\alphab^N]:= \frac{1}{N}\sum_{i=1}^N \delta_{(X^{i}_{t},\alpha^i(t,\Xbb^N_t) )}~\mbox{and}~\varphi^{N}_{t}[\alphab]:= \frac{1}{N}\sum_{i=1}^N \delta_{X^{i}_{t}},
    \end{align}
where $T>0$ is a fixed time horizon, $(B,W^1,\dots,W^N)$ are independent Brownian motions where $B$ is called the common noise, and $\alpha^i$ is a Borel measurable function playing the role of the control of player $i$. An important feature here is the presence in the dynamics $X^i$ of player $i$ of the empirical distribution $\overline{\varphi}^{N}[\alphab^N]$ of states and controls of all players.
Given a strategy $(\alpha^1,\dots,\alpha^N),$ the reward to the player $i$ is 
\begin{align*}
    J^N_i(\alpha^1,\dots,\alpha^N)
    :=
    \E \bigg[
        \int_0^T L\big(t,X^{i}_{t},\overline{\varphi}^{N}_{t}[\alphab] ,\alpha^i(t,\Xbb^N_t) \big) \mathrm{d}t 
        + 
        g \big( X^{i}_{T}, \varphi^{N}_{T}[\alphab] \big)
        \bigg].
\end{align*}
For $\varepsilon_N \ge 0,$ the strategy $(\alpha^1,\dots,\alpha^N)$ will be called an $\varepsilon_N$--$closed$--$loop$ Nash equilibrium if for any admissible control $\beta,$ and each $i \in \{1,\cdots,N\},$
\begin{align} \label{eq:opti-player_i}
    J^N_i(\alpha^1,\dots,\alpha^N)
    \ge
    J^N_i(\alpha^1,...,\alpha^{i-1},\beta,\alpha^{i+1},\dots,\alpha^N) - \varepsilon_N.
\end{align}
The presence of the term $closed$--$loop$ indicates the fact that we consider controls which are Borel maps of $[0,T] \x (\R^n)^N$ into $U$. The convergence problem here consists in characterizing the Nash equilibria when the number of players $N$ goes to infinity.
It is now well known that, when $N$ tends to infinity, the Nash equilibria are related to the MFG here called mean field game of controls (MFGC) or Extended mean field game, which has the following structure (the precise definition is given in \Cref{section:strong_formulation}): for $\varepsilon \ge 0,$ a $(\sigma\{B_s,s \le t\})_{t \in [0,T]}$--predictable measure--valued process $(\mub_t)_{t \in [0,T]}$ is an $\varepsilon$--strong Markovian MFG equilibrium (or approximate strong Markovian MFG equilibrium) if for all $t \in [0,T],$ $\mub_t=\Lc(X_t,\alpha(t,X_t,\mub_t)|B),$ where the state process $X$ is governed by
\begin{align*}
		\mathrm{d}X_t
		&= 
		b \big(t, X_t,\mub_t, \alpha(t,X_t,\mub_t) \big) \mathrm{d}t
		+
	    \sigma\big(t, X_t \big) \mathrm{d} W_t
		+ 
		\sigma_0 \mathrm{d}B_t,,~~t \in [0,T]
		\\
		\mu_t&:=\Lc(X_t|B),
	\end{align*}
	and one has
	\begin{align} \label{def-epsilon-MFG}
		    \begin{cases} \displaystyle \E \bigg[
				\int_0^T L(t, X_t, \mub_t,\alpha(t,X_t,\mub_t)) \mathrm{d}t 
				+ 
				g(X_T, \mu_T) 
			\bigg]
			\ge
			\sup_{\alpha'}\;\E \bigg[
				\int_0^T L(t, X'_t, \mub_t,\alpha'(t,X'_t,\mub_t)) \mathrm{d}t 
				+ 
				g(X'_T, \mu_T) 
			\bigg] - \varepsilon,
\\[0.7em]
		    \displaystyle  \mbox{where  {\color{black} the optimization is over the solutions}}\;\; 
		    \mathrm{d}X'_t
		    = 
		    b \big(t, X'_t, \mub_t, \alpha'(t,X'_t,\mub_t) \big) \mathrm{d}t
		    +
	        \sigma\big(t, X'_t\big) \mathrm{d} W_t
		    + 
		    \sigma_0 \mathrm{d}B_t.
		    \end{cases}
		\end{align} 

\paragraph*{Discussion about the problematic and related studies}

The now well--known MFG was introduced in the seminal works of \citeauthor*{lasry2006jeux} \cite{lasry2007mean} and \citeauthor*{huang2003individual} \cite{huang2006large} as a way of studying the $N$--player game described here above, when the number of players $N$ is very large.
Since these pioneering works, this topic has been the subject of much research in the field of applied mathematics, in particular for its wide variety of application (see \citeauthor*{carmona2018probabilisticI} \cite{carmona2018probabilisticI} for examples of applications of MFG). The initial formulations of MFG did not consider the presence of the empirical distribution of controls like that formulated above in \Cref{intro:Nplayer}. Only the presence of the empirical distribution of states was envisaged. In order to deal with some modeling issues occurring in finance for instance, a natural $``$extension$"$ of MFG known as extended MFG or MFG of controls has been formulated and studied by many studies these recent years, see \citeauthor*{DiogoVardan-ExtMFG} \cite{DiogoVardan-ExtMFG}, \citeauthor*{carmona2015probabilistic} \cite{carmona2015probabilistic}, \citeauthor*{Lehalle-card} \cite{Lehalle-card}, {\color{black} \citeauthor*{GOMES201449} \cite{GOMES201449}, \citeauthor*{bonnans_pfeifer2019} \cite{bonnans_pfeifer2019}, \citeauthor*{P-JamesonGraber} \cite{P-JamesonGraber}, \citeauthor*{Alasseur2020} \cite{Alasseur2020}, \citeauthor*{Kobeissi} \cite{Kobeissi}. 
Our study in this paper treats of extended MFG (or MFG of controls) by giving some convergence and equivalence results. 

\medskip
As mentioned in the beginning, our primarily goal is to make a rigorous connection between the $N$--player game and the extended mean field game. More precisely, ideally, we want to show two main results. First, the convergence result i.e. given an $\varepsilon_N$--Nash equilibrium $\alphab^N:=(\alpha^{1,N},\cdots,\alpha^{N,N})$, the sequence of empirical distribution of states and controls $(\overline{\varphi}^N[\alphab^N])_{N \in \N^*}$ $``$converges$"$ to a solution of the MFG of controls when $N \to \infty,$ with $\Lim_{N \to \infty} \varepsilon_N=0$. Second, the $converse$ convergence result i.e. any solution of the MFG of controls is the $``$limit$"$, when $N \to \infty,$ of a sequence of empirical distribution of states and controls $(\overline{\varphi}^N[\alphab^N])_{N \in \N^*}$ associated to an $\varepsilon_N$--Nash equilibrium $\alphab^N:=(\alpha^{1,N},\cdots,\alpha^{N,N}),$ for some $\varepsilon_N$ satisfying $\Lim_{N \to \infty} \varepsilon_N=0.$   

\medskip
Establishing this type of connection justifies the interpretation of MFG as the right limit formulation of the $N$--player game. Besides the presence of the empirical distribution of states and controls, another  important feature of our setting is the consideration of $closed$--$loop$ controls i.e. controls depending on the position of the players. Indeed, the other type of controls usually considered is the $open$--$loop$ controls i.e. controls adapted to the filtration generated by the initial values and the Brownian motions. When $N$($ \ge 2,$ the number of players) is fixed, the situation generated by these two concepts of equilibrium is very different, see the discussion in \cite[Section 2.1.2]{carmona2018probabilisticI}.

\medskip
In the setting of $open$--$loop$ controls, the connection between $N$--player game and MFG (convergence result and $converse$ convergence result) is now well--known and established. In the situation without the empirical distribution of controls, under relatively general assumptions, a complete picture has been proposed by \citeauthor*{M-Fisher} \cite{M-Fisher} and \citeauthor*{lacker2016general} \cite{lacker2016general} (for the case with common noise) using the notions of relaxed MFG equilibrium. For the extended MFG, while allowing the volatility to be controlled, a similar study to \cite{lacker2016general} is provided by \citeauthor*{MFD-2020_MFG} \cite{MFD-2020_MFG} thanks to the notion of measure--valued MFG equilibrium. With stronger assumptions but by providing convergence rates, \citeauthor*{M_Lauriere-Tangpi} \cite{M_Lauriere-Tangpi} study these convergence problems in the situation without common noise using notably a notion of backward propagation of chaos.

\medskip
The convergence problems in the case of $closed$--$loop$ controls turn out to be much more problematic than in the case of $open$--$loop$ controls. Indeed, as discussed in \cite[Section 2.1.2]{carmona2018probabilisticI}, $open$--$loop$ Nash equilibrium and $closed$--$loop$ Nash equilibrium behave differently. 
Without taking into account the empirical distribution of states and controls (no extended MFG), a first answer of the convergence problems is provided by \citeauthor*{cardaliaguet2015master} \cite[Section 3.7]{cardaliaguet2015master}. They use notably an {\color{black}infinite dimensional PDE} associated to the limit problem, the so--called $master\;equation$. An important work is done to ensure that the master equation is correctly defined, in particular by using the now well--known Lasry--Lions monotonicity condition. The master equation appears to be a powerful tools to the study of MFG see for instance \citeauthor*{cardaliaguet2010notes} \cite{cardaliaguet2010notes}, \citeauthor*{GANGBO20156573} \cite{GANGBO20156573}, \citeauthor*{bensoussanYam-2019} \cite{bensoussanYam-2019}, \cite{DiogoVardan-ExtMFG}, \citeauthor*{carmona2018probabilisticII} \cite{carmona2018probabilisticII}, \cite{Kobeissi}, \citeauthor*{delarueLackerRamanan2020} \cite{delarueLackerRamanan2020,delarueLackerRamanan2020--2}, \citeauthor*{BAYRAKTAR202198} \cite{BAYRAKTAR202198}, \citeauthor*{bertucci_2021_mastercN} \cite{bertucci_2020_mastercN,bertucci_2021_mastercN}. Unfortunately, for the convergence problems, the need of uniqueness of the MFG equilibrium (through Lasry--Lions monotonicity condition) makes its use unsatisfactory. Indeed, it is well known that there are situations where there are multiple MFG equilibria. A breakthrough has been made by \citeauthor*{lacker2020-closed} \cite{lacker2020-closed} for the convergence problem of $closed$--$loop$ controls without requirement of uniqueness of MFG equilibrium. In a non--common noise setting while imposing a non--degenerate volatility $\sigma,$ \cite{lacker2020-closed} proves the convergence of $closed$--$loop$ Nash equilibria to a $weak$ MFG equilibrium. In particular, using some probability changes by Girsanov's Theorem. As pointed out by \citeauthor*{Cardaliaguet2020} \cite{Cardaliaguet2020}, to avoid a situation like in Folk Theorem, a non--degenerate volatility seems to be crucial for this convergence. Although answering the question of the convergence of Nash equilibria to MFG equilibria, \cite{lacker2020-closed} is not able to give a complete answer for the $converse$ convergence result. Indeed, \cite{lacker2020-closed} cannot show that any $weak$ MFG equilibrium is the limit of a sequence of approximate Nash equilibria. In the same spirit, \citeauthor*{LeflemLacker-2021} \cite{LeflemLacker-2021} provide, first, a result relating to the convergence of $closed$--$loop$ Nash equilibria in a common noise setting. Second, by considering an extension of the notion of $closed$--$loop$ Nash equilibrium, unlike \cite{lacker2020-closed}, they are able to show the $converse$ convergence result. However, they do not answer completely the shortcoming of \cite{lacker2020-closed} regarding the converse convergence result. Indeed, while being appropriate for the study, their extension of $closed$--$loop$ Nash equilibria is not what we can naturally expected (see their own comment \cite[Remark 2.11,Theorem 2.12]{LeflemLacker-2021}). In the situation of MFG of controls without common noise, the only article to our knowledge dealing with convergence issues for the $closed$--$loop$ controls is \citeauthor*{tangpipossamai2021} \cite{tangpipossamai2021}. They use the result of backward propagation of chaos of \cite{M_Lauriere-Tangpi} to deal with this issue among many others. Nevertheless, it seems that no many cases of non--unique MFG equilibria are considered because of their assumptions of existence and uniqueness for a McKean--Vlasov BSDE. In this paper, in a MFG of controls setting while allowing a common noise, we will give a complete characterization of the $closed$--$loop$ case. The convergence of Nash equilibria and $converse$ convergence is provided. In particular, we are able to prove that any $weak$ MFG is the limit of Nash equilibria without extending the notion of Nash equilibria as in \cite{LeflemLacker-2021}.

\medskip
{\color{black} Our secondary objective} is to provide some equivalence results. 
More precisely, equivalence results between $open$--$loop$ and $closed$--$loop$ controls are showed in both MFG of controls and extended mean field control (MFC). Indeed, first, we prove that the limit of $open$--$loop$ Nash equilibria and $closed$--$loop$ Nash equilibria are the same. Even if when the number of players is finite both equilibria can be different, in the infinite players setting, they are the same. Second, we show that the $open$--$loop$ formulation and the $closed$--$loop$ formulation of the MFC problem are the same. This result appears to be the first of this kind in the literature. Especially, we are able to prove the equivalence without using any convexity assumptions as it is usually done in the literature (see for instance \citeauthor*{Filippov} \cite{Filippov}, \citeauthor*{Roxin} \cite{Roxin}, \citeauthor*{el1987compactification} \cite{el1987compactification}, \citeauthor*{lacker2017limit} \cite{lacker2017limit}).

\medskip
\paragraph*{Main contributions}
In the first part, we begin by investigating the convergence problems in the setting of extended mean field game or mean field game of controls. \Cref{thm:limit_thm_closed-loop} is our first result in this framework and shows two main results. First, in an appropriate space (see \Cref{thm:limit_thm_closed-loop} for details), the sequence of empirical distribution $(\varphi^N[\alphab^N],\overline{\varphi}^N[\alphab^N])_{N \in \N^*}$ is relatively compact, and when $\varepsilon_N \to 0,$ any limit point is a measure--valued MFG solution (see \Cref{def:measure--valued--MFG}). Recall that the notion of measure--valued MFG solution is similar to the one used in \cite{MFD-2020_MFG} (see equivalence in \Cref{prop:eq--measure--valuedMFG_control} and \Cref{rem:eq--measure--valuedMFG_control}) to treat the convergence of $open$--$loop$ Nash equilibria. While being the first result dealing with convergence of $closed$--$loop$ Nash equilibria for MFG of controls with common noise, this first result of \Cref{thm:limit_thm_closed-loop} shows that the limits of $open$--$loop$ and $closed$--$loop$ Nash equilibria are exactly the same. This result is achieved by adapting the arguments of \cite{lacker2020-closed} in the setting of MFGC with common noise, and by using a delicate estimate of the regularity of the Fokker--Planck equation proved by \citeauthor{AronsonSerrin67} \cite[Theorem 4]{AronsonSerrin67}. It is worth mentioning that this result contains part of those of \cite{LeflemLacker-2021} in the classical MFG framework but allow in addition $\sigma$ to be non--constant. Second, similarly to the convergence of $closed$--$loop$ Nash equilibria, \Cref{thm:limit_thm_closed-loop} shows the convergence of approximate strong Markovian MFG equilibria to the measure--valued MFG equilibria.

\medskip
The second main result for the MFGC is \Cref{thm:converse-limit} dealing with the $converse$ convergence problem. We first show that any measure--valued MFG equilibrium is the limit of sequence of approximate strong Markovian MFG equilibria. This result combined with the convergence result of approximate strong Markovian MFG equilibria mentioned in \Cref{thm:limit_thm_closed-loop} proves that the notion of approximate strong Markovian MFG equilibrium is the correct infinite players formulation equivalent to the notion of measure--valued MFG equilibrium or more generally the correct infinite players formulation equivalent to any notion of weak MFG equilibrium. In other words, considering measure--valued MFG equilibrium (or weak MFG) is equivalent to considering approximate strong Markovian MFG equilibrium. The approximate strong Markovian MFG equilibrium has the advantage to be close to the usual notion considered in the literature. The only difference is for the approximate strong Markovian MFG equilibrium, the optimality is not an exact one but an $\varepsilon$--optimality (see above \Cref{def-epsilon-MFG}).  In a second time, \Cref{thm:converse-limit} shows that any measure--valued MFG equilibrium is the limit of approximate $closed$--$loop$ Nash equilibria. Even in the framework of classical MFG, this result seems to be the first of this kind. Indeed, the $converse$ convergence result of \cite{LeflemLacker-2021} needs to be done with a slight extension of the notion of $closed$--$loop$ Nash equilibria which is called in \cite{LeflemLacker-2021} $S$--$closed\;loop$ Nash equilibria. Our result shows that it is possible to avoid this extension. Mention that \Cref{thm:limit_thm_closed-loop} and \Cref{thm:converse-limit} are done without any use of convexity assumptions.

\medskip
In the second part, we treat the case of extended mean field control problem and obtain \Cref{thm:equivalence_mckv_cl} and \Cref{thm:limit_mckv_cl}.  In \Cref{thm:equivalence_mckv_cl}, we show that the $open$--$loop$ and the $closed$--$loop$ formulation of the extended mean field control problem are exactly the same. A notable point in this result is the fact that there is no use of convexity assumptions. This appears to be the first such result, even for the classical mean field control problem. Next, in \Cref{thm:limit_mckv_cl}, we show some limit theory results, namely: any sequence of approximate Pareto equilibria (see \Cref{def:Nplayers--equilibria}) is relatively compact and any limit point is the limit of sequence of approximate strong Markovian McKean--Vlasov (see \Cref{sec:open-closed-McKVl}).

\medskip
In the assumptions we use in this paper, the condition: $\sigma_0$ is invertible needs to be underlined. Although we believe that there must be a way of proving our results without this condition, in this article, this condition is important. Indeed, under this assumption, the filtration generated by the conditional distribution of state $(\mu_t)_{t \in [0,T]}$ is very $``$close$"$ to the filtration generated by the common noise $(B_t)_{t \in [0,T]}$. They are even equal when the control $\alpha$ has some regularities $\alpha$ (see for instance approximation in \Cref{prop:appr_general_closed} and \Cref{rm:equal-filtration}). This fact is actually quite classic. Indeed, we know that when a volatility $\sigma$ is non--degenerated, the unique strong solution $X$ of $\mathrm{d}X_t=\tilde b(t,X_t)\mathrm{d}t + \sigma(t,X_t)\mathrm{d}W_t,$ $X_0=0,$ has his natural filtration equal to the filtration generated by the Brownian motion $W.$ This kind of phenomenon appears in our framework. Making the presence of a common noise obligatory allows us to prove our results especially the $converse$ convergence result and the equivalence between $closed$--$loop$ and $open$--$loop$ formulations for the MFC problem. This effect of the common noise has been observed by many authors for various purposes (see for instance \citeauthor*{Delarue2019} \cite{Delarue2019}, \citeauthor*{Foguen2018} \cite{Foguen2018}, \cite{BAYRAKTAR202198}, \citeauthor*{delarueVas2021} \cite{delarueVas2021}).

\paragraph*{Outline of the paper}

The paper is structured as follows. After introducing some notations, \Cref{sec:set-resutls} introduces both the MFG and MFC frameworks, defines all the notions and concepts of equilibrium, and states the main results of the article. \Cref{sec:proofs} is devoted to the proofs. Namely, first, \Cref{section:relaxed_formulation} provides, over a canonical space, an equivalence of our notions appropriate for the proofs. Then, \Cref{sec:limit_NashEquilibria} characterizes the limit of approximate $closed$--$loop$ Nash equilibria. Next, \Cref{sec:limitset_converse} characterizes the limit of approximate strong Markovian MFG equilibria and provides the $converse$ convergence result. Finally, \Cref{sec:proof_Mc-Vl} gives the proofs relating to the McKean--Vlasov control problem.


\vspace{1.5mm}

{\bf \large Notations}.
	$(i)$
	Given a {\color{black}Polish} space $(E,\Delta)$ and $p \ge 1,$ we denote by $\Pc(E)$ the collection of all Borel probability measures on $E$,
	and by $\Pc_p(E)$ the subset of Borel probability measures $\mu$ 
	such that $\int_E \Delta(e, e_0)^p  \mu(de) < \infty$ for some $e_0 \in E$.
	We equip $\Pc_p(E)$ with the Wasserstein metric $\Wc_p$ defined by
	\[
		\Wc_p(\mu , \mu') 
		~:=~
		\bigg(
			\inf_{\lambda \in \Lambda(\mu, \mu')}  \int_{E \x E} \Delta(e, e')^p ~\lambda( \mathrm{d}e, \mathrm{d}e') 
		\bigg)^{1/p},
	\]
	where $\Lambda(\mu, \mu')$ denote the collection of all probability measures $\lambda$ on $E \x E$ 
	such that $\lambda( \mathrm{d}e, E) = \mu$ and $\lambda(E,  \mathrm{d}e') = \mu'( \mathrm{d}e')$. Equipped with $\Wc_p,$ $\Pc_p(E)$ is a Polish space (see \cite[Theorem 6.18]{villani2008optimal}). For any  $\mu \in \Pc(E)$ and $\mu$--integrable function $\varphi: E \to \R,$ we define
	\begin{align*}
	    \langle \varphi, \mu \rangle
	    =
	    \langle \mu, \varphi \rangle
	    :=
	    \int_E \varphi(e) \mu(\mathrm{d}e),
	\end{align*}
	and for another metric space $(E^\prime,\Delta^\prime)$, we denote by $\mu \otimes \mu^\prime \in \Pc(E \x E')$ the product probability of any $(\mu,\mu^\prime) \in \Pc(E) \x \Pc(E^\prime)$.
	
	Given a probability space $(\Om, \Fc, \P)$ supporting a sub-$\sigma$-algebra $\Gc \subset \Fc$ then for a Polish space $E$ and any random variable $\xi: \Om \longrightarrow E$, both the notations $\Lc^{\P}( \xi | \Gc)(\om)$ and $\P^{\Gc}_{\om} \circ (\xi)^{-1}$ are used to denote the conditional distribution of $\xi$ knowing $\Gc$ under $\P$.
	
	
	
\medskip
	\noindent $(ii)$	
	For any $(E,\Delta)$ and $(E',\Delta')$ two Polish spaces, we use $C_b(E,E')$ to denote the set of continuous functions $f$ from $E$ into $E'$ such that $\sup_{e \in E} \Delta'(f(e),e'_0) < \infty$ for some $e'_0 \in E'$.
	Let $\N^*$ denote the set of positive integers. Given non--negative integers $m$ and $n$, we denote by $\S^{m \x n}$ the collection of all $m \x n$--dimensional matrices with real entries, equipped with the standard Euclidean norm, which we denote by $|\cdot|$ regardless of the dimensions. 
	We also denote $\S^n:=\S^{n \times n}$, and denote by $0_{m \times n}$ the element in $\S^{m \times n}$ whose entries are all $0$, and by $\mathrm{I}_n$ the identity matrix in $\S^n$. 
	Let $k$ be a positive integer, we denote by $C^k_b(\R^n;\R)$ the set of bounded maps $f: \R^n \longrightarrow \R$, having bounded continuous derivatives of order up to and including $k$. 
	Let $f: \R^n \longrightarrow \R$ be twice differentiable, we denote by $\nabla f$ and $\nabla^2f$ 
	the gradient and Hessian of $f$.


\medskip
    \noindent $(iii)$
	Let $T > 0$ and $(\Sigma,\rho)$ be a Polish space, we denote by $C([0,T]; \Sigma)$ the space of all continuous functions on $[0,T]$ taking values in $\Sigma$.
	Then, $C([0,T]; \Sigma)$ is a Polish space under the uniform convergence topology, and we denote by $\|\cdot\|$ the uniform norm. 
	When $\Sigma=\R^k$ for some $k\in\N$, we simply write $\Cc^k := C([0,T]; \R^k),$ also we shall denote by $\Cc^{k}_{\Wc}:=C([0,T]; \Pc(\R^k)),$ and for $p \ge 1,$ $\Cc^{k,p}_{\Wc}:=C([0,T]; \Pc_p(\R^k)).$ 

\medskip
	With a Polish space $E$, we denote by $\M(E)$ the space of all Borel measures $q( \mathrm{d}t,  \mathrm{d}e)$ on $[0,T] \x E$, 
	whose marginal distribution on $[0,T]$ is the Lebesgue measure $ \mathrm{d}t$, 
	that is to say $q( \mathrm{d}t, \mathrm{d}e)=q(t,  \mathrm{d}e) \mathrm{d}t$ for a family $(q(t,  \mathrm{d}e))_{t \in [0,T]}$ of Borel probability measures on $E$.
	For any $q \in \M(E)$ and $t \in [0,T],$ we define $q_{t \wedge \cdot} \in \M(E)$ by
	\begin{equation}\label{eq:lambda}
	q_{t \wedge \cdot}(\mathrm{d}s, \mathrm{d}e) :=  q(\mathrm{d}s, \mathrm{d}e) \big|_{ [0,t] \x E} + \delta_{e_0}(\mathrm{d}e) \mathrm{d}s \big|_{(t,T] \x E},\; \text{for some fixed $e_0 \in E$.}
	\end{equation}
	We will say that a Borel measurable function $h:[0,T] \x \R^n \x C([0,T];\Sigma) \x \M(E)$ is progressively Borel measurable if it verifies $h(t,x,\mathsf{a},\mathsf{z})=h(t,x,\mathsf{a}_{t \wedge \cdot},\mathsf{z}_{t \wedge \cdot}),$ for any $(t,x,\mathsf{a},\mathsf{z}) \in [0,T]\x \R^n \x C([0,T];\Sigma) \x \M(E).$


\section{Setup and Main results} \label{sec:set-resutls}

In this section, we first introduce a $N$--player game, and the definition of $\varepsilon_N$--Nash and $\varepsilon_N$--Pareto equilibria. Next, we formulate the notions of approximate strong Markovian and measure--valued MFG equilibria which will be essential to describe the limit of the Nash equilibria. Finally, we give the $open$--$loop$ and $closed$--$loop$ formulations of the McKean--Vlasov optimal control or MFC and use them to deal with the limit of Pareto equilibria. 

\medskip
The general assumptions used throughout this paper are now formulated. The dimension $n \in \N^*,$ the nonempty Polish space $(U,d)$, the horizon time $T>0$ are fixed and $\Pc^n_U$ denotes the space of all Borel probability measures on $\R^n \x U$ i.e. $\Pc^n_U:=\Pc(\R^n \x U).$ {\color{black} Also, we set $p \ge 2,$ $\nu \in \Pc_{p'}(\R^n)$ with $p'> p,$ and the probability space $(\Om,\H:=(\Hc_t)_{t \in [0,T]},\Hc,\P)$\footnote{The probability space $(\Om,\H,\P)$ contains as many random variables as we want in the sense that: each time we need a sequence of independent uniform random variables or Brownian motions, we can find them on $\Om$ without mentioning an enlarging of the space. }. } 
We give ourselves the following bounded Borel measurable functions
	\[
		\big[b,L \big]:[0,T] \x \R^n \x \Pc_p(\R^n \x U) \x U \longrightarrow \R^n \x \R,\;\sigma:[0,T] \x \R^n \longrightarrow \S^{n \x n} \;\mbox{and}\; 
		g: \R^n \x \Pc_p(\R^n) \longrightarrow \R.
	\]
		

    \begin{assumption} \label{assum:main1} 
		
		
		$(i)$ $\sigma_0 \in \S^{n \x n}$ is an invertible constant matrix and $U \subset \R^q$, for $q \in \N^*,$ is a compact convex nonempty set;
	
	
	\medskip	
		$(ii)$ The maps $b$ and $\sigma$ are Lipschitz in all their variables.
		Also, the maps $L$ and $g$ are s.t. for each $t,$ $\R^n \x \Pc(\R^n \x U) \x U \x \Pc_p(\R^n) \ni (x,\nub,u,\nu) \to \big(L(t,x,\nub,u), g(x,\nu) \big) \in \R^n \x \R \x \R$ is continuous;

    \medskip
		$(iii)$ ${\rm\underline{Non\mbox{--}degeneracy\;condition}}$: for some constant $\theta >0$, one has, for all $(t,x) \in [0,T] \x \R^n$,
		\begin{align*}
		    ~\theta \mathrm{I}_{n} \le ~\sigma \sigma^{\top}(t,x);
		\end{align*}
		
    \medskip
		$(iv)$ ${\rm \underline{Separability\;condition}}$: There exist Borel functions $(b^\circ, b^\star,L^\circ,L^\star)$ satisfying
		\begin{align*}
		    b ( t,x,\nub, u )
		    :=
		    b^\star (t,\nub ) 
		    +
		    b^\circ ( t,x, u )\;\;\mbox{and}\;\;L ( t,x,\nub, u )
		    :=
		    L^\star(t,x,\nub ) 
		   +
		   L^\circ( t,x, \nu, u ),
		\end{align*}
		for all $(t,x,\nub, u) \in [0,T] \x \R^n \x \Pc^n_U \x U$ where $\nu(\mathrm{d}x):=\nub(\mathrm{d}x,U).$
		
	\end{assumption}

\begin{remark}
    The previous assumptions are standard in the probabilistic approach of mean field game and control problems. The ${\rm \underline{separability\;condition}}$ is more specific to the extended mean field game and control problems $($see {\rm \citeauthor*{carmona2015probabilistic} \cite{carmona2015probabilistic}, \citeauthor*{Lehalle-card} \cite{Lehalle-card}, \citeauthor*{M_Lauriere-Tangpi} \cite{M_Lauriere-Tangpi}, \citeauthor*{tangpipossamai2021} \cite{M_Lauriere-Tangpi}, \citeauthor*{MFD-2020_MFG} \cite{MFD-2020_MFG}}$)$. It is mainly used for technical reasons. Notice that some conditions can be weakened. But, in order to avoid certain unnecessary technicalities in the proofs, we have chosen these.
\end{remark}

\subsection{The $N$--player games} \label{sec:finite-player-game} Let $N \in \N^*.$
We denote by $\Ac^c_N$ the collection of all progressively Borel measurable functions $\alpha:[0,T] \x (\Cc^n)^N \to U.$
On the filtered probability space $(\Om,\H,\Hc,\P),$ let $(W^i)_{i \in \N^*}$ be a sequence of independent $\H$--adapted $\R^n$--valued Brownian motions, $B$ be an $\R^n$--valued $\H$--adapted Brownian motion and $(\xi^i)_{i \in \N^*}$ a sequence of iid $\Hc_0$--random variables of law $\nu.$ Besides, $(W^i)_{i \in \N^*},$ $B$ and $(\xi^i)_{i \in \N^*}$ are independent. Then, given the control rule/strategy $\alphab:=(\alpha^1,\dots,\alpha^N) \in (\Ac^c_N)^N$, denote by $\Xbb^{\alphab}:=(X^{\alphab,1}_{\cdot},\dots,X^{\alphab,N}_{\cdot})$ the processes satisfying: for each $i \in \{1,\dots,N\},$ $\E\big[\|X^{\alphab,i}\|^p \big]< \infty,$
    \begin{align} \label{eq:N-agents_StrongMV_CommonNoise-law-of-controls}
        \mathrm{d}X^{\alphab,i}_t
        =
        b\big(t,X^{\alphab,i}_{t},\overline{\varphi}^{N}_{t}[\alphab] ,\alpha^i(t,\Xbb^{\alphab}) \big) \mathrm{d}t 
        +
        \sigma \big(t,X^{\alphab,i}_{t} \big) \mathrm{d}W^i_t
        +
        \sigma_0 \mathrm{d}B_t,\;X^{\alphab,i}_0=\xi^i
    \end{align}
	with 
	\[
	    \varphi^{N}_{t}[\alphab](\mathrm{d} x) := \frac{1}{N}\sum_{i=1}^N \delta_{X^{\alphab,i}_{t}}(\mathrm{d} x)\;\;\mbox{and}\;\;
	    \overline{\varphi}^{N}_{t}[\alphab](\mathrm{d} x, \mathrm{d} u) := \frac{1}{N}\sum_{i=1}^N \delta_{\big(X^{\alphab,i}_{t},\;\alpha^i(t,\Xbb^{\alphab}) \big)}(\mathrm{d} x, \mathrm{d} u)
	    ,~
	    \mbox{for all}
	    ~~
	    t \in [0,T].
	\] 
The reward value of player $i$ associated with control rule/strategy $\alphab:=(\alpha^1,\dots,\alpha^N)$ is then defined by
	\[
	    J_i(\alphab)
	    :=
	    \E \bigg[
        \int_0^T L\big(t,X^{\alphab,i}_t,\overline{\varphi}^{N}_{t}[\alphab] ,\alpha^i(t,\Xbb^{\alphab}) \big) \mathrm{d}t 
        + 
        g \big( X^{\alphab,i}_T, \varphi^{N}_T[\alphab] \big)
        \bigg].
	\]

\medskip
    Now, we give the precise definition of what we call approximate $closed$--$loop$ Nash and Pareto equilibria.
	
\begin{definition}{\rm((approximate) equilibria)} \label{def:Nplayers--equilibria}

\medskip    
    Let $N \in \N^*$ and $\varepsilon \in \R_{+}.$ We will say that a control rule/strategy $(\alpha^1,\dots,\alpha^N) \in (\Ac^c_N)^N$ is an
	\begin{itemize}
	    \item {\rm$\mathbf{Nash\; equilibrium}$:}  $\varepsilon$--$closed$--$loop$ {\rm Nash equilibrium} if 
	    \[
	        J_i(\alpha^1,\dots,\alpha^N) \ge \sup_{\beta \in \Ac^c_N} J_i\big( \alpha^1,\dots,\alpha^{i-1},\beta,\alpha^{i+1},\dots,\alpha^N \big)-\varepsilon,\;\mbox{for each}\;i \in \{1,\dots,N\}.
	    \]
	    
	    \item {\rm$\mathbf{Pareto\; equilibrium}$:} $\varepsilon$--closed--loop {\rm Pareto equilibrium} if 
	    \[
	        \sum_{i=1}^N J_i(\alpha^1,\cdots,\alpha^N) \ge \sup_{(\beta^1,\dots,\beta^N) \in (\Ac^c_N)^N} \sum_{i=1}^N J_i( \beta^1,\dots,\beta^N )-\varepsilon.
	    \]
	\end{itemize}
	    
\end{definition}

\begin{remark}
    
    $(i)$ Because of the fact that the controls $(\alpha^1,\cdots,\alpha^N)$ are only Borel measurable, the strong existence and uniqueness of \eqref{eq:N-agents_StrongMV_CommonNoise-law-of-controls} are not standard. We refer to {\rm\citeauthor*{Veretennikov_1981} \cite{Veretennikov_1981}} for the well--posedness of this type of SDE when $\alpha^i(t,\mathsf{x}^1,\cdots,\mathsf{x}^N)=\alpha^i(t,\mathsf{x}^1(t),\cdots,\mathsf{x}^N(t)).$ For the general situation, we consider existence and uniqueness in law of \eqref{eq:N-agents_StrongMV_CommonNoise-law-of-controls} by a classical application of the  Girsanov’s theorem. Let us mention that, as weak solution, $\Xbb$ may not be defined only on $(\Om,\H,\P).$ But for ease of reading and to avoid heavy notations, we assumed $(\Om,\H,\P)$ in order to be able to define all our variables on this space.
    
    \medskip
    $(ii)$ While the existence of approximate Pareto equilibria is obvious, that of approximate Nash equilibria is unclear. As a consequence of our results, we will {\color{black}see} that when $\varepsilon >0,$ the approximate Nash equilibria are well--defined under our assumptions.
    
    \medskip
    $(iii)$ Given our framework, the natural shape for our controls should be $Markovian$ that is to say $\alpha^i(t,\mathsf{x}^1,\cdots,\mathsf{x}^N)=\alpha^i(t,\mathsf{x}^1(t),\cdots,\mathsf{x}^N(t))$ and not fully path dependent as we consider. In the MFG setting, although the establishment of the convergence of approximate Markovian Nash equilibria to MFG can be done with our techniques, we are unfortunately only able to establish the $converse$ convergence result for fully path dependent controls $($see {\rm\Cref{thm:converse-limit}}$)$. Therefore, we make the choice to present our article only with fully path dependent controls. This technical limitation seems to be a strange phenomenon highlighted in other articles $($see {\rm \cite{LeflemLacker-2021,lacker2020-closed}}$)$, and only appears for the mean field game setting. 
\end{remark}


\subsection{The mean field game of controls} \label{sec:MFGC}
{\color{black}We first formulate here the formulation} of the extended mean field game (or mean field game of controls). We will call this formulation (approximate) strong Markovian MFG equilibrium. Second, inspired by \cite{MFD-2020_MFG}, we give the notion of measure--valued MFG which generalizes the strong Markovian MFG equilibrium.
\subsubsection{$\varepsilon$--strong Markovian MFG equilibrium}
\label{section:strong_formulation}

Let $(B,W)$ be an $\R^{n} \x \R^n$--valued $\H$--Brownian motion, and $\xi$ be a $\Fc_0$--random variable $\xi$ such that $\Lc(\xi)=\nu.$ We denote by $\F^B:=(\Fc^B_t)_{t \in [0,T]}$ the filtration generated by $B$ i.e. $\Fc^B_t:=\sigma\{B_s: s \le t \}.$

\begin{definition}
        For each $\varepsilon \ge 0,$ we say that a $\F^B$ predictable process $(\mub_t)_{t \in [0,T]}$ is an $\varepsilon$--strong Markovian MFG solution if: 
    \begin{enumerate}
        
        \item[$(i)$] $\mub_t=\Lc(X_t, \alpha(t,X_t,\mu)|\Fc^B_t)$ $\mathrm{d}t \otimes \mathrm{d}\P$--a.e. where $\alpha$ is a progressively Borel measurable function $\alpha: [0,T] \x \R^n \x C([0,T];\Pc(\R^n)) \to U,$ $X$ the unique strong solution of
    \begin{align} \label{eq:MKV_strong-MFG}
		\mathrm{d}X_t
		= 
		b \big(t, X_t, \mub_t, \alpha(t,X_t,\mu) \big) \mathrm{d}t
		+
		\sigma\big(t, X_t \big) \mathrm{d} W_t
		+
		\sigma_0 \mathrm{d}B_t\;\mbox{with}\;X_0=\xi\;\mbox{and}\;\mu_t:=\Lc(X_t|\Fc^B_t);
	\end{align}
	
        \item[$(ii)$] For any progressively Borel  measurable function $\alpha': [0,T] \x \R^n \x C([0,T];\Pc(\R^n)) \to U,$ and $X'$ the solution of 
    \begin{align} \label{eq:fixed-Markovian-SDE-MFG}
		\mathrm{d}X'_t
		= 
		b \big(t, X'_t, \mub_t, \alpha'(t,X'_t,\mu) \big) \mathrm{d}t
		+
		\sigma\big(t, X'_t \big) \mathrm{d} W_t
		+
		\sigma_0 \mathrm{d}B_t,\;\;X'_0=\xi,
	\end{align}    
	one has 
	\begin{align*} 
		    \E \bigg[
				\int_0^T L(t, X_t, \mub_t,\alpha(t,X_t,\mu)) \mathrm{d}t 
				+ 
				g(X_T, \mu_T) 
			\bigg]
			\ge
			\E \bigg[
				\int_0^T L(t, X'_t, \mub_t,\alpha'(t,X'_t,\mu)) \mathrm{d}t 
				+ 
				g(X'_T, \mu_T) 
			\bigg] - \varepsilon.
		\end{align*} 
    \end{enumerate}
\end{definition}   

\begin{remark}
    $(i)$ A reader familiar with MFG of controls would expect another shape for an $($approximate$)$ optimal control $\alpha$ $($see {\rm \Cref{def-epsilon-MFG}} in the Introduction for instance$)$. Indeed, in our setting, the expected shape of an $($approximate$)$ optimal control should be $\alpha(t,X_t,(\mub_s)_{s \in [0,t]})$ and not $\alpha(t,X_t,(\mu_s)_{s \in [0,t]})$ i.e. controls depending on the time, the state and the conditional distribution of control and state, and not controls depending on the time, the state and the conditional distribution of state. However, under our assumptions, in particular the condition $\sigma_0$ is invertible, these two representations are close. Let us briefly explain this fact. First, as $(\mu_s)_{s \in [0,t]}$ is a function of $(\mub_s)_{s \in [0,t]},$ we can always see a control of the form $\alpha(t,X_t,(\mu_s)_{s \in [0,t]})$ as a control of the form $\alpha(t,X_t,(\mub_s)_{s \in [0,t]}).$
    Second, using the fact that $\mub$ is $\F^B$--adapted, so a progressively Borel function of $B,$ as we will see in the proof $($ see for instance {\rm \Cref{prop:open-closed-approx}} or {\rm \Cref{prop:appr_general_closed}} $)$ , it is always possible to approximate a control of type $\alpha(t,X_t,(\mub_s)_{s \in [0,t]})$ by a sequence of controls of type $\alpha(t,X_t,(\mu_s)_{s \in [0,t]}).$   
    
    \medskip
    $(ii)$ There are two other facts that make us choose this shape of control. Firstly, this allows us to have a formulation that fully falls within the framework of MFG without the $($conditional$)$ law of control. Indeed, for the classical MFG, the controls are of type  $\alpha(t,X_t,(\mu_s)_{s \in [0,t]})$ $($see {\rm \cite{cardaliaguet2015master,carmona2018probabilisticII,lacker2020-closed}}$)$. Secondly, for the construction of approximate Nash equilibrium from MFG equilibrium, in MFG of control setting, this formulation gives a more natural way. Indeed, when $\alpha$ is regular enough $($this will be our case$)$, for each $N \in \N^*,$ an approximate Nash equilibrium is constructed by defining: $\alpha^{i}(t,\mathsf{x}^1,\dots,\mathsf{x}^N)
        :=
        \alpha\big(t,\mathsf{x}^i(t),\pi^N\big)\;\mbox{with}\;\pi^N(t):=\frac{1}{N} \sum_{i=1}^N \delta_{\mathsf{x}^i(t)}.$ But if $\alpha$ is of type $\alpha(t,X_t,(\mub_s)_{s \in [0,t]}),$ a natural construction is $\alpha^{i}(t,\mathsf{x}^1,\dots,\mathsf{x}^N)
        :=
        \alpha\big(t,\mathsf{x}^i(t),\overline{\pi}^N\big)\;\mbox{with}\;\overline{\pi}^N(t):=\frac{1}{N} \sum_{i=1}^N \delta_{\mathsf{x}^i(t),\alpha^{i}(t,\mathsf{x}^1,\dots,\mathsf{x}^N)}.$ Due to the appearance of $\alpha^i$ in both side of the equality, it is unclear that this construction is well defined.
        
        \medskip
        $(iii)$ The well--posedness of this notion of equilibrium is not clear at first sight. But, as we will see later $($see {\rm \Cref{prop:existence}} with {\rm \Cref{thm:converse-limit}}$)$, for any $\varepsilon>0,$ there exists an $\varepsilon$--strong Markovian MFG equilibrium.
\end{remark}

\subsubsection{Measure--valued MFG equilibrium}

\begin{definition} [measure--valued equilibrium] \label{def:measure--valued--MFG}
		We say that a term
		\[
			\big( \Om, \Fc,  \Pr,  \F := (\Fc_t)_{0 \le t \le T}, W,B, X, \Lambda, \mu \big)
		\]
		is a measure--valued MFG equilibirum if
		\begin{enumerate}
			\item[$(i)$] $(\Om, \F, \Fc,\Pr)$ is a filtered probability space. $(W,B)$ is an $\R^n \x \R^n$--valued $\F$--Brownian motion, $(X,\mu)$ is an $\R^n \x \Pc(\R^n)$--valued $\F$--adapted continuous process and $(\Lambda_t)_{t \in [0,T]}$  is a $\F$--predictable $\Pc(\Pc^n_U)$--valued process.
			
			\item[$(ii)$] $X_0,$ $W$ and $(B,\Lambda)$ are $\Pr$--independent.
			
			\item[$(iii)$] For each $t \in [0,T],$ $\mu_t=\Lc^{\Pr} \big(X_t \big| \Gc_t \big)$ $\Pr$--a.e. where $\Gc_t:=\sigma\{B_{t \wedge \cdot},\mu_{t \wedge \cdot},\Lambda_{t \wedge \cdot} \},$ $X$ satisfies
			\begin{align*}
			    \mathrm{d}X_t
		        = 
		        \int_{\Pc^n_U}\; \int_{U} b \big(t, X_t, \nub, u \big) \nub^{X_t} (\mathrm{d}u) \;\Lambda_t (\mathrm{d}\nub) \mathrm{d}t
		        +
		        \sigma\big(t, X_t \big) \mathrm{d} W_t
		        +
		        \sigma_0 \mathrm{d}B_t,\;\;X_0=\xi
			\end{align*}
			and 
			\begin{align*}
			    \Lambda_t \big( \{ \nub \in \Pc^n_U:\;\nub(\mathrm{d}x,U)= \mu_t(\mathrm{d}x) \} \big)=1,\;\;\mathrm{d}t \otimes \mathrm{d}\Pr\mbox{--a.e.}
			\end{align*}
			where for each $\nub \in \Pc^n_U,$ $\R^n \ni x \to \nub^x \in \Pc(U)$ is Borel measurable and satisfies $\nub(\mathrm{d}x,\mathrm{d}u)=\nub^x(\mathrm{d}u)\nub(\mathrm{d}x,U).$ 
		
			
			\item[$(iv)$] For any Borel progressive measurable function $\alpha': [0,T] \x \R^n \x \M(\Pc^n_U) \to U,$ and for each $t \in [0,T],$ $\mu'_t=\Lc^{\Pr}(X'_t \big| \Gc_t)$ $\Pr$--a.e. where $X'$ is the solution of 
        \begin{align}         
		    \mathrm{d}X'_t
		    = 
		    \int_{\Pc^n_U}\; b \big(t, X'_t, \nub, \alpha'(t,X'_t,\Lambda) \big) \;\Lambda_t (\mathrm{d}\nub) \mathrm{d}t
		    +
		    \sigma\big(t, X'_t \big) \mathrm{d} W_t
		    +
		    \sigma_0 \mathrm{d}B_t,\;\;X'_0=\xi
	   \end{align}    
	        one has 
	   \begin{align*} 
		    &\E^{\Pr} \bigg[
				\int_0^T \int_{\Pc^n_U} \langle L(t, \cdot, \nub,\cdot), \nub \rangle \Lambda_t (\mathrm{d}\nub) \mathrm{d}t 
				+ 
				\langle g(\cdot, \mu_T), \mu_T \rangle 
			\bigg]
			\\
			&\ge
			\E^{\Pr} \bigg[
				\int_0^T \int_{\Pc^n_U} \langle L(t, \cdot, \nub,\alpha'(t,\cdot,\Lambda)), \mu'_t \rangle \Lambda_t (\mathrm{d}\nub) \mathrm{d}t 
				+ 
				\langle g(\cdot, \mu_T), \mu'_T \rangle
			\bigg].
		\end{align*} 
	\end{enumerate}	
	\end{definition}
	
	\begin{remark}
	    $(i)$ Although formulated differently, the definition of the measure--valued MFG equilibrium given in this paper is quite equivalent to those introduced in {\rm \cite[Definition 2.7]{MFD-2020_MFG}} for $open$--$loop$ setting $($see {\rm \Cref{prop:eq--measure--valuedMFG_control}} and {\rm \Cref{rem:eq--measure--valuedMFG_control}} for details$)$. Moreover, in the classical MFG framework i.e. when there is no law of control, this notion of measure--valued MFG equilibrium turns out to be equivalent to the classical notion of weak MFG equilibrium $($see comparison in {\rm \cite[Comparison Definition 2.7 ...]{MFD-2020_MFG}} after {\rm Remark 2.8}$)$.
	    
	    \medskip
	    $(ii)$ We would like to emphasize the fact that our controls do not need to depend on Brownian motion $B$.
	    The control processes considering only the fixed measures $(\mu,\Lambda)$ and not $B$ are more natural. Intuitively, a player has to consider his position and the distribution of {\color{black} the whole population.} For more mathematical justification, refer to {\rm \cite{cardaliaguet2015master,carmona2018probabilisticII}}.
	\end{remark}

\medskip
With the equivalence between our notion of measure--valued MFG equilibrium and the notion considered in {\rm \cite[Definition 2.7]{MFD-2020_MFG}} for $open$--$loop$ framework that we will prove in \Cref{prop:eq--measure--valuedMFG_control}, thanks to the result proved in \cite[Theorem 7.2.4]{djete2020some}, we have the following result.

\begin{proposition}\label{prop:existence}
    Under \Cref{assum:main1}, there exists at least one measure--valued MFG equilibrium.
\end{proposition}

\subsubsection{Limit theorem and converse limit theorem}  	

We are now ready to formulate the main convergence results of this paper regarding the extended mean field game. We begin with the convergence of sequence of approximate Nash equilbria and approximate strong Markovian MFG equilibria towards the measure--valued MFG equilibria. The proof is given in \Cref{proof_thm_1_MFG} and \Cref{proof_thm_2_MFG}.

\begin{theorem}[limit theorem] \label{thm:limit_thm_closed-loop}
    \medskip Let {\rm \Cref{assum:main1}} hold true. 
    
\begin{itemize}
    \item For each $N \in \N^*$, let $\alphab^N:=(\alpha^{1,N}, \dots, \alpha^{N,N})$ be an $\varepsilon_N$--closed--loop Nash equilibrium, then the sequence $(\mathrm{P}^N)_{N \in \N^*}$ with $\mathrm{P}^N:=\mathrm{P}^N[\alphab^N] \in \Pc \big( \Cc^n_{\Wc} \x \M(\Pc^n_U) \big)$ is relatively compact in $\Wc_p$ where
    \[
        \mathrm{P}^N[\alphab^N]
        :=
        \Lc^{\P}\Big(\varphi^{N}[\alphab^N],\;\delta_{\overline{\varphi}^{N}_t[\alphab^N]}(\mathrm{d}\nub)\mathrm{d}t \Big)
    \]
    and if $\Lim_{N \to \infty} \varepsilon_N=0,$ then for each limit point $\mathrm{P}^\infty$ there exists a measure--valued solution  $\big( \Om, \Fc,  \Pr,  \F, W, B, X, \Lambda, \mu \big)$ such that
    \begin{align*}
        \mathrm{P}^\infty
        =
        \Lc^{\Pr}(\mu, \Lambda).
    \end{align*}
    
    \item For each $\ell \in \N^*,$ let $\mub^\ell$ be an $\varepsilon_\ell$--strong Markovian MFG equilibirum, then $(\Pr^\ell)_{\ell \in \N^*} \subset  \Pc \big( \Cc^n_{\Wc} \x \M(\Pc^n_U) \big)$ is relatively compact in $\Wc_p$ where $\Pr^\ell:=\P \circ (\mu^\ell,\delta_{\mub^\ell_s}(\mathrm{d}u)\mathrm{d}s)^{-1}$ and if $\Lim_{\ell \to \infty} \varepsilon_\ell=0,$ then for each limit point $\mathrm{P}^\infty$ there exists a measure--valued solution  $\big( \Om, \Fc,  \Pr,  \F, W, B, X, \Lambda, \mu \big)$ such that
    \begin{align*}
        \mathrm{P}^\infty
        =
        \Lc^{\Pr}(\mu, \Lambda).
    \end{align*}
\end{itemize}

\end{theorem}

\begin{remark}
    In the framework of extended mean field game, these convergence results seem to be the first of this type. Unlike the recent paper of {\rm \cite{tangpipossamai2021}} where the setting is without common noise, $\sigma$ constant and strong assumptions, in our paper, we treat the case with common noise and the assumptions are less strong.
\end{remark}

Now we present the converse convergence result. That is, any measure--valued MFG equilibrium is the limit of sequence of approximate $closed$--$loop$ Nash equilibria or sequence of approximate strong Markovian MFG equilibria. The proof is given in \Cref{proof_thmconv_1_MFG} and \Cref{proof_thmconv_2_MFG}.


    
    \begin{theorem}[Converse limit theorem] \label{thm:converse-limit}
        Let {\rm \Cref{assum:main1}} hold true. For any measure--valued MFG equilibrium $\big( \Om, \Fc,  \Pr,  \F, W, B, X, \Lambda, \mu \big)$ there exists:
    \begin{itemize}
        
        \item $(\mub^\ell)_{\ell \in \N^*}$ s.t. for each $\ell \in \N^*,$ $\mub^\ell$ is an $\varepsilon_\ell$--strong Markovian MFG equilibrium with $\varepsilon_\ell>0,$ $\Lim_{\ell \to \infty} \varepsilon_\ell=0,$ and
        \begin{align*}
            \Lc^{\Pr}(\mu, \Lambda)
            =
            \Lim_{\ell \to \infty}\Lc^{\P}\big(\mu^\ell,\;\delta_{\mub^\ell}(\mathrm{d}m)\mathrm{d}t \big)\;\mbox{in}\;\Wc_p.
        \end{align*}

        \item $(\alpha^{1,N},\cdots,\alpha^{N,N})_{N \in \N^*}$ s.t. for each $N \in \N^*$, $\alphab^N:=(\alpha^{1,N}, \dots, \alpha^{N,N})$ is an $\varepsilon_N$--$closed$--$loop$ Nash equilibrium  with $\varepsilon_N >0,$ $\Lim_{N \to \infty}\varepsilon_N=0$ and 
        \begin{align*}
            \Lc^{\Pr}(\mu, \Lambda)
            =
            \Lim_{N \to \infty}\Lc^{\P}\Big(\varphi^{N}[\alphab^N],\;\delta_{\overline{\varphi}^{N}_t[\alphab^N]}(\mathrm{d}m)\mathrm{d}t \Big)\;\mbox{in}\;\Wc_p.
        \end{align*}
    \end{itemize}

    \end{theorem}
    
    \begin{remark}
        Under these general assumptions, especially allowing the MFG equilibrium to be potentially non unique, {\rm \Cref{thm:converse-limit}} seems to be the first result about the approximation of any MFG equilibrium via a sequence of approximate $closed$--$loop$ Nash equilibria or approximate strong Markovian MFG equilibria.
        Without counting the fact that we take into account the empirical distribution of states and controls, even for the classical MFG framework, this result appears as new in the literature. As highlighted in {\rm \cite{lacker2020-closed}} and later in {\rm \cite{LeflemLacker-2021}}, in the presence of multiple MFG equilibria, this kind of converse result is delicate to prove without extension of the notion of $closed$--$loop$ Nash equilibrium $($see {\rm \cite{LeflemLacker-2021}}$)$. To bypass this difficulty, the condition $\sigma_0$ is invertible turns out to be quite useful. 
    \end{remark}

\subsection{The McKean--Vlasov control problem}


\subsubsection{The $open$--$loop$ and $closed$--$loop$ formulations} \label{sec:open-closed-McKVl}

    Under \Cref{assum:main1}, we will formulate here the $open$--$loop$ and $closed$--$loop$ formulations of the McKean--Vlasov control problem on the probability space $(\Om, \H, \Hc, \P)$ with the $\R^n \x \R^n$--valued $\H$--adapted Brownian motion $(W,B)$ and the $\Fc_0$--random variable $\xi$ s.t. $\Lc(\xi)=\nu.$  We denote by $\F:=(\Fc_t)_{t \in [0,T]}$ and $\G:=(\Gc_t)_{t \in [0,T]}$ the $\P$--completion filtration of $\big(\sigma\{\xi,B_r,W_r:\;r \le t\} \big)_{t \in [0,T]}$ and $\big(\sigma\{B_r:\;r \le t\} \big)_{t \in [0,T]}$ respectively.
    
	\paragraph*{$Open$--$loop$ formulation}  Let $\Ac^o$ be the collection of all $U$--valued processes $\alpha = (\alpha_s)_{0 \le s \le T}$ which are $\F$--predictable.
	Then, given $\alpha \in \Ac^o$, let $X^{o,\alpha}$ be the unique strong solution of the SDE: $\E^{\P} \big[\|X^{o,\alpha}\|^p \big]< \infty,$
	$X^{o,\alpha}_0= \xi$,
	\begin{align} \label{eq:MKV_strong-law-of-controls}
		\mathrm{d}X^{o,\alpha}_t
		= 
		b \big(t, X^{o,\alpha}_t, \mub^{o,\alpha}_t, \alpha_t \big) \mathrm{d}t
		+
		\sigma\big(t, X^{o,\alpha}_t \big) \mathrm{d} W_t
		+
		\sigma_0 \mathrm{d}B_t\;\mbox{with}\;\mub^{o,\alpha}_t:=\Lc(X^{o,\alpha}_t, \alpha_t | \Gc_t )\;\mbox{and}\;\mu^{o,\alpha}_t:=\Lc(X^{o,\alpha}_t | \Gc_t ).
	\end{align}
	We will say that $X^{o,\alpha}$ is an $open$--$loop$ McKean--Vlasov process associated to $\alpha.$
	Let us now introduce the following $open$--$loop$ McKean--Vlasov control problem by
	    \begin{align} \label{eq:strong_Value}
	    	V^{o}_S
			~:=~
			\sup_{\alpha \in \Ac^o} 
			\Phi^{o}(\alpha)\;\mbox{where}	~~\Phi^{o}(\alpha):=\E \bigg[
				\int_0^T L(t, X^{o,\alpha}_t,\mub^{o,\alpha}_t, \alpha_t) \mathrm{d}t 
				+ 
				g\big(X^{o,\alpha}_T, \mu^{o,\alpha}_T \big) 
			\bigg].
		\end{align}
		
	\paragraph*{$Closed\;loop$ formulation}	
    Let $\Ac^c$ be the collection of all Borel progressively measurable maps $[0,T] \x \R^n \x C([0,T];\Pc(\R^n)) \ni (t,x,\pi) \to \alpha(t,x,\pi) \in U.$ 
	Then, given $\alpha \in \Ac^c$, let $X^{c,\alpha}$ be the process satisfying: $\E^{\P^\star} \big[\|X^{c,\alpha}\|^p \big]< \infty,$
	$X^{c,\alpha}_0= \xi$, and
	\begin{align} \label{eq:MKV_strong-law-of-controls}
		\mathrm{d}X^{c,\alpha}_t
		= 
		b \big(t, X^{c,\alpha}_t, \mub^{c,\alpha}_t, \alpha(t,X^{c,\alpha}_t,\mu^{c,\alpha}) \big) \mathrm{d}t
		+
		\sigma\big(t, X^{c,\alpha}_t \big) \mathrm{d} W_t
		+
		\sigma_0 \mathrm{d}B_t
	\end{align}
	$\;\mbox{with}\,\mub^{c,\alpha}_t:=\Lc(X^{c,\alpha}_t, \alpha(t,X^{c,\alpha}_t,\mu^{c,\alpha})|\Gc_t )\;\mbox{and}\;\mu^{c,\alpha}_t:=\Lc(X^{c,\alpha}_t|\Gc_t ).$
	We will say that $X^{c,\alpha}$ is an $closed$--$loop$ McKean--Vlasov process associated to $\alpha.$
	The $closed$--$loop$ McKean--Vlasov control problem is given by
	    \begin{align} \label{eq:strong_Value}
	    	V^{c}_S
			~:=~
			\sup_{\alpha \in \Ac^c} 
			\Phi^{c}(\alpha)\;\mbox{,}	~~\Phi^{c}(\alpha):=\E \bigg[
				\int_0^T L\big(t, X^{c,\alpha}_t,\mub^{c,\alpha}_t, \alpha(t,X^{c,\alpha}_t,\mu^{c,\alpha}) \big) \mathrm{d}t 
				+ 
				g\big(X^{c,\alpha}_T, \mu^{c,\alpha}_T \big) 
			\bigg].
		\end{align}

\subsubsection{Equivalence and limit theory results} 
We now provide here the equivalence result and the limit theory relating to the McKean--Vlasov control problem. The proofs are in \Cref{proof_thm1_MFC} and \Cref{proof_thm2_MFC}.
\begin{theorem}[equivalence between $closed$--$loop$ and $open$--$loop$ controls] \label{thm:equivalence_mckv_cl}
   Let $X^{o,\alpha}$ be an $open$--$loop$  McKean--Vlasov process  associated to $\alpha \in \Ac^o.$ There exists a sequence $(\alpha^\ell)_{\ell \in \N^*} \subset \Ac^c$ s.t. $X^{c,\alpha^\ell}$ is a $closed$--$loop$ McKean--Vlasov process associated to $\alpha^\ell,$ and one has
   \begin{align*}
       \P \circ \Big( \mu^{o,\alpha}, \delta_{\mub^{o,\alpha}_t}(\mathrm{d}\nub)\mathrm{d}t \Big)^{-1}
       =
       \Lim_{\ell \to \infty }\P \circ \Big( \mu^{c,\alpha^\ell}, \delta_{\mub^{c,\alpha^\ell}_t}(\mathrm{d}\nub)\mathrm{d}t \Big)^{-1}\;\mbox{in}\;\Wc_p.
   \end{align*}
   Consequently,
    \begin{align*}
        V_S^{o}=V_S^{c}.
    \end{align*}
\end{theorem}
    
{\color{black}Recall that the definitions of $\varphi^N$ and $\overline{\varphi}^N$ are given in  \Cref{eq:N-agents_StrongMV_CommonNoise-law-of-controls}.} 
\begin{theorem}[Limit theory] \label{thm:limit_mckv_cl}
    For each $N \in \N^*,$ let $\varepsilon_N \in \R_{+}$ and $\alphab^N:=(\alpha^{1,N},\dots,\alpha^{N,N})$ be an $\varepsilon_N$--$closed$--$loop$ Pareto equilibrium. Then, the sequence $(\mathrm{P}^N)_{N \in \N^*}$ with $\mathrm{P}^N:=\mathrm{P}^N[\alphab^N] \in \Pc \big( \Cc^n_{\Wc} \x \M(\Pc^n_U) \big)$ is relatively compact in $\Wc_p$ where
    \[
        \mathrm{P}^N[\alphab^N]
        :=
        \Lc^{\P}\Big(\varphi^{N}[\alphab^N],\;\delta_{\overline{\varphi}^{N}_t[\alphab^N]}(\mathrm{d}\nub)\mathrm{d}t \Big)
    \]
    and if $\Lim_{N \to \infty} \varepsilon_N=0,$ then for each limit point $\mathrm{P}^\infty$, there exists a sequence $(\alpha^{\star,N})_{N \in \N^*} \subset \Ac^{c}$ s.t if we define for each $N \in \N^*,$ $\alphab^{\star,N}:=(\alpha^{\star,1,N},\dots,\alpha^{\star,N,N})$ where
    \begin{align*}
        \alpha^{\star,i,N}(t,\mathsf{x}^1,\dots,\mathsf{x}^N)
        :=
        \alpha^{\star,N}\big(t,\mathsf{x}^i(t),\pi^N\big)\;\mbox{with}\;\pi^N(t):=\frac{1}{N} \sum_{i=1}^N \delta_{\mathsf{x}^i(t)},\;\mbox{for}\;(\mathsf{x}^1,\dots,\mathsf{x}^N) \in (\Cc^n)^N\;\mbox{and}\;i \in \{1,\dots,N\}
    \end{align*}
    then
    \begin{align*}
        \mathrm{P}^\infty
        =
        \Lim_{N \to \infty} \Lc^{\P} \Big(\varphi^{N}[\alphab^{\star,N}], \delta_{\overline{\varphi}^{N}_t[\alphab^{\star,N}]}(\mathrm{d}\nub)\mathrm{d}t \Big)\;\mbox{in}\;\Wc_p,\;\mbox{therefore}\;V_S^{o}
        =
        V_S^{c}
        =
        \Lim_{N \to \infty}\; \sup_{\alphab \in (\Ac^c_N)^N}\; \frac{1}{N} \sum_{1=1}^N J_i ( \alphab ).
    \end{align*}

\end{theorem}

\begin{remark}
    This type of results has been evoked by some authors, see for instance {\rm \cite{lacker2017limit,Mmotte-Pham_2019,djete2019general,Lacker-Shkolnikov-Zhang_2020,MFD-2020}}. However, two points are new here: the absence of convexity assumptions and the presence of the law of controls. Indeed, first, with convexity assumptions, without the law of controls, these results have been established by some authors. Here, we are able to avoid this restriction. Second, in the presence of the law of controls, even with convexity assumptions, not many authors provide this kind of equivalence. The equivalence between $closed$--$loop$ and $open$--$loop$ is usually proved by representing the optimal $open$--$loop$ control as a $closed$--$loop$ control. This representation is done by $``$projecting$"$ the optimal $open$--$loop$ control on the state process $X_t$ and then using the convexity assumptions $($see {\rm \cite{lacker2017limit,Lacker-Shkolnikov-Zhang_2020}}$)$. The problem is, when there is the law of control, after $``$projection$"$, it is not possible to recover the law of control. Some information gets lost along the way.
    Let us mention that the equivalence between $closed$--$loop$ and $open$--$loop$ formulations appears to be false when the coefficients are path--dependent as shown by {\rm \citeauthor*{YONG2021104948} \cite{YONG2021104948}}.

\end{remark}

\section{Proof of main results} \label{sec:proofs}

    \subsection{Measure--valued equilibrium: a canonical formulation} \label{section:relaxed_formulation}

In this section, we present equivalent formulations of the measure--valued MFG equilibrium and McKean--Vlasov control problem over a canonical space. These formulations turn out to be those used in \cite{MFD-2020_MFG} and \cite{MFD-2020} to manage the $open$--$loop$ framework. Therefore, we will see from our proofs that the set of limits of $open$--$loop$ and $closed$--$loop$ controls is the same. These formulations have the advantage to facilitate the presentation of the proofs.

\subsubsection{Measure-valued control rules}

    Denote by $\M := \M \big(\Pc^n_U \big)$ the collection of all finite (Borel) measures $q(\mathrm{d}t, \mathrm{d}e)$ on $[0,T] \x \Pc^n_U$, 
	whose marginal distribution on $[0,T]$ is the Lebesgue measure $\mathrm{d}s$ 
	i.e., $q(\mathrm{d}s,\mathrm{d}e)=q(s,\mathrm{d}e)\mathrm{d}s$ for a measurable family $(q(s, \mathrm{d}e))_{s \in [0,T]}$ of Borel probability measures on $\Pc^n_U.$
	Let $\Lambda$ be the canonical element on $\M$.
	We then introduce a canonical filtration $\F^\Lambda = (\Fc^\Lambda_t)_{0 \le t \le T}$ on $\M$ by
	$$
		\Fc^\Lambda_t := \sigma \big\{ \Lambda(C \x [0,s]) ~: \forall s \le t,\;C \in \Bc(\Pc^n_U) \big\}.
	$$
	For each $q \in \M$, one has the disintegration property: $q(\mathrm{d}t, \mathrm{d}e) = q(t, \mathrm{d}e) \mathrm{d}t$, and there is a version of the disintegration
	such that $(t, q) \mapsto q(t, \mathrm{d}e)$ is $\F^\Lambda$-predictable.
	

    \vspace{4mm}
    \noindent
    The canonical element on $\Omb:=\Cc^n_{\Wc} \x \Cc^n_{\Wc} \x \M \x \M \x \Cc^n$ is denoted by $(\mu', \mu,\Lambda',\Lambda, B).$ Then, the canonical filtration $\Fb = (\Fcb_t)_{t \in [0,T]}$ is defined by: for all $t \in [0,T]$  
    $$ 
        \Fcb_t
        :=
        \sigma 
        \big\{\mu'_{t \wedge \cdot},\mu_{t \wedge \cdot},\Lambda'_{t \wedge \cdot}, \Lambda_{t \wedge \cdot}, B_{t \wedge \cdot} \big\}, 
    $$
    with $\Lambda'_{t \wedge \cdot}$ and $\Lambda_{t \wedge \cdot}$ denote the restriction of $\Lambda'$ and $\Lambda$ on $[0,t] \x \Pc^n_U$ (see definition \ref{eq:lambda}). Notice that we can choose a version of the disintegration $\Lambda(\mathrm{d}\nub,\mathrm{d}t)=\Lambda_t(\mathrm{d}\nub)\mathrm{d}t$ (resp $\Lambda'(\mathrm{d}\nub,\mathrm{d}t)=\Lambda'_t(\mathrm{d}\nub)\mathrm{d}t$) such that $(\Lambda_t)_{t \in [0,T]}$ (resp $(\Lambda'_t)_{t \in [0,T]}$) is a $\Pc(\Pc^n_U)$--valued $\Fb$--predictable process.
    Let us also introduce the filtration $(\Gcb_t)_{t \in [0,T]}$ by
    $$ 
        \Gcb_t
        :=
        \sigma \big\{ \mu_{t \wedge \cdot},\;\Lambda_{t \wedge \cdot},\;B_{t \wedge \cdot} \big\}.
    $$
    
    We consider $\Lc$ the following generator: for $(t,x,\nub,u) \in [0,T] \x \R^n \x \Pc^n_U \x U $, and $ \varphi \in C^2(\R^n)$
    \begin{align}
    \label{eq:def_generator}
        \Lc_t\varphi(x,\nub,u) 
        &:= 
        \Lc^\circ_t\varphi(x,u)
        +
         b^\star(t,\nub)^\top \nabla \varphi(x)
    \end{align}
    where
    \begin{align}
    \label{eq:def_generatorC}
        \Lc^\circ_t\varphi(x,u) 
        &:= 
        \frac{1}{2}  \text{Tr}\big[\sigma \sigma^\top(t,x) \nabla^2 \varphi(x) 
        \big] 
        + b^\circ(t,x,u)^\top \nabla \varphi(x).
    \end{align}
    Also, for every $f \in C^{2}(\R^n),$ let us define $N_t(f):=N_t[\mu',\mu,\Lambda',\Lambda](f)$ by
    \begin{align} \label{eq:FP-equation}
        N_t[\mu',\mu,\Lambda',\Lambda](f)
        :=
        \langle f(\cdot-\sigma_0 B_t),\mu'_t \rangle
	    -
	    \langle f,\mu'_0 \rangle
	    &-\int_0^t \bigg[ \int_{\Pc^n_U} \int_{\R^n}b^\star(r,\nub)^\top \nabla f(x-\sigma_0 B_t) \mu'_r(\mathrm{d}x)\Lambda_r(\mathrm{d}\nub) \nonumber
	    \\
	    &~~~~~~~~~-\int_{\Pc^n_U}  \langle \Lc^\circ_r [f(\cdot- \sigma_0 B_r)](\cdot,\cdot), \nub \rangle \Lambda'_r(\mathrm{d}\nub) \bigg]\mathrm{d}r,
    \end{align}
    and for each $\pi \in \Pc(\R^n),$ the Borel set $\Z_{\pi}$ by
    \begin{align*}
        \Z_{\pi}:=\Big\{ \nub \in \Pc^n_U: \nub(\mathrm{d}x,U)=\pi(\mathrm{d}x) \Big\}.
    \end{align*}
\begin{definition}[measure--valued control rule] \label{def:RelaxedCcontrol}
    We say that $\mathrm{P} \in \Pc(\Omb)$ is a measure--valued control rule if:
    \begin{enumerate}
        \item[$(i)$] $\mathrm{P} \big(\mu'_0=\nu \big)=1$.
        
        \item[$(ii)$] $(B_t)_{t \in [0,T]}$ is a $(\mathrm{P},\Fb)$ Wiener process starting at zero and for $\mathrm{P}$--almost every $\om \in \Omb$, $N_t(f)=0$ for all $f \in C^{2}_b(\R^n)$ and every $t \in [0,T].$
        
        \item[$(iii)$] $(\Lambda'_t)_{t \in [0,T]}$ is a $\Gb$--predictable process.
        
        \item[$(iv)$] For $\mathrm{d}\mathrm{P} \otimes \mathrm{d}t$ almost every $(t,\om) \in [0,T] \x \Omb$, $ {\Lambda}'_t\big(\Z_{\mu'_t} \big)=1.$
    \end{enumerate}

\end{definition}
    
    We shall denote $\Pcb_V$ the set of all measure--valued control rules. 

\begin{remark}
    To do an analogy with {\rm \Cref{def:measure--valued--MFG}}, the variables $(\mu',\mu,\Lambda',\Lambda,B)$ may be seen as follows: $B$ is the common noise, $\mu'$ plays the role of $(\Lc^{\Pr}(X'_t | \Gc_t ))_{t \in [0,T]},$ $\Lambda'$ that of {\color{black}$\delta_{\Lc(X'_s,\;\alpha(s,X'_s,\Lambda) | \Gc_s )}(\mathrm{d}m)\mathrm{d}s,$} $\mu$ and $\Lambda$ represent the fixed measures, in particular $\mu_t=\Lc^{\Pr}(X_t|\Gc_t)$.  
\end{remark}

\begin{remark} \label{rm:com_def}
   The set of measure--valued control rules $\Pcb_V$ that we introduced is the same as the one used in {\rm \cite{MFD-2020_MFG}}. 
   However, as we will see later $($see {\rm \Cref{prop:eq--measure--valuedMFG_control} and \Cref{rem:eq--measure--valuedMFG_control}}$)$, for the definition of measure--valued MFG equilibrium, only the case where $B$ is $(\sigma\{ \mu_{t \wedge \cdot}, \Lambda_{t \wedge \cdot}\})_{t \in [0,T]}$--adapted matters. 
\end{remark} 

{\color{black}The following result is one of the key steps to understanding the measurability property satisfied} by the Brownian motion $B.$
\begin{lemma} \label{lemma:measurability_B}
    For any $\Pr \in \Pcb_V,$ there exists a continuous function $\varphi:\Cc^{n,\ell}_{\Wc} \x \Cc^{n,\ell}_{\Wc} \x \M \x \M \to \Cc^n$ $\mbox{for}\;\;\ell \ge 1$ s.t.
    $$
        \Pr\big(B_t=\varphi_t(\mu'_{t \wedge \cdot},\mu_{t \wedge \cdot},\Lambda'_{t \wedge \cdot},\Lambda_{t \wedge \cdot}),\;t \in [0,T] \big)=1.
    $$
\end{lemma}

\begin{proof}
       Let $(\Omt,\Ft,\Prt)$ be an extension of $(\Omb,\Fb,\Pr)$ supporting an $\R^n$--valued $\F$--Brownian motion $W,$ and a $\Fct_0$--random variable $\xi$ s.t. $\Lc^{\Prt}(\xi)=\nu.$ Besides, $W,$ $\xi$ and $\Gcb_T$ are independent. Given $(\mu,\Lambda,\Lambda'),$ let $X'$ be the solution of
    \begin{align*}
        \mathrm{d}X'_t
        =
        \int_{\Pc^n_U} \int_U
        b(t,X'_t,\nub,u) \nub'^{X'_{t}}(\mathrm{d}u) \Lambda'_t(\mathrm{d}\nub')\Lambda_t(\mathrm{d}\nub)\mathrm{d}t
        +
        \sigma(t,X'_t)\mathrm{d}W_t
        +
        \sigma_0 \mathrm{d}B_t\;\mbox{with}\;X'_0=\xi
        ,\;\Prt\mbox{--a.e.}
    \end{align*}
    By the uniqueness of the previous equation, we can check that $\Lc^{\Prt}(X'_t|\Gcb_t)=\mu'_t,$ $\Prt$--a.e. for all $t \in [0,T]$ (see similar arguments in \Cref{thm:unique_stochastic-FP} ). By taking the conditional expectation, we find that
    \begin{align*}
        B_t=\sigma_0^{-1} \bigg[ \int_{\R^n} x \mu'_t(\mathrm{d}x)
        -
        \int_{\R^n} x \mu'_0(\mathrm{d}x)
        -
        \int_0^t\int_{(\Pc^n_U)^2}\; \int_{\R^n \x U} b \big(s, x',\nub, u' \big) \nub'(\mathrm{d}x',\mathrm{d}u') \Lambda'_s (\mathrm{d}\nub') \Lambda_s (\mathrm{d}\nub) \mathrm{d}s \bigg].
    \end{align*}
    Therefore, the function $\varphi$ is defined by
    \begin{align*}
        \varphi_t(\pi',\pi,q',q)
        :=
        \sigma_0^{-1} \bigg[ \int_{\R^n} x \pi'_t(\mathrm{d}x)
        -
        \int_{\R^n} x \pi'_0(\mathrm{d}x)
        -
        \int_0^t\int_{(\Pc^n_U)^2}\; \int_{\R^n \x U} b \big(s, x, \nub, u \big) \nub'(\mathrm{d}u,\mathrm{d}x)
        q'_s (\mathrm{d}\nub')q_s (\mathrm{d}\nub) \mathrm{d}s \bigg].
    \end{align*}
    This is enough to conclude.
\end{proof}

\medskip
Let $\Pr \in \Pcb_V$ s.t. $\Pr[\mu=\mu',\Lambda=\Lambda']=1$ and $\Pr' \in \Pcb_V$ satisfying $\Lc^{\Pr}(\mu,\Lambda,B)=\Lc^{\Pr'}(\mu,\Lambda,B).$ By the previous Lemma, we deduce that $ \Pr\big(B_t=\varphi_t(\mu_{t \wedge \cdot},\mu_{t \wedge \cdot},\Lambda_{t \wedge \cdot},\Lambda_{t \wedge \cdot}),\;t \in [0,T] \big)=1,$ and also that $ \Pr'\big(B_t=\varphi_t(\mu_{t \wedge \cdot},\mu_{t \wedge \cdot},\Lambda_{t \wedge \cdot},\Lambda_{t \wedge \cdot}),\;t \in [0,T] \big)=1.$ In other words, $B$ is a function of $(\mu,\Lambda).$  Let $(\Omt,\Ft,\Prt')$ be an extension of $(\Omb,\Fb,\Pr')$ supporting an $\R^n$--valued $\F$--Brownian motion $W,$ and a $\Fct_0$--random variable $\xi$ s.t. $\Lc(\xi)=\nu.$ Besides, $W,$ $\xi$ and $\Gcb_T$ are independent. 

\medskip
The next Lemma is essentially an application of \Cref{prop:approximation_weak_2}. It basically indicates that we can replace $\Lambda',$ which plays the role of a control, by a sequence of more regular controls. This fact will be useful to show the canonical formulation of the measure--valued MFG equilibrium (see \Cref{prop:eq--measure--valuedMFG_control}) and to deal with the convergence of Nash equilibria (see \Cref{prop:convergenceNashEquilibrium}).


\begin{lemma} \label{lemm:equiv_controls}
    There exists a sequence of continuous functions $(\widehat{\beta}^j)_{j \in \N^*}$ satisfying: for each $j \in \N^*,$
    \begin{align*}
        \widehat{\beta}^j: [0,T] \x \R^n \x \Cc^{n,1}_{\Wc} \x \M \x [0,1] \to  U\;\mbox{with}\; \widehat{\beta}^j(t,x,\pi,q)=\widehat{\beta}^j(t,x,\pi_{t \wedge \cdot},q_{t \wedge \cdot}),
    \end{align*}
    s.t. if we define  $X'^{j}$ the solution of   
\begin{align} \label{eq:appr_Y}
        \mathrm{d}X'^{j}_t
        =
        \int_{\Pc^n_U}
        b\big(t,X'^{j}_t,\nub,\widehat{\beta}^{j}(t,X'^{j}_t,\mu,\Lambda)\big) \Lambda_t(\mathrm{d}\nub) \mathrm{d}t
        +
        \sigma(t,X'^j_t) \mathrm{d}W_t
        +
        \sigma_0 \mathrm{d}B_t,\;X'^j_0=\xi\;\;\;\;\;\Prt'\mbox{--a.s.}
    \end{align}    
    then 
    \begin{align*}
        \Lim_{j \to \infty} \Prt' \circ \big( \mu'^j,\mu, \Lambda'^j, \Lambda, B \big)^{-1}
        =
        \Pr\;\;\mbox{in}\;\Wc_p,
    \end{align*}
    with $\mu'^j_t=\Lc^{\Pr}\big(X'^j_t\big|\Gcb_t\big),$  $\mub'^j_t=\Lc^{\Pr}\big(X'^j_t,\widehat{\beta}^{j}(t,X'^{j}_t,\mu,\Lambda)\big|\Gcb_t \big)$ and $\Lambda'^{j}=\delta_{\mub'^j_t}(\mathrm{d}\nub)\mathrm{d}t.$
\end{lemma}

\begin{proof}
       It is enough to apply \Cref{prop:approximation_weak_2} and using the information that $ \Pr'\big(B_t=\varphi_t(\mu_{t \wedge \cdot},\mu_{t \wedge \cdot},\Lambda_{t \wedge \cdot},\Lambda_{t \wedge \cdot}),\;t \in [0,T] \big)=1$ i.e. $B$ is a continuous function of $(\mu,\Lambda).$
\end{proof}
    Now, using the measure--valued control rules, we give an equivalent definition of the measure--valued MFG equilibrium.

\subsubsection{MFG equilibrium on the canonical space} 
    Let us introduce the set of measure--valued equilibrium
    \begin{align*}
        \Sc^{\star}
        :=
        \big\{
            \P \circ (\mu,\Lambda)^{-1}:\;\mbox{where}\;( \Om, \Fc,  \P,  \F, W, B, X, \Lambda, \mu )\;\mbox{is a measure--valued MFG equilibrium}
        \big\}.
    \end{align*}

\medskip    
    For all $(\pi',\pi, q', q) \in \Cc^n_{\Wc} \x \Cc^n_{\Wc} \x \M \x \M,$ one defines 
    \begin{align} \label{eq:function_j}
        J\big(\pi', \pi, q', q \big)
        :=
        \int_0^T \bigg[\int_{\Pc^n_U}\langle L^\circ\big(t,\cdot,\cdot \big), \nub' \rangle q'_t(\mathrm{d}\nub') +
        \int_{\Pc^n_U}\langle L^\star\big(t,\cdot,\nub \big), \pi'_t \rangle q_t(\mathrm{d}\nub)
        \bigg]
        \mathrm{d}t 
        +
        \langle g(\cdot,\pi_T),\pi'_T \rangle.
\end{align}

\begin{proposition}  \label{prop:eq--measure--valuedMFG_control}
        The probability measure $\mathrm{Q}^{\star}$ belongs to $\Sc^{\star}$ if and only if $\mathrm{Q}^{\star}=\Pr^\star \circ(\mu,\Lambda)^{-1}$ where $\mathrm{P}^\star \in \Pcb_V$, and for every $\mathrm{P} \in \Pcb_V$ such that $\Lc^{\mathrm{P}^{\star}}\big(\mu,\Lambda,B\big)=\Lc^{\mathrm{P}}\big(\mu,\Lambda, B \big)$, one has
        \begin{align} \label{eq:optimality-relaxed}
            \E^{\mathrm{P}^{\star}}\big[J(\mu',\mu,\Lambda',\Lambda) \big] \ge \E^{\mathrm{P}} \big[J(\mu',\mu,\Lambda',\Lambda) \big],
        \end{align}
        and for $\mathrm{P}^{\star}$--almost every $\om \in \Omb,$
        \begin{align} \label{eq:consistency}
            \Lambda'_t(\mathrm{d}\nub)\mathrm{d}t
            =
            \Lambda_t\big(\mathrm{d}\nub \big)\mathrm{d}t\;\mbox{and}\;\mu'=\mu.
        \end{align}
        
\end{proposition}

\begin{remark} \label{rem:eq--measure--valuedMFG_control}
    This definition of equilibrium is exactly the one proposed in {\rm \cite{MFD-2020_MFG}} for the $open$--$loop$ case. 
    But as mention in {\rm \Cref{lemm:equiv_controls}} $($see also {\rm\Cref{rm:com_def}}$),$ as $\Pr^\star[\mu=\mu',\Lambda=\Lambda']=1,$ $B$ is in fact $(\sigma\{ \mu_{t \wedge \cdot}, \Lambda_{t \wedge \cdot}\})_{t \in [0,T]}$--adapted. Therefore, for the definition of measure--valued MFG equilibrium, we only need to focus on measure--valued control rules s.t. $\Lambda'$ is $(\sigma\{ \mu_{t \wedge \cdot}, \Lambda_{t \wedge \cdot}\})_{t \in [0,T]}$--predictable.  
\end{remark}

\begin{proof}
    Let $( \Om, \Fc,  \Pr,  \F, W, B, X, \Lambda, \mu )$ be a measure--valued MFG equilibrium. Let us check that $\Pr^\star:=\Lc^{\P}\big( \mu, \mu, \Lambda, \Lambda, B \big)$ satisfies the desired properties \eqref{eq:optimality-relaxed} and \eqref{eq:consistency}. Recall that, for each $t \in [0,T],$ $\mu_t=\Lc^{\Pr} \big(X_t \big| \Gc_t \big)$ $\Pr$--a.e. where $\Gc_t:=\sigma\{\mu_{t \wedge \cdot},\Lambda_{t \wedge \cdot},B_{t \wedge \cdot} \}$ and $X$ is a weak solution of
			\begin{align*}
			    \mathrm{d}X_t
		        = 
		        \int_{\Pc^n_U}\; \int_{U} b \big(t, X_t, \nub, u \big) \nub^{X_t} (\mathrm{d}u) \;\Lambda_t (\mathrm{d}\nub) \mathrm{d}t
		        +
		        \sigma\big(t, X_t \big) \mathrm{d} W_t
		        +
		        \sigma_0 \mathrm{d}B_t,\;\;X_0=\xi
			\end{align*}
    and $\Lambda_t(\Z_{\mu_t})=1$ $\mathrm{d}\Pr \otimes \mathrm{d}t$--a.e. By definition, the property \eqref{eq:consistency} is obviously verified. It is straightforward that $\Pr^\star \in \Pcb_V.$
    Consequently, by \Cref{lemma:measurability_B}, $B_t$ is a continuous function of $(\mu_{t \wedge \cdot}, \Lambda_{t \wedge \cdot}).$  This means there exists a continuous function $\varphi:\Cc^{n,1}_{\Wc} \x \M \to \Cc^n$ s.t. $\Pr\big(B_t=\varphi_t(\mu_{t \wedge \cdot},\Lambda_{t \wedge \cdot}),\;t \in [0,T] \big)=1.$
    
\medskip    
    Let $\Pr \in \Pcb_V$ with $\Pr \circ (\mu,\Lambda,B)^{-1}=\Pr^\star \circ (\mu,\Lambda,B)^{-1},$ by \Cref{lemm:equiv_controls}, there exists $(\Pr^j)_{j \in \N^*} \subset \Pcb_V$ s.t for all $j \in \N^*,$ $\Pr^j \circ (\mu,\Lambda,B)^{-1}=\Pr^\star \circ (\mu,\Lambda,B)^{-1},$ $\Pr^j:=\P^\star \circ \big( \mu'^j, \mu, \Lambda'^{j}, \Lambda,B \big)^{-1}$ where $(\mu'^j, \mu, \Lambda'^{j}, \Lambda)$ is defined in \Cref{lemm:equiv_controls}  and
    \begin{align*}
        \Lim_{j \to \infty} \E^{\Pr^j} \big[ J(\mu',\mu,\Lambda', \Lambda) \big]
        =
        \E^{\Pr} \big[ J(\mu',\mu,\Lambda', \Lambda) \big].
    \end{align*}
    As $( \Om, \Fc,  \Pr,  \F, W, B, X, \Lambda, \mu )$ is a measure--valued MFG equilibrium, one has the second property i.e. the optimality
    \begin{align*}
        \E^{\Pr} \big[ J(\mu',\mu,\Lambda', \Lambda) \big]
        =
        \Lim_{j \to \infty} \E^{\Pr^j} \big[ J(\mu',\mu,\Lambda', \Lambda) \big]
        \le 
        \E^{\Pr^\star} \big[ J(\mu',\mu,\Lambda', \Lambda) \big].
    \end{align*}
    This is enough to conclude one part.
    
\medskip    
    Now, let $\Qr^\star \in \Pcb_V$ satisfying \eqref{eq:optimality-relaxed} and \eqref{eq:consistency}. Using \Cref{lemma:measurability_B}, we can verify that $B$ is a continuous function of $(\mu,\Lambda).$ Let $\mu,$  $\Lambda$ and $B$ be random variables on $(\Om,\H,\P)$ s.t. $\Lc^{\P} (\mu,\Lambda,B,W,\xi)=\Lc^{\Qr^\star}(\mu,\Lambda,B) \otimes \Lc^{\P}(W,\xi).$ Let $X$ be the solution of
    \begin{align*}
        \mathrm{d}X_t
        =
        \int_{\Pc^n_U} \int_U
        b(t,X_t,\nub,u) \nub^{X_{t}}(\mathrm{d}u) \Lambda_t(\mathrm{d}\nub)\mathrm{d}t
        +
        \sigma(t,X_t)\mathrm{d}W_t
        +
        \sigma_0 \mathrm{d}B_t,\;X_0=\xi
        ,\;\P\mbox{--a.e.}
    \end{align*}
    then by uniqueness $\mu_t=\Lc^{\P}(X_t | \mu,\Lambda)=\Lc^{\P}(X_t |\mu_{t \wedge \cdot}, \Lambda_{t \wedge \cdot}),$ $\P$--a.e. for all $t \in [0,T].$ It is then easy to verify that $( \Om, \Fc,  \P,  \F, W, B, X, \Lambda, \mu )$ is a measure--valued MFG equilibrium and $\Lc^{\P}(\mu,\Lambda)=\Lc^{\Qr}(\mu,\Lambda)$. 
\end{proof}
    Next, using the same measure--valued control rules, we define an equivalent formulation of the McKean--Vlasov optimal control.
\subsubsection{McKean--Vlasov optimal control on the canonical space}
Let us define
\begin{align*}
    \Vc
        :=
        \big\{
            \Pr \in \Pcb_V:\;\;\Pr  \big( \mu=\mu', \Lambda=\Lambda' \big)=1
        \big\},
\end{align*}
$\Kc:=\big\{
            \Pr \in \Pcb_V:\exists\;\gamma: [0,T] \x \Cc^n \to \Pc(\R^n)\;\mbox{s.t.}\;\;\mathrm{d}\Pr \otimes \mathrm{d}t\mbox{--a.e.}\;(t,\om), \mu_t=\gamma(t,B_{t \wedge \cdot})
        \big\}$ and
\begin{align*}
    \Vc^c
        :=
        \Big\{
            \Pr \in \Vc \cap \Kc: \exists\; \alpha \in \Ac^c\;\mbox{s.t.}\;\mathrm{d}\Pr \otimes \mathrm{d}t\mbox{--a.e.}\;(t,\om),\;\Lambda_t(\om) \big(\nub:\;\nub=\delta_{\alpha(t,x,\mu(\om))}(\mathrm{d}u) \mu_t(\om)(\mathrm{d}x) \big)=1
        \Big\}    
\end{align*}

\begin{proposition} \label{prop:canonical_McKV}
    We have the following reformulation of $V^{c}_S$
    \begin{align*}
        V^{c}_S
        =
        \sup_{\Pr \in \Vc^c} \E^{\Pr} \big[J (\mu',\mu,\Lambda,\Lambda') \big].
    \end{align*}
    where $J$ is defined in {\rm \Cref{eq:function_j}}.  Besides, under {\rm \Cref{assum:main1}}, one has
    \begin{align*}
        V^{o}_S
        =
        \sup_{\Pr \in \Vc} \E^{\Pr} \big[J (\mu',\mu,\Lambda,\Lambda') \big].
    \end{align*}
\end{proposition}
\begin{proof}
    The first result is just a reformulation. 
    {\color{black}The second one comes from \cite[Theorem 3.1]{MFD-2020}. An idea of the proof of \cite[Theorem 3.1]{MFD-2020} is provided in \cite[Section 5]{MFD-2020}. The proof consists of an approximation of the Fokker--Planck control equation by a sequence of $open$--$loop$ McKean--Vlasov equations}.
\end{proof}

\subsection{Limit of Nash equilibria} \label{sec:limit_NashEquilibria}

This section is devoted to the analysis of the behavior of sequence of $closed$--$loop$ Nash equilibria when the number of players goes to infinity. At the end of this section, we will show that any limit point of sequence of Nash equilibria is a measure--valued MFG equilibrium.

\subsubsection{Technical results}

Recall that $(\Om, \Hc,\P)$ is a filtered probability space supporting a $\Hc_0$--random variable $\xi$ s.t. $\Lc^{\P}(\xi)=\nu,$ and an $\R^{n} \x \R^n$--valued $\H$--Brownian motion $(W,B).$  Besides, let $\mu$ be a $\H$--adapted $\Pc(\R^n)$--valued continuous process and $\Lambda$ be a $\Pc(\Pc^n_U)$--valued $\H$ predictable process s.t. $B$ is $\G:=(\Gc_t)_{t \in [0,T]}$--adapted where $\Gc_t:=\sigma\{\mu_{t \wedge \cdot}, \Lambda_{t \wedge \cdot}\}.$ Besides, $(\mu,\Lambda)$ is $\P$--independent of $(\xi,W).$ For a continuous function 
\begin{align*}
    \beta: [0,T] \x \R^n \x \Cc^{n,1}_{\Wc} \x \M \to U,\;\mbox{satisfying}\; \beta(t,x,\pi,q)=\beta(t,x,\pi_{t \wedge \cdot},q_{t \wedge \cdot}),
\end{align*}
let $X$ be the unique strong solution of: $\E^{\P}[\|X\|^{p'}]< \infty$ with $p' > p,$ for all $t \in [0,T]$
    \begin{align} \label{eq:aux:rules}
        \mathrm{d}X'_t
        =
        \int_{\Pc^n_U}
        b(t,X'_t,\nub,\beta(t,X'_t,\mu,\Lambda)) \Lambda_t(\mathrm{d}\nub) \mathrm{d}t
        +
        \sigma(t,X'_t) \mathrm{d}W_t
        + \sigma_0 \mathrm{d}B_t\;\mbox{with}\;X_0=\xi.
    \end{align}
Also, we denote for all $t \in [0,T],$ $\mu'_t:=\Lc(X'_t|\Gc_t),$ $\mub'_t:=\Lc(X'_t, \beta(t,X'_t,\mu,\Lambda)|\Gc_t)$ a.e. and $\Lambda':=\delta_{\mub'_t}(\mathrm{d}\nub)\mathrm{d}t.$

\medskip
Notice that the process \eqref{eq:aux:rules} is exactly the one used for the approximation of measure--valued control rule in \Cref{lemm:equiv_controls}. In what follows, we will show that it is possible to approximate this type of process by a sequence of interacting processes when a certain condition (see \Cref{cond:converg}) is satisfied (see \Cref{lemm:convergenceNashEquilibrium}).

\medskip
For each $N \in \N^*,$ and any $N$--progressively Borel measurable functions $\alphab:=(\alpha^1,\cdots,\alpha^N)$  s.t. $\alpha^i: [0,T] \x (\Cc^n)^N \to U$ and define  
\begin{align} \label{eq:def-controlBeta}
    \beta^{\alphab,i}(t,\xb)
    :=
    \beta \big(t, \xf^i(t), \pi[\xb], q^{\alphab}[\xb]  \big),
\end{align}
where
\begin{align*}
    \xb:=(\x^1,\cdots,\x^N),\;
    \pi[\xb]:=\frac{1}{N} \sum_{j=1}^N \delta_{\xf^j},\; \overline{m}^{\alphab}(t,\xb):=\frac{1}{N} \sum_{j=1}^N \delta_{\big(\xf^j(t),\;\alpha^j(t,\xb) \big)},\;\mbox{and}\;q^{\alphab}[\xb]:=\delta_{\overline{m}^{\alphab}(t,\xb)}(\mathrm{d}\nub)\mathrm{d}t.
\end{align*}

\begin{lemma}
\label{lemm:convergenceNashEquilibrium}
    Let {\rm \Cref{assum:main1}} hold true and a sequence $(\alpha^i)_{i \in \N^*}$ s.t. for each $N \in \N^*,$ $\alphab^N:=(\alpha^1,\dots,\alpha^N) \in (\Ac^c_N)^N$ and
    \begin{align} \label{cond:converg}
        \Lim_{N \to \infty} \P \circ \Big( \varphi^{N}[\alphab^N], \delta_{\big( \overline{\varphi}^{N}_t[\alphab^N] \big)} (\mathrm{d}\nub')\mathrm{d}t, B \Big)^{-1}
        =
        \P \circ ( \mu,\Lambda, B )^{-1}\;\;\mbox{in}\;\Wc_p.
    \end{align}
    Then
\begin{align*}
    &\Lim_{N \to \infty}
    \frac{1}{N} \sum_{i=1}^N
    J_i \big( \alpha^1, \cdots, \alpha^{i-1}, \beta^{\alphab^N,i}, \alpha^{i+1},\cdots, \alpha^N \big)
        =
        \E^{\P}\big[J(\mu', \mu,\Lambda',\Lambda) \big].
\end{align*}
\end{lemma}

\begin{proof}

    The proof is largely inspired by \cite[Proof of Proposition 5.6]{lacker2020-closed}. We use some probability changes on $(\Om,\H,\P).$ Given $\alphab:=(\alpha^1,\dots,\alpha^N) \in (\Ac^c_N)^N,$ $\Xbb^{\alphab}:=(X^{\alphab,1},\dots,X^{\alphab,N})$ satisfies \eqref{eq:N-agents_StrongMV_CommonNoise-law-of-controls}.
    For each $i \in \{1,\cdots,N\},$ we define 
    $$
        \alphab^i:=(\alphab^{[-i]},\beta^{\alphab,i})\;\;\mbox{and}\;\;\Ybb^i:=(Y^{i,1},\cdots,Y^{i,N}):=(X^{\alphab^i,1},\dots,X^{\alphab^i,N}).
    $$
    Notice that, $Y^{i,i}$ satisfies
    \begin{align*} 
        \mathrm{d}Y^{i,i}_t
        =
        b\big(t,Y^{i,i}_{t},\mub^{i,N}_{t} ,\beta^{\alphab,i}(t,\Ybb^{i}) \big) \mathrm{d}t 
        +
        \sigma \big(t,Y^{i,i}_t \big) \mathrm{d}W^i_t
        +
        \sigma_0 \mathrm{d}B_t\;\mbox{with}\;Y^{i,i}_0=X^i_0
    \end{align*}
	with 
	\[
	    \mu^{i,N}_{t}(\mathrm{d} x) := \frac{1}{N}\sum_{j=1}^N \delta_{Y^{i,j}_{t} }(\mathrm{d} x),\;\;
	    \mub^{i,N}_{t}(\mathrm{d} x, \mathrm{d} u) := \frac{1}{N} \big(\sum_{j \neq i}^N \delta_{(Y^{i,j}_{t} ,\;\alpha^j(t,\Ybb^{i}) )}(\mathrm{d} x, \mathrm{d} u) + \delta_{(Y^{i,i}_{t} ,\;\beta^{\alphab,i}(t,\Ybb^{i})) }(\mathrm{d} x, \mathrm{d} u) \big)
	\] 
	$\mbox{and}\;\;\Lambda^{i,N}:=\delta_{\mub^{i,N}_t}(\mathrm{d}\nu)\mathrm{d}t.$

\medskip
Now, let us introduce
	\begin{align*}
	    Z^i_t
	    :=
	    \exp\bigg\{ \int_0^t \phi^i_r \mathrm{d}W^i_r - \frac{1}{2} \int_0^t |\phi^i_r|^2 \mathrm{d}r \bigg\}\;\mbox{for all}\;t \in [0,T],\; \mbox{and}\;\frac{\mathrm{d}\Q^i}{\mathrm{d}\P}:=Z^i_T
	\end{align*}
	with 
	\begin{align*}
	    \phi^i_t:=\sigma(t,X^{\alphab,i}_t)^{-1}\Big(b\big(t,X^{\alphab,i}_{t},\overline{\zeta}^{i,N,\beta}_{t} ,\beta^{\alphab,i}(t,\Xbb^{\alphab}) \big) - b\big(t,X^{\alphab,i}_{t},\overline{\varphi}^{N}_{t}[\alphab] ,\alpha^i(t,\Xbb^{\alphab}) \big) \Big)
	\end{align*}
	and
	\begin{align*}
	    \overline{\zeta}^{i,N,\beta}_t
	    :=
	    \frac{1}{N} \big(\sum_{j \neq i}^N \delta_{(X^{\alphab,j}_{t} ,\;\alpha^j(t,\Xbb^{\alphab}) )}(\mathrm{d} x, \mathrm{d} u) + \delta_{(X^{\alphab,i}_{t} ,\;\beta^{\alphab,i}(t,\Xbb^{\alphab})) }(\mathrm{d} x, \mathrm{d} u) \big).
	\end{align*}
	 By uniqueness, we can check that $\Lc^{\Q^i} (\Xbb^{\alpha},B)=\Lc^{\P}(\Ybb^i,B)$ for each $i.$ Notice that
	\begin{align*}
	    J_i \big( \alpha^1, \cdots, \alpha^{i-1}, \beta^{\alphab,i}, \alpha^{i+1},\cdots, \alpha^N \big)
	    &=
	    \E^{\P} \bigg[
        \int_0^T L\big(t,Y^{i,i}_t,\mub^{i,N}_{t} ,\beta^{\alphab,i}(t,\Ybb^i) \big) \mathrm{d}t 
        + 
        g \big( Y^{i,i}_T, \mu^{i,N}_T \big)
        \bigg]
        \\
        &=
        \E^{\Q^i} \bigg[
        \int_0^T L\big(t,X^{\alphab,i}_t,\overline{\zeta}^{i,N,\beta}_t ,\beta^{\alphab,i}(t,\Xbb^{\alphab}) \big) \mathrm{d}t 
        + 
        g \big( X^{\alphab,i}_T, \varphi^N_T[\alphab] \big)
        \bigg]
        \\
        &=
        \E^{\P} \bigg[ Z^i_T \bigg(
        \int_0^T L\big(t,X^{\alphab,i}_t,\overline{\zeta}^{i,N,\beta}_t ,\beta^{\alphab,i}(t,\Xbb^{\alphab}) \big) \mathrm{d}t 
        + 
        g \big( X^{\alphab,i}_T, \varphi^N_T[\alphab]\big) \bigg)
        \bigg].
	\end{align*}
	Next, let us consider the sequence $(\mathrm{Q}^N)_{N \in \N^*} \subset \Pc \Big( \Pc\big( \Cc^n \x \Cc^n \x \Cc^1 \x \Cc^n_{\Wc} \x \M(\Pc^n_U \x U) \x \Cc^n \big) \x \Cc^n_{\Wc} \x \M(\Pc^n_U) \x \Cc^n \Big)$
	\begin{align*}
	    \mathrm{Q}^N
	    :=
	    \P \circ \bigg( \frac{1}{N} \sum_{i=1}^N \delta_{\big(X^{\alphab,i}-\sigma_0B_t,\;W^i,\;Z^i,\;\varphi^N[\alphab],\;\Phi^i,\;\gamma^N,\;B \big)},\;\;\varphi^N[\alphab],\;\gamma^N,\;\;B \bigg)^{-1}
	\end{align*}
	where
	\begin{align*}
	    \Phi^i:=\delta_{\big(\overline{\zeta}^{i,N,\beta}_{t},\; \alpha^i(t,\Xbb^{\alphab}) \big)}(\mathrm{d}m,\mathrm{d}u)\mathrm{d}t\;\mbox{and}\;\gamma^N:=\delta_{{\overline{\varphi}}^{N}_{t}[\alphab]}(\mathrm{d}\nub)\mathrm{d}t.
	\end{align*}
	By using the fact that $(b,\sigma)$ is bounded and the non--degeneracy of $\sigma,$ with similar arguments to \cite[Proposition A.1, Proposition A.2, Proposition B.1]{Lacker_carmona_delarue_CN}, it is straightforward to check that the sequence $(\mathrm{Q}^N)_{N \in \N^*}$ is relatively compact in $\Wc_p$ (recall that $\Lc^\P(\xi)=\nu \in \Pc_{p'}(\R^n) $ with $p'>p$). Denote $\mathrm{Q}^\infty$ the limit of a sub--sequence, for sake of simplicity, we will use the same notation for the sequence and its sub--sequence. We now want to identify the limit.
	Mention that $Z^i$ verifies: $\mathrm{d}Z^i_t= Z^i_t \phi^i_t \mathrm{d}W^i_t$ with $Z^i_0=1.$ For any function twice differentiable $f: \R^n \x \R^n \x \R \to \R,$ let us define
	\begin{align*}
	    &\Ac f(t,x,w,h,z,\pi,q,m,u)
	    \\
	    &:=
	    \nabla_x f(x,w,z) b\big(t,x+\sigma_0 h,m ,a \big)
	    +
	    \frac{1}{2} \mathrm{Tr} \big[\nabla^2_x f(x,w,z)  \sigma(t,x+\sigma_0 h)\sigma(t,x+\sigma_0 h)^\top\big]
	    \\
	    &+
	    \frac{1}{2} \mathrm{Tr} \big[ \nabla^2_w f(x,w,z) \big]
	    +
	    \frac{1}{2} \mathrm{Tr} \big[ \nabla^2_z f(x,w,z) \big|\sigma(t,x+\sigma_0 h)^{-1}\big(b\big(t,x+\sigma_0 h,m ,\beta(t,x+\sigma_0 h,\pi,q) \big) - b\big(t,x+\sigma_0 h,m ,u \big) \big) \big|^2 \big] |z|^2
	    \\
	    &+
	    \nabla_w\nabla_x f(x,w,z) \sigma(t,x+\sigma_0 h)
	    +
	    \nabla_z\nabla_x f(x,w,z) z \big(b\big(t,x+\sigma_0 h,m ,\beta(t,x+\sigma_0 h,\pi,q) \big) - b\big(t,x+\sigma_0 h,m ,u \big) \big)
	    \\
	    &+
	    \nabla_z\nabla_w f(x,w,z)z \sigma(t,x+\sigma_0 h)^{-1} \big(b\big(t,x+\sigma_0 h,m ,\beta(t,x+\sigma_0 h,\pi,q) \big)- b\big(t,x+\sigma_0 h,m ,u \big) \big)
	\end{align*}
	and for $(t,\mathsf{x},\mathsf{w},\mathsf{z},\mathsf{b},\pi,q,\widehat{\qb}) \in [0,T] \x \Cc^n \x \Cc^n \x \Cc^1 \x \Cc^n \x \Cc^n_{\Wc} \x \M(\Pc^n_U) \x \M(\Pc^n_U \x U) $
	\begin{align*}
	    \mathrm{M}f[t,\mathsf{x},\mathsf{w},\mathsf{b},\mathsf{z},\pi,q,\widehat{\qb}]
	    :=
	    f(\mathsf{x}(t),\mathsf{w}(t),\mathsf{z}(t))
	    -
	    \int_0^t \int_{\Pc^n_U \x U}
	    \Ac f(s,\mathsf{x}(s),\mathsf{w}(s),\mathsf{z}(s),\mathsf{b}(s),\pi,\qb,m,u) \widehat{\qb}_s(\mathrm{d}m,\mathrm{d}u)\mathrm{d}s.
	\end{align*}
	We denote by $\Omh:=\Cc^n \x \Cc^n \x\Cc^n \x \Cc^1\x \Cc^n_{\Wc} \x \M(\Pc^n_U \x U),$ $(\widehat{X},\widehat{W},\widehat{B},\widehat{Z},\widehat{\mu},\widehat{\Phi},\widehat{\gamma})$ the canonical processes, and $\widehat{\F}$ its canonical filtration. Also, we set $(\Pi,\mu,\gamma,B)$ the canonical element of $\Om:=\Pc(\Omh) \x \Cc^n_{\Wc} \x \M(\Pc^n_U) \x \Cc^n.$ Let $0 \le s \le t \le T $ and a continuous bounded function $h: \Omh \to \R,$ we define
	\begin{align*}
	    h_s:=
	    h(\Xh_{s \wedge \cdot},\Wh_{s \wedge \cdot},\Bh_{s \wedge \cdot},\Zh_{s \wedge \cdot}, \muh_{s \wedge \cdot},\Phih_{s \wedge \cdot}).
	\end{align*}
	Let $\widehat{X}^{\alphab,i}=X^{\alphab,i}-\sigma_0 B.$ By Itô formula, we know that
	\begin{align*}
	    \mathrm{d}f(\widehat{X}^{\alphab,i}_t,W^i_t,Z^i_t)
	    &=
	    \nabla_x f(\widehat{X}^{\alphab,i}_t,W^i_t,Z^i_t) \sigma(t,X^{\alphab,i}_t) \mathrm{d}W^i_t
	    +
	    \nabla_w f(\widehat{X}^{\alphab,i}_t,W^i_t,Z^i_t) \mathrm{d}W^i_t
	    +
	    \nabla_z f(\widehat{X}^{\alphab,i}_t,W^i_t,Z^i_t) Z^i_t \phi^i_t \mathrm{d}W^i_t
	    \\
	    &+
	    \nabla_x f(\widehat{X}^{\alphab,i}_t,W^i_t,Z^i_t) b\big(t,X^{\alphab,i}_{t},\overline{\varphi}^{N}_{t}[\alphab] ,\alpha^i(t,\Xbb^{\alphab}) \big) \mathrm{d}t
	    +
	    \frac{1}{2} \mathrm{Tr} \big[\nabla^2_x f(\widehat{X}^{\alphab,i}_t,W^i_t,Z^i_t)  \sigma(t,X^{\alphab,i}_t)\sigma(t,X^{\alphab,i}_t)^\top\big] \mathrm{d}t
	    \\
	    &+
	    \frac{1}{2} \mathrm{Tr} \big[ \nabla^2_w f(\widehat{X}^{\alphab,i}_t,W^i_t,Z^i_t) \big] \mathrm{d}t
	    +
	    \frac{1}{2} \mathrm{Tr} \big[ \nabla^2_z f(\widehat{X}^{\alphab,i}_t,W^i_t,Z^i_t) \phi^i_t (\phi^i_t)^\top \big] |Z^i_t|^2 \mathrm{d}t
	    \\
	    &+
	    \nabla_w\nabla_x f(\widehat{X}^{\alphab,i}_t,W^i_t,Z^i_t) \sigma(t,X^{\alphab,i}_t) \mathrm{d}t
	    +
	    \nabla_z\nabla_x f(\widehat{X}^{\alphab,i}_t,W^i_t,Z^i_t) \sigma(t,X^{\alphab,i}_t) Z^i_t \phi^i_t \mathrm{d}t
	    \\
	    &+
	    \nabla_z\nabla_w f(\widehat{X}^{\alphab,i}_t,W^i_t,Z^i_t) Z^i_t \phi^i_t \mathrm{d}t.
	\end{align*}
	Let us introduce $h^i_s:=h(\widehat{X}^{\alphab,i}_{s \wedge \cdot},W^i_{s \wedge \cdot},Z^i_{s \wedge \cdot}, \varphi^N_{s \wedge \cdot}[\alphab],\Phi^i_{s \wedge \cdot})$ and
	\begin{align*}
	    \mathrm{M}^i_tf
	    &:=
	    f(\widehat{X}^{\alphab,i}_t,W^i_t,Z^i_t)
	    \\
	    &-
	    \int_0^t \bigg[
	    \nabla_x f(\widehat{X}^{\alphab,i}_r,W^i_r,Z^i_r) b\big(r,X^{\alphab,i}_{r},\overline{\varphi}^{N}_r[\alphab] ,\alpha^i(r,X^{\alphab}) \big)
	    +
	    \frac{1}{2} \mathrm{Tr} \big[\nabla^2_x f(\widehat{X}^{\alphab,i}_r,W^i_r,Z^i_r)  \sigma(r,X^{\alphab,i}_r)\sigma(r,X^{\alphab,i}_r)^\top\big]
	    \\
	    &\;\;+
	    \frac{1}{2} \mathrm{Tr} \big[ \nabla^2_w f(\widehat{X}^{\alphab,i}_r,W^i_r,Z^i_r) \big]
	    +
	    \frac{1}{2} \mathrm{Tr} \big[ \nabla^2_z f(\widehat{X}^{\alphab,i}_r,W^i_r,Z^i_r) \phi^i_r (\phi^i_r)^\top \big] |Z^i_r|^2
	    \\
	    &\;\;+
	    \nabla_w\nabla_x f(\widehat{X}^{\alphab,i}_r,W^i_r,Z^i_r) \sigma(r,X^{\alphab,i}_r)
	    +
	    \nabla_z\nabla_x f(\widehat{X}^{\alphab,i}_r,W^i_r,Z^i_r) \sigma(r,X^{\alphab,i}_r) Z^i_r \phi^i_r
	    +
	    \nabla_z\nabla_w f(\widehat{X}^{\alphab,i}_r,W^i_r,Z^i_r) Z^i_r \phi^i_r \bigg] \mathrm{d}r.
	\end{align*}
	
	Notice that for $\mathrm{d}t \otimes \mathrm{d}\P$--a.e.
	\begin{align} \label{eq:equal_limit}
	    \Lim_{N \to \infty} \frac{1}{N} \sum_{i=1}^N \Wc_p \Big( \overline{\zeta}^{i,N,\beta}_{t}, \overline{\varphi}^{N}_t[\alphab] \Big)
	    =
	    0.
	\end{align}
	Then, in the sub--sequence $(\mathrm{Q}^N)_{N \in \N^*},$ we can use $\overline{\zeta}^{i,N,\beta}_{t}$ or $\overline{\varphi}^{N}_t[\alphab]$ without {\color{red}affecting } the limit $\Qr^\infty.$ With all the previous observations, one can check that 
	\begin{align*}
	    &\E^{\mathrm{Q}^\infty} \Big[ \Big|\E^{\Pi} \Big[\Big(\mathrm{M}f[t,\Xh,\Wh,\Bh,\Zh,\muh,\gammah,\Phih]-  \mathrm{M}f[s,\Xh,\Wh,\Bh,\Zh,\muh,\gammah,\Phih] \Big) h_s \Big]\Big|^2 \Big]
	    \\
	    &=
	    \Lim_N \E^{\mathrm{Q}^N} \Big[ \Big|\E^{\Pi} \Big[\Big(\mathrm{M}f[t,\Xh,\Wh,\Bh,\Zh,\muh,\gammah,\Phih]-  \mathrm{M}f[s,\Xh,\Wh,\Bh,\Zh,\muh,\gammah,\Phih] \Big) h_s \Big]\Big|^2 \Big]
	    =
	    \Lim_N \E^{\P} \Big[ \Big|\frac{1}{N} \sum_{i=1}^N \big(\mathrm{M}^i_t f-  \mathrm{M}^i_s f \big) h^i_s\Big|^2 \Big]
	    \\
	    &=
	    \Lim_N \frac{1}{N^2} \sum_{i=1}^N \E^{\P} \Big[ \Big| \big( \mathrm{M}^i_t f-  \mathrm{M}^i_s f \big) h^i_s\Big|^2 \Big]
	    =0.
	\end{align*}
	 Notice {\color{blue}that the bounded character of the coefficients} $(b,\sigma)$ and test functions $(f,h)$ is important for passing to the limit. {\color{black}This is true for all $(f,h)$. We can then use an appropriate countable set of maps of type $f$ and $g$, and deduce that $\mathrm{Q}^\infty$ $\om$--a.e., for each $\varphi$ bounded twice differentiable,  $\mathrm{M}\varphi$ is a $(\Pi(\om),\widehat{\F})$--martingale} (see similar method in \cite[Proof of Proposition 5.1]{lacker2017limit}, \cite[Proof of Proposition 5.6]{lacker2020-closed}, \cite[roposition 4.17]{djete2019general}). Also, using \Cref{eq:equal_limit}, it is easy to check that
	\begin{align*}
	    \Pi(\om) \big[ \muh=\mu(\om),\;\gammah=\gamma(\om),\;\Bh=B(\om) \big]=1,\;\mathrm{Q}^\infty\;\om\mbox{--a.e.}
	\end{align*}
	
	\medskip
	Consequently, $\mathrm{Q}^\infty\;\om\mbox{--a.e.},$ on an extension $\big( \widehat{\Om}, \widehat{\F}, \widehat{\P}_{{\om}} \big):=\big( \widehat{\Om}\x [0,1], (\widehat{\Fc}_t \otimes \Bc([0,1]))_{t \in [0,T]}, \Pi({\om}) \otimes \lambda \big)$ of $\big( \widehat{\Om}, (\widehat{\Fc}_t)_{t \in [0,T]}, \Pi({\om}) \big),$ there exists $\widehat{M}$ a $(\widehat{\F}, \widehat{\P}_{{\om}})$--martingale measure with quadratic variation $\Phih$ s.t. $\widehat{\P}_{\om}$--a.e.
	\begin{align*}
	    &\mathrm{d}\Xh_t
        =
        \int_{\Pc^n_U}
        b^\star(t,\nub) \gamma_t(\om)(\mathrm{d}\nub) \mathrm{d}t
        +
        \int_{U}
        b^\circ(t,\Xt_t,u) \Phih_t(\Pc^n_U,\mathrm{d}u) \mathrm{d}t
        +
        \sigma(t,\Xt_t) \mathrm{d}\Wh_t,\;\Wh_t=\Mh(\Pc^n_U \x U \x [0,t]),
        \\
        &\mathrm{d}\Zh_t
        = \int_{\Pc^n_U \x U} \Zh_t \; \sigma(t,\Xt_t)^{-1}\Big(b\big(t,\Xt_{t},\nub ,\beta(t,\Xt_t,\mu(\om),\gamma(\om)) \big) - b\big(t,\Xt_{t},\nub ,u \big) \Big) \Mh(\mathrm{d}\nub,\mathrm{d}u,\mathrm{d}t)\;\;\mbox{with}\;\Xt:=\Xh+\sigma_0 B(\om),
	\end{align*}
	and $\widehat{\P}_{\om} \circ \big(\Xh,\Wh,\Bh,\Zh, \mu(\om), \Phi,\gamma(\om) \big)^{-1}=\Pi(\om).$

\medskip	
	Now, we define the following change of probability
	\begin{align*}
	    \frac{\mathrm{d}\widehat{\Q}_\om}{\mathrm{d}\widehat{\P}_\om}
	    :=
	    \Zh_T=\exp\Big\{ \Delta_T - \frac{1}{2} \langle \Delta \rangle_T \Big\}
	\end{align*}
	where
	\begin{align*}
	    \Delta_t:=\int_0^t \int_{\Pc^n_U \x U} \sigma(r,\Xt_r)^{-1}\Big(b\big(r,\Xt_{r},\nub ,\beta(r,\Xt_r,\mu(\om),\gamma(\om)) \big) - b\big(r,\Xt_{r},\nub ,u \big) \Big) \Mh(\mathrm{d}\nub,\mathrm{d}u,\mathrm{d}r).
	\end{align*}
	Let us introduce 
	\begin{align*}
	    \widehat{V}_t
	    :=
	    \Wh_t
	    -
	    \int_0^t \int_{\Pc^n_U \x U} \Big(b\big(r,\Xt_{r},\nub ,\beta(r,\Xt_r,\mu(\om),\gamma(\om)) \big) - b\big(r,\Xt_{r},\nub ,u \big) \Big) \Mh(\mathrm{d}\nub,\mathrm{d}u,\mathrm{d}r),
	\end{align*}
	then, by Girsanov's theorem, $\widehat{V}$ is a $(\widehat{\F},\widehat{\Q}_\om)$--Brownian motion, and $\Xh$ satisfies
	\begin{align} \label{eq:change-alter}
	    \Xh_t
        =
        \xi
        +
        \int_0^t \int_{\Pc^n_U}
        b\big(r,\Xh_{r}+\sigma_0 B_r(\om),\nub ,\beta(r,\Xh_r+\sigma_0 B_r(\om),\mu(\om),\gamma(\om)) \big) \gamma_r(\om)(\mathrm{d}\nub) \mathrm{d}r
        +
        \int_0^t
        \sigma(r,\Xh_r+\sigma_0 B_r(\om)) \mathrm{d}\widehat{V}_r.
	\end{align}
	By \Cref{eq:equal_limit}, one finds that
    \begin{align*}
        \mathrm{Q}^\infty \circ \big( \mu,\gamma, B \big)^{-1}
        =
        \Lim_N \P \circ \Big(\varphi^{N}[\alphab], \delta_{\overline{\varphi}^{N}_t[\alphab]}(\mathrm{d}\nub)\mathrm{d}t , B \Big)^{-1}
        =
        \P \circ (\mu,\Lambda, B)^{-1}.
    \end{align*}
    By uniqueness of \eqref{eq:change-alter} (and \Cref{eq:aux:rules}), we deduce that
    \begin{align*}
        \int_{\Om} \Lc^{\widehat{\Q}_{\om}} \big( \Xh, \mu(\om), \gamma(\om), B(\om) \big) \mathrm{Q}^\infty(\mathrm{d} \om)
        =
        \E^{\P} \big[ \Lc^{\P} \big( X' - \sigma_0 B, \mu, \Lambda, B \big| \Gc_T \big)
        \big]
        =
        \P \circ \big( X' - \sigma_0 B,\mu,\Lambda, B \big)^{-1}.
    \end{align*}
    This is true for any limit point $\mathrm{Q}^\infty,$ then we have the convergence of the entire sequence $(\mathrm{Q}^N)_{N \in \N^*}.$ Therefore
    \begin{align*}
	    &\Lim_{N \to \infty}\frac{1}{N} \sum_{i=1}^N J_i \big( \alpha^1, \cdots, \alpha^{i-1}, \beta^{\alphab,i}, \alpha^{i+1},\cdots, \alpha^N \big)
	    \\
	    &=
        \Lim_{N\to \infty} \frac{1}{N} \sum_{i=1}^N \E^{\P} \bigg[ Z^i_T \bigg(
        \int_0^T L\big(t,X^{\alphab,i}_t,\overline{\zeta}^{i,N,\beta}_{t} ,\beta^{\alphab,i}(t,\Xbb^{\alphab}) \big) \mathrm{d}t 
        + 
        g \big( X^{\alphab,i}_T, \varphi^{N}_T[\alphab] \big) \bigg)
        \bigg]
        \\
        &=
        \int_{\Om} \E^{\Pi(\om)} \bigg[ \Zh_T \bigg(
        \int_0^T \int_{\Pc^n_U} L\big(t,\Xt_t,\nub ,\beta(t,\Xt_t,\mu(\om),\gamma(\om) \big) \gamma_t(\om)(\mathrm{d}\nub) \mathrm{d}t 
        + 
        g \big( \Xt_T, \mu_T(\om) \big) \bigg)
        \bigg] \mathrm{Q}^\infty(\mathrm{d}\om)
        \\
        &=
        \int_{\Om} \E^{\widehat{\Q}(\om)} \bigg[         \int_0^T \int_{\Pc^n_U} L\big(t,\Xt_t,\nub ,\beta(t,\Xt_t,\mu(\om),\gamma(\om) \big) \gamma_t(\om)(\mathrm{d}\nub) \mathrm{d}t 
        + 
        g \big( \Xt_T, \mu_T(\om) \big)
        \bigg] \mathrm{Q}^\infty(\mathrm{d}\om)
        \\
        &=
        \E^{\P} \bigg[         \int_0^T \int_{\Pc^n_U} L\big(t,X'_t,\nub ,\beta(t,X'_t,\mu,\Lambda \big) \Lambda_t(\mathrm{d}\nub) \mathrm{d}t 
        + 
        g \big( X'_T, \mu_T \big)
        \bigg].
	\end{align*}
\end{proof}


\medskip

\medskip
Now, by combining the previous proposition and \Cref{lemm:equiv_controls}, we show that any measure--valued control rule satisfying a certain condition ( see \Cref{cond:converg-2}) is the limit of a sequence of interacting processes.
\begin{proposition}
\label{prop:convergenceNashEquilibrium}
    Let $\Pr \in \Pcb_V$ s.t. there exists a continuous function $\varphi:\Cc^{n,1}_{\Wc} \x \M \to \Cc^n$ s.t. $\Pr(B_t=\varphi_t(\mu_{t \wedge \cdot},\Lambda_{t \wedge \cdot}),\;t \in [0,T])=1,$ and a sequence $(\alpha^i)_{i \in \N^*}$ s.t. for each $N \in \N^*,$ $\alphab^N:=(\alpha^1,\dots,\alpha^N) \in (\Ac^c_N)^N$ and
    \begin{align} \label{cond:converg-2}
        \Lim_{N \to \infty} \P \circ \Big( \varphi^{N}[\alphab^N], \delta_{\big( \overline{\varphi}^{N}_t[\alphab^N] \big)} (\mathrm{d}\nub)\mathrm{d}t, B \Big)^{-1}
        =
        \Pr \circ \big( \mu,\Lambda, B \big)^{-1}.
    \end{align}
    Then there exists a sequence $(\beta^{i,N})_{(i,N) \in \{1,\cdots,N\} \x \N^*}$ satisfying for each $(i,N) \in \{1,\cdots,N\} \x \N^*,$ $\beta^{i,N} \in \Ac^c_N,$ and
\begin{align*}
    &\Lim_{N \to \infty}
    \frac{1}{N} \sum_{i=1}^N
    J_i \big( \alpha^1, \cdots, \alpha^{i-1}, \beta^{i,N}, \alpha^{i+1},\cdots, \alpha^N \big)
        =
        \E^{\Pr}\big[J(\mu',\mu,\Lambda',\Lambda) \big].
\end{align*}

\end{proposition}

\begin{proof}
    
    By \Cref{lemm:equiv_controls}, we can assume that $\Pr= \P\circ \big( \mu', \mu, \Lambda', \Lambda,B\big)^{-1}$ where for each $t \in [0,T],$ $\mu'_t=\Lc^{\P}\big(X'_t\big|\Gc_t\big),$  $\mub'_t=\Lc^{\P}\big(X'_t,\widehat{\beta}(t,X'_t,\mu,\Lambda)\big|\Gc_t \big)$ and $\Lambda'=\delta_{\mub'_t}(\mathrm{d}\nub)\mathrm{d}t,$ with
    
    \begin{align*} 
        \mathrm{d}X'_t
        =
        \int_{\Pc^n_U}
        b(t,X'_t,\nub^\prime,\widehat{\beta}(t,X'_r,\mu,\Lambda)) \Lambda_r(\mathrm{d}\nub^\prime) \mathrm{d}t
        +
        \sigma(t,X'_t) \mathrm{d}W_t
        +
        \sigma_0 \mathrm{d}B_t\;\mbox{with}\;X'_0=\xi
    \end{align*}
    and $\widehat{\beta}$ is a continuous function satisfying:
    \begin{align*}
        \widehat{\beta}: [0,T] \x \R^n \x \Cc^{n,1}_{\Wc} \x \M \to  U\;\mbox{with}\; \widehat{\beta}(t,x,\pi,q)=\widehat{\beta}(t,x,\pi_{t \wedge \cdot},q_{t \wedge \cdot}).
    \end{align*}

    Then by \Cref{lemm:convergenceNashEquilibrium}, if $(\widehat{\beta})^{\alpha,i}$ is defined in \eqref{eq:def-controlBeta}, one has
\begin{align*}
    &\Lim_{N \to \infty}
    \frac{1}{N} \sum_{i=1}^N
    J_i \big( \alpha^1, \cdots, \alpha^{i-1}, (\widehat{\beta})^{\alphab^N,i}, \alpha^{i+1},\cdots, \alpha^N \big)
        =
        \E^{\P^\star_u}\big[J(\mu',\mu,\Lambda',\Lambda) \big].
\end{align*}
By combining our previous results, we get the proof of the proposition.

\end{proof}

\subsubsection{Proof of \Cref{thm:limit_thm_closed-loop} (limit set of approximate Nash equilibria)}\label{proof_thm_1_MFG}

By using \cite[Proposition 5.4.]{MFD-2020} 
, one finds $(\mathrm{P}^N)_{N \in \N^*}$ is relatively compact in $\Wc_p$ where
    \begin{align*}
        \mathrm{P}^N
        :=
        \P^\star \circ \Big( (\varphi^{N}_t[\alphab^N])_{t \in [0,T]},(\varphi^{N}_t[\alphab^N])_{t \in [0,T]},\delta_{\overline{\varphi}^{N}_s[\alphab^N]}(\mathrm{d}\nub)\mathrm{d}s, \delta_{\overline{\varphi}^{N}_s[\alphab^N]}(\mathrm{d}\nub')\mathrm{d}s, B  \Big)^{-1},
    \end{align*}
    and each limit point $\mathrm{P}^\infty$ of any sub--sequence belongs to $\Pcb_V.$ Next, let us show that $\mathrm{P}^\infty \circ (\mu,\Lambda)^{-1} \in \Sc^\star.$ To simplify, the sequence $(\mathrm{P}^N)_{N \in \N^*}$ and its sub--sequence share the same notation. It is straightforward that $\Pr^\infty [\mu=\mu',\Lambda=\Lambda']=1.$ Therefore, by \Cref{lemma:measurability_B}, there exists a continuous function $\varphi:\Cc^{n,1}_{\Wc} \x \M \to \Cc^n$ s.t. $\Pr^\infty(B_t=\varphi_t(\mu_{t \wedge \cdot},\Lambda_{t \wedge \cdot}),\;t \in [0,T])=1.$

\medskip    
    Let $\mathrm{P} \in \Pcb_V$ such that $\Lc^{\mathrm{P}} \big(\mu, \Lambda,B \big)=\Lc^{\mathrm{P}^\infty} \big(\mu, \Lambda,B \big).$ By \Cref{prop:convergenceNashEquilibrium}, there exists a sequence $(\beta^{i,N})_{(i,N) \in \{1,\cdots,N\}\x \N^* }$ satisfying for each $(i,N) \in \{1,\cdots,N\}\x \N^*,$ $\beta^{i,N} \in \Ac^c_N,$
    and
\begin{align*}
    &
    {\color{red}\Lim_{N \to \infty}}\frac{1}{N} \sum_{i=1}^N
    J_i \big( \alpha^1, \cdots, \alpha^{i-1}, \beta^{i,N}, \alpha^{i+1},\cdots, \alpha^N \big)
        =
        \E^{\Pr}\big[J(\mu',\mu,\Lambda',\Lambda) \big].
\end{align*}
Therefore
\begin{align*}
    \E^{\mathrm{P}^\infty}\big[J(\mu',\mu,\Lambda',\Lambda) \big]
    &=
    \Lim_{N \to \infty} \frac{1}{N} \sum_{i=1}^N J_i (\alphab^N)
    \\
    &\ge 
    \Lim_{N \to \infty} \bigg(\frac{1}{N} \sum_{i=1}^N
    J_i \big( \alpha^1, \cdots, \alpha^{i-1}, \beta^{i,N}, \alpha^{i+1},\cdots, \alpha^N \big) - \varepsilon_N \bigg)
    =
    \E^{\mathrm{P}}\big[J(\mu',\mu,\Lambda',\Lambda) \big],
\end{align*}
then $\E^{\mathrm{P}^\infty}\big[J(\mu',\mu,\Lambda',\Lambda) \big] \ge \E^{\mathrm{P}}\big[J(\mu,\mu',\Lambda,\Lambda') \big],$ for any $\mathrm{P} \in \Pcb_V$ such that $\Lc^{\mathrm{P}} \big(\mu, \Lambda, B \big)=\Lc^{\mathrm{P}^\infty} \big(\mu, \Lambda, B \big).$ We conclude that $\mathrm{P}^\infty \circ (\mu,\Lambda)^{-1} \in \Sc^\star.$

\subsection{Limit set of approximate strong Markovian MFG equilibria and approximation of measure--valued MFG equilibrium} \label{sec:limitset_converse}

In this section, first, we will prove the second part of \Cref{thm:limit_thm_closed-loop} namely showing that the limit set of approximate strong Markovian MFG solution is the measure--valued MFG equilibrium. Second, we will prove the results of \Cref{thm:converse-limit}. 

\medskip
To achieve these goals, we need to introduce the notion of approximate $open$--$loop$ MFG equilibrium.
Let $\F^B:=(\Fc^B_t)_{t \in [0,T]}$ be the natural filtration of $B$ i.e.  $\Fc^B_t:=\sigma\{B_s: s \le t \}.$ Let $\alpha$ be a $(\sigma\{X_0, W_{t \wedge \cdot}, B_{t \wedge \cdot}\})_{t \in [0,T]}$--predictable process. We denote by $X$ the unique solution of
\begin{align} \label{eq:MKV_strong_op-MFG}
		\mathrm{d}X_t
		= 
		b \big(t, X_t, \mub_t, \alpha_t \big) \mathrm{d}t
		+
		\sigma(t, X_t ) \mathrm{d} W_t
		+
		\sigma_0 \mathrm{d}B_t\;\;\;\mbox{with}\;X_0=\xi,\;\mu_t:=\Lc(X_t|\Fc^B_t)\;\mbox{and}\;\mub_t:=\Lc(X_t, \alpha_t|\Fc^B_t).
	\end{align}
	
	Given $\mub,$ let $\alpha'$ be another $(\sigma\{X_0, W_{t \wedge \cdot}, B_{t \wedge \cdot}\})_{t \in [0,T]}$--predictable process and $X'$ be the solution of 
\begin{align} \label{eq:open_loop-aux}
		\mathrm{d}X'_t
		= 
		b \big(t, X'_t, \mub_t, \alpha'_t \big) \mathrm{d}t
		+
		\sigma(t, X'_t) \mathrm{d} W_t
		+
		\sigma_0 \mathrm{d}B_t\;\;\;\mbox{with}\;X'_0=\xi.
	\end{align}
	We denote $\mu^{\alpha'}_t:=\Lc(X'_t|\Fc^B_t)$ and $\mub^{\alpha'}_t:=\Lc(X'_t,\alpha'_t|\Fc^B_t)$ a.e. Let $\varepsilon \ge 0,$ we say that the McKean--Vlasov process $X$ associated to the control $\alpha$ is an  $\varepsilon$--$open$--$loop$ MFG equilibrium if : recall that $J$ is defined in \Cref{eq:function_j},
	\begin{align*}
	    \E\big[J\big(\mu,\mu,\delta_{\mub_t}(\mathrm{d}\nub)\mathrm{d}t, \delta_{\mub_t}(\mathrm{d}\nub)\mathrm{d}t\big)\big]
	    \ge 
	    \sup_{\alpha'}\E\big[J\big(\mu^{\alpha'},\mu,\delta_{\mub^{\alpha'}_t}(\mathrm{d}\nub)\mathrm{d}t, \delta_{\mub_t}(\mathrm{d}\nub)\mathrm{d}t\big)\big] - \varepsilon.
	\end{align*}

\subsubsection{Approximate strong Markovian MFG equilibrium as an approximate $open$--$loop$ MFG equilibrium}

The next Proposition show that any approximate strong Markovian MFG equilibrium can be seen as an approximate $open$--$loop$ MFG equilibrium.

\begin{proposition} \label{prop:strongM_as_strongO}
    For any $\varepsilon$--strong Markovian MFG equilibrium $(\mub_t)_{t \in [0,T]}=(\Lc(X_t,\alpha(t,X_t,\mu)|\Fc^B_t))_{t \in [0,T]},$ the process $X$ associated to the control $(\alpha(t,X_t,\mu))_{t \in [0,T]}$ is an $\varepsilon$--$open$--$loop$ MFG equilibrium.
\end{proposition}

\begin{proof}
       As $X$ is a strong solution of \Cref{eq:MKV_strong-MFG} and $\mu$ is $\F^B$--adapted, $(\alpha(t,X_t,\mu))_{t \in [0,T]}$ is a $(\sigma\{X_0, W_{t \wedge \cdot}, B_{t \wedge \cdot}\})_{t \in [0,T]}$--predictable process. To prove our proposition, it suffices to show the optimality condition. Let $\alpha'$ be  a $(\sigma\{X_0, W_{t \wedge \cdot}, B_{t \wedge \cdot}\})_{t \in [0,T]}$--predictable process and $X'$ be the solution of \Cref{eq:open_loop-aux}. Let us introduce $\mu'_t:=\Lc(X'_t|\Fc^B_t)$ and $\mub'_t:=\Lc(X'_t,\alpha'_t|\Fc^B_t).$ We denote by $R'$ the solution of  
       \begin{align*}
           \mathrm{d}R'_t
		    = 
		    \int_{U}b \big(t, R'_t, \mub_t, u \big) \mub'^{R'_t}_t \mathrm{d}t
		    +
		    \sigma(t, R'_t) \mathrm{d} W_t
		    +
		    \sigma_0 \mathrm{d}B_t\;\;\;\mbox{with}\;R'_0=\xi,
       \end{align*}
       where $(\mub'^x_t)_{x \in \R^n}$ is the disintegration of $\mub'_t$ i.e. $\mub'_t(\mathrm{d}u,\mathrm{d}x)=\mub'^x_t(\mathrm{d}u)\mu_t(\mathrm{d}x).$ By uniqueness, we can check that $\mu'_t=\Lc(R'_t|\Fc^B_t)$ a.e. Now, we can apply \Cref{prop:approximation_weak_2}. By \Cref{prop:approximation_weak_2}, there exists a sequence of functions $(\beta^k)_{k \in \N^*}$ satisfying: for each $k \in \N^*,$ $[0,T] \x \R^n \x \Cc^n \x \M \ni (t,x,b,q) \to \beta^k(t,x,b,q) \in U$ is progressively measurable and Lipschitz in $(x,b,q)$ uniformly in $t$ such that if we let $X'^k$ be the unique strong solution of:
    \begin{align*}
        \mathrm{d}X'^{k}_t
        =
        b(t,X'^{k}_t,\mub_t,\beta^k(t,X'^{k}_t,B,\Lambda)) \mathrm{d}t
        +
        \sigma(t,X'^k_t) \mathrm{d}W_t
        +
        \sigma_0 \mathrm{d}B_t,\;X'^k_0=\xi,\;\Lambda:=\delta_{\mub_t}(\mathrm{d}\nub)\mathrm{d}t,
    \end{align*}
    with $\mu'^k_t:=\Lc(X'^{k}_t|\Fc^B_t)$ and $\mub'^k_t:=\Lc\big(X'^{k}_{t},\beta^k(t,X^{k}_t, B,\Lambda) | \Fc^B_t\big)$  then
    \begin{align*}
        \Lim_{k \to \infty} \big(\mu'^k, \delta_{\mub'^k_t}(\mathrm{d}\nub) \mathrm{d}t \big)= \big(\mu', \delta_{\mub'_t}(\mathrm{d}\nub)\mathrm{d}t \big),\;\P\mbox{--a.e.},\;\mbox{for the Wasserstein metric}\;\Wc_p.
    \end{align*}
    With the same techniques used in \Cref{lemma:measurability_B}, we can show that $B$ is $(\sigma\{ \mu_{t \wedge \cdots}, \Lambda_{t \wedge \cdot}\})_{t \in [0,T]}$--adapted. As a result, we can find a progressively Borel measurable function $\tilde \beta^k:[0,T] \x \R^n \x \Cc^n_\Wc \x \M \to U$ such that $\beta^k(t,X'^k_t,B,\Lambda)=\tilde \beta^k(t,X'^k_t,\mu,\Lambda)$ a.e.   Consequently, as $\mub$ is an $\varepsilon$--strong Markovian MFG equilibrium, one has
    \begin{align*}
        \E\big[J\big(\mu,\mu,\delta_{\mub_t}(\mathrm{d}\nub)\mathrm{d}t, \delta_{\mub_t}(\mathrm{d}\nub)\mathrm{d}t\big)\big]
        &\ge \Lim_{k \to \infty} 
	    \E\big[J\big(\mu'^k,\mu,\delta_{\mub'^k_t}(\mathrm{d}\nub)\mathrm{d}t, \delta_{\mub_t}(\mathrm{d}\nub)\mathrm{d}t\big)\big] - \varepsilon
	    \\
	    &\ge 
	    \E\big[J\big(\mu',\mu,\delta_{\mub'_t}(\mathrm{d}\nub)\mathrm{d}t, \delta_{\mub_t}(\mathrm{d}\nub)\mathrm{d}t\big)\big] - \varepsilon.
    \end{align*}
    This is enough to conclude.
\end{proof}

\subsubsection{Proof of \Cref{thm:limit_thm_closed-loop} (limit set of approximate strong Markovian MFG equilibria)}\label{proof_thm_2_MFG}

Let $(\mub^\ell)_{\ell \in \N*}$ be a sequence s.t. for each $\ell \ge 1,$ $\mub^\ell$ is an $\varepsilon_\ell$--strong Markovian MFG equilibrium for $\varepsilon_\ell \ge 0.$ By \Cref{prop:strongM_as_strongO}, $\mub^\ell$ can be seen as an $\varepsilon_\ell$ $open$--$loop$ MFG equilibrium. Let $(\Pr^\ell)_{\ell \in \N^*}$ be the sequence defined by
\begin{align*}
    \Pr^\ell
    :=
    \P \circ (\mu^\ell,\mu^\ell,\delta_{\mub^\ell_t}(\mathrm{d}\nub)\mathrm{d}t, \delta_{\mub^\ell_t}(\mathrm{d}\nub)\mathrm{d}t, B)^{-1}.
\end{align*}
It is enough to apply \cite[Theorem 2.12]{MFD-2020_MFG} to conclude that the sequence $(\Pr^\ell)_{\ell \in \N^*}$ of approximate $open$--$loop$ MFG equilibria is relatively compact in $\Wc_p$ and that, when $\lim_{\ell \to \infty} \varepsilon_\ell =0,$ each limit point $\Pr$ is such that $\Pr \circ (\mu,\Lambda)^{-1} \in \Sc^\star.$ An idea of the proof of \cite[Theorem 2.12]{MFD-2020_MFG} is provided in \cite[Section 3]{MFD-2020_MFG}. First, there is an identification of any limit point as a measure--valued control rule satisfying $\mu=\mu'$ and $\Lambda=\Lambda$ a.e. under the corresponding measure. Second, any measure--valued control rule can be approximated by a sequence of controlled processes of type \eqref{eq:open_loop-aux}. The second point is used to show the optimality condition.

\subsubsection{Approximation of approximate $open$--$loop$ MFG equilibrium by approximate strong Markovian MFG equilibrium}

In this part, we show that any approximate $open$--$loop$ MFG equilibrium is the limit of a sequence of approximate strong Markovian MFG equilibria.

\medskip
Let $X$ be an $\varepsilon$--$open$--$loop$ MFG equilibrium associated to the control $\alpha.$ Recall that $\mu$ and $\mub$ are defined in \Cref{eq:MKV_strong_op-MFG}.
\begin{proposition} \label{prop:open-closed-approx}
    There exists a sequence of Borel functions $(\beta^k)_{k \in \N^*}$ satisfying: for each $k \in \N^*,$
    \begin{align*}
        \beta^k: [0,T] \x \R^n \x C([0,T];\Pc_p(\R^n)) \to  U\;\mbox{with}\; \beta^k(t,x,\pi)=\beta^k(t,x,\pi_{t \wedge \cdot}).
    \end{align*}
    such that $\beta^k$ is lipschitz in $(x,\pi)$ uniformly in $t,$ and if we let $X^k$ be the unique strong solution of:
    \begin{align*}
        \mathrm{d}X^{k}_t
        =
        b(t,X^{k}_t,\mub^k_t,\beta^k(t,X^{k}_t,\mu^k_t)) \mathrm{d}t
        +
        \sigma(t,X^k_t) \mathrm{d}W_t
        +
        \sigma_0 \mathrm{d}B_t
    \end{align*}
    with $\mu^k_t:=\Lc(X^{k}_t|\Fc^B_t)$ and $\mub^k_t:=\Lc\big(X^{k}_{t},\beta^k(t,X^{k}_t,\mu^k) | \Fc^B_t\big)$  then
    \begin{align*}
        \Lim_{k \to \infty} \big(\mu^k, \delta_{\mub^k_t}(\mathrm{d}\nub) \mathrm{d}t \big)= \big(\mu, \delta_{\mub_t}(\mathrm{d}\nub)\mathrm{d}t \big),\;\P\mbox{--a.e.},\;\mbox{for the Wasserstein metric}\;\Wc_p.
    \end{align*}
\end{proposition}

\begin{proof}
       This is essentially an application of \Cref{prop:approximation_weak_2} and \Cref{prop:appr_general_closed}. Indeed, since $\mu$ and $\mub$ are adapted to the canonical filtration of B, by \Cref{prop:approximation_weak_2}, there exists a sequence of functions $(\gamma^k)_{k \in \N^*}$ Lipschitz in $(x,b)$ uniformly in $t$ satisfying: for each $k \in \N^*,$
    \begin{align*}
        \gamma^k: [0,T] \x \R^n \x \Cc^n \to  U\;\mbox{with}\; \gamma^k(t,x,b)=\gamma^k(t,x,b_{t \wedge \cdot})
    \end{align*}
    such that if we let $X^k$ be the unique strong solution of:
    \begin{align*}
        \mathrm{d}X^{k}_t
        =
        b(t,X^{k}_t,\mub^k_t,\gamma^k(t,X^{k}_t,B)) \mathrm{d}t
        +
        \sigma(t,X^k_t) \mathrm{d}W_t
        +
        \sigma_0 \mathrm{d}B_t\;\mbox{where}\;\Lambda:=\delta_{\mub_t}(\mathrm{d}\nub)\mathrm{d}t,
    \end{align*}
    with $\mu^k_t:=\Lc(X^{k}_t|B_{t \wedge \cdot},\mu_{t \wedge \cdot},\Lambda_{t \wedge \cdot})=\Lc(X^{k}_t|B_{t \wedge \cdot})$ and $\mub^k_t:=\Lc\big(X^{k}_{t},\gamma^k(t,X^{k}_t, B) | B_{t \wedge \cdot},\mu_{t \wedge \cdot},\Lambda_{t \wedge \cdot}\big)=\Lc\big(X^{k}_{t},\gamma^k(t,X^{k}_t, B) | B_{t \wedge \cdot}\big)$  then
    \begin{align*}
        \Lim_{k \to \infty} \big(\mu^k, \delta_{\mub^k_t}(\mathrm{d}\nub) \mathrm{d}t \big)= \big(\mu, \Lambda_t(\mathrm{d}\nub)\mathrm{d}t \big),\;\P\mbox{--a.e.},\;\mbox{for the Wasserstein metric}\;\Wc_p.
    \end{align*} 
    We can rewrite $X^k$ as 
    \begin{align*}
        \mathrm{d}X^{k}_t
        =
        b\big(t,X^{k}_t,\delta_{\widetilde{\gamma}^k(t,x,B)}(\mathrm{d}u)\mu^k_t(\mathrm{d}x),\widetilde{\gamma}^k(t,X^{k}_t,B) \big) \mathrm{d}t
        +
        \sigma(t,X^k_t) \mathrm{d}W_t
        +
        \sigma_0 \mathrm{d}B_t.
    \end{align*}
    Now, we take $k$ as fixed. By \Cref{prop:appr_general_closed}, there exists a sequence of functions $(\phi^j)_{j \in \N^*}$ s.t. for each $j \in \N^*,$ $\phi^j: [0,T] \x \Cc^{n,p}_{\Wc} \ni (t,\pi) \to \phi^j(t,\pi_{t \wedge \cdot}) \in  \Cc^n$ with $\phi^j$ Lipschitz in $\pi$ uniformly in $t$ satisfying for each $j \in \N^*$
    \begin{align*}
        B_t
        =
        \phi^j(t,\zeta^j)
    \end{align*}
    where $\zeta^j_t=\Lc(X^j_t|\Fc^B_t)=\Lc(X^j_t|\Fc^B_T)$ with $X^j$ satisfying
\begin{align*} 
    \mathrm{d}X^j_t
    =
     b\big(t,X^{j}_t,\delta_{\widetilde{\gamma}^k(t,x,B)}(\mathrm{d}u)\zeta^j_t(\mathrm{d}x),\widetilde{\gamma}^k(t,X^{j}_t,B) \big)\mathrm{d}t 
    +
    \sigma(t,X^j_t)\mathrm{d}W_t
    +
    \sigma_0 \mathrm{d}B_t,
\end{align*}
 and $\Lim_{j \to \infty} \E \bigg[\int_0^T\| \zeta^j_t-\mu^k_t \|_{\rm TV}\;\mathrm{d}t\bigg]=0.$ If we define $\beta^k(t,x,\pi):=\widetilde{\gamma}^k\big(t,x, (\phi^j(s,\pi))_{s \in [0,t]} \big).$ This is enough to conclude.
 
\end{proof}

\begin{proposition}\label{prop:open_closed}
    Let us stay in the context in {\rm \Cref{prop:open-closed-approx}}. There exists a sequence $(\varepsilon_k)_{k \in \N^*} \subset (0,\infty)$ s.t. $\Limsup_{k \to \infty} \varepsilon_k \in [0,\varepsilon]$ and $\mu^k$ is an $\varepsilon_k$--strong Markovian MFG equilibrium.
\end{proposition}

\begin{proof}
       Let $X'^k$ be the solution of 
    \begin{align*}
        \mathrm{d}X'^{k}_t
        =
        b(t,X'^{k}_t,\mub^k_t,\beta'^k(t,X'^{k}_t,\mu^k_t)) \mathrm{d}t
        +
        \sigma(t,X'^k_t) \mathrm{d}W_t
        +
        \sigma_0 \mathrm{d}B_t\;\mbox{with}\;\mu'^k_t:=\Lc(X'^k_t|\Fc^B_t),\;\mub'^k_t:=\Lc(X'^k_t,\beta'^k(t,X'^{k}_t,\mu^k_t)|\Fc^B_t)
    \end{align*}
    where $(\beta^k)_{k \in \N^*}$ is a sequence s.t.
    \begin{align*}
	    &\E\big[J\big(\mu'^k,\mu^k,\delta_{\mub'^k_t}(\mathrm{d}\nub)\mathrm{d}t, \delta_{\mub^k_t}(\mathrm{d}\nub)\mathrm{d}t\big)\big]
	    -
	    \E\big[J\big(\mu^k,\mu^k,\delta_{\mub^k_t}(\mathrm{d}\nub)\mathrm{d}t, \delta_{\mub^k_t}(\mathrm{d}\nub)\mathrm{d}t\big)\big]
	    \\
	    &\ge
	    \sup_{\alpha \in \Ac^c}
	    \E\big[J\big(\mu^\alpha,\mu^k,\delta_{\mub^\alpha_t}(\mathrm{d}\nub)\mathrm{d}t, \delta_{\mub^k_t}(\mathrm{d}\nub)\mathrm{d}t\big)\big]
	    -
	    \E\big[J\big(\mu^k,\mu^k,\delta_{\mub^k_t}(\mathrm{d}\nub)\mathrm{d}t, \delta_{\mub^k_t}(\mathrm{d}\nub)\mathrm{d}t\big)\big] - 1/2^k.
	\end{align*}
	 Given $\mub,$ $\mu^k$ and $X'^k,$ let $\Xt'^k$ be the solution of 
    \begin{align*}
        \mathrm{d}\Xt'^{k}_t
        =
        b(t,\Xt'^{k}_t,\mub_t,\beta'^k(t,X'^{k}_t,\mu^k_t)) \mathrm{d}t
        +
        \sigma(t,\Xt'^k_t) \mathrm{d}W_t
        +
        \sigma_0 \mathrm{d}B_t\;\mbox{with}\;\widetilde{\mu}'^k_t:=\Lc(\Xt'^k_t|\Fc^B_t),\;\widetilde{\mub}'^k_t:=\Lc(\Xt'^k_t,\beta'^k(t,X'^{k}_t,\mu^k_t)|\Fc^B_t).
    \end{align*}
    There is a constant $C$ independent of $k$ s.t. for each $t \in [0,T],$
    \begin{align*}
        \E\Big[\sup_{s \in [0,t]}|\Xt'^k_s-X'^k_s|^2 \Big]
        \le C \;\;
        \E\bigg[ \bigg| \int_0^t b^\star(s,\mub^k_s) \mathrm{d}s - \int_0^t b^\star(s,\mub_s) \mathrm{d}s  \bigg|^2 + \int_0^t \sup_{r \in [0,s]} |\Xt'^k_r-X'^k_r|^2 \mathrm{d}r \bigg],
    \end{align*}
    by applying Gr\"{o}nwall's lemma and using {\rm \Cref{prop:open-closed-approx}}, it is easy to check that $\Lim_{k \to \infty}\E\Big[\sup_{t \in [0,T]}|\Xt'^k_t-X'^k_t| \Big]=0.$ This leads to 
    $$
        \Lim_{k \to \infty} \Wc_p \big( (\mu'^k,\delta_{\mub'^k_t}(\mathrm{d}\nub)\mathrm{d}t) , (\widetilde{\mu}'^k,\delta_{\widetilde{\mub}'^k_t}(\mathrm{d}\nub)\mathrm{d}t) \big)=0,\;\P\mbox{--a.e.}.
    $$
    Then, as $X$ is an $\varepsilon$--$open$--$loop$ MFG equilibrium,
    \begin{align*}
        &\Lim_{k \to \infty} 
        \E\big[J\big(\mu'^k,\mu^k,\delta_{\mub'^k_t}(\mathrm{d}\nub)\mathrm{d}t, \delta_{\mub^k_t}(\mathrm{d}\nub)\mathrm{d}t\big)\big]
        -
	    \E\big[J\big(\mu^k,\mu^k,\delta_{\mub^k_t}(\mathrm{d}\nub)\mathrm{d}t, \delta_{\mub^k_t}(\mathrm{d}\nub)\mathrm{d}t\big)\big]
	    \\
	    &\Lim_{k \to \infty} 
	    \E\big[J\big(\widetilde{\mu}'^k,\mu,\delta_{\widetilde{\mub}'^k_t}(\mathrm{d}\nub)\mathrm{d}t, \delta_{\mub_t}(\mathrm{d}\nub)\mathrm{d}t\big)\big]
	    -
	    \E\big[J\big(\mu,\mu,\delta_{\mub_t}(\mathrm{d}\nub)\mathrm{d}t, \delta_{\mub_t}(\mathrm{d}\nub)\mathrm{d}t\big)\big]
	    \le \varepsilon.
    \end{align*}
    To conclude, it is enough to define $(\varepsilon_k)_{k \in \N^*}$ by 
    \begin{align*}
        \varepsilon_k
        :=
        \sup_{\alpha \in \Ac^c}
	    \E\big[J\big(\mu^\alpha,\mu^k,\delta_{\mub^\alpha_t}(\mathrm{d}\nub)\mathrm{d}t, \delta_{\mub^k_t}(\mathrm{d}\nub)\mathrm{d}t\big)\big]
	    -
	    \E\big[J\big(\mu^k,\mu^k,\delta_{\mub^k_t}(\mathrm{d}\nub)\mathrm{d}t, \delta_{\mub^k_t}(\mathrm{d}\nub)\mathrm{d}t\big)\big].
    \end{align*}
\end{proof}

\paragraph*{Proof of \Cref{thm:converse-limit} (measure--valued MFG via approximate strong markovian MFG)}\label{proof_thmconv_1_MFG}

Let $\Pr \in \Pcb_V$ such that $\Pr \circ (\mu,\Lambda)^{-1} \in \Sc^\star.$ By \cite[Theorem 2.13]{MFD-2020_MFG}, there exists a sequence of $(\sigma\{X_0, W_{t \wedge \cdot}, B_{t \wedge \cdot}\})_{t \in [0,T]}$--predictable processes $(\alpha^\ell)_{\ell \in \N^*}$ such that $X^{\alpha^\ell}$ the McKean--Vlasov process solution of \Cref{eq:MKV_strong_op-MFG} associated to $\alpha^\ell$ is an $\varepsilon^\ell$--$open$--$loop$ MFG equilibrium with $\varepsilon_\ell>0,$  $\Lim_{\ell \to \infty} \varepsilon_\ell=0,$ and
\begin{align*}
    \Lim_{\ell \to \infty} \P \circ \big(\mu^\ell,\mu^\ell,\delta_{\mub^\ell_t}(\mathrm{d}\nub)\mathrm{d}t, \delta_{\mub^\ell_t}(\mathrm{d}\nub)\mathrm{d}t\big)^{-1}
    =
    \Pr\;\mbox{in}\;\Wc_p,\;\mbox{where}\;\mu^\ell_t:=\Lc(X^{\alpha^\ell}_t|\Fc^B_t)\;\mbox{and}\;\mub^\ell_t=\Lc(X^{\alpha^\ell}_t, \alpha^\ell_t|\Fc^B_t).
\end{align*}

To conclude, it is enough to approximate $open$--$loop$ MFG equilibrium by sequence of approximate strong markovian MFG equilibrium. This is done by \Cref{prop:open_closed}.

\subsubsection{Approximation of approximate strong Markovian MFG equilibrium by approximate $closed$--$loop$ Nash equilibria}
Now, in this part, we show that any approximate strong Markovian MFG equilibrium is the limit of sequence of approximate $closed$--$loop$ Nash equilibria.

\medskip
Let $\mub$ be an $\varepsilon$--strong markovian MFG equilibrium associated to control $\beta$ where $\beta:[0,T] \x \R^n \x \Cc^{n,p}_{\Wc} \ni (t,x,\pi) \to \beta(t,x,\pi_{t \wedge \cdot}) \in U$ is Lipschitz in $(x,\pi)$ uniformly in $t.$ We denote by $X$ the McKean--Vlasov process associated to $\mub,$ 
\begin{align*}
        \mathrm{d}X_t
        =
        b(t,X_t,\mub_t,\beta(t,X_t,\mu)) \mathrm{d}t
        +
        \sigma(t,X_t) \mathrm{d}W_t
        +
        \sigma_0 \mathrm{d}B_t\;\;\mbox{with}\;\;\mu_t:=\Lc(X_t|\Fc^B_t),\;\mub_t:=\Lc\big(X_{t},\beta(t,X_t,\mu) | \Fc^B_t\big).
    \end{align*}

For each $N \in \N^*,$ let us define 
\begin{align} \label{eq:def_beta}
    \beta^{i,N}(t,\xb):=\beta(t,\xf^i(t),\pi[\xb])
\end{align}
for $\xb:=(\xf^1,\cdots,\xf^N) \in (\Cc^n)^N$ with $\pi_t[\xb]:=\frac{1}{N} \sum_{i=1}^N \delta_{\xf^i(t)}.$ We denote $\alphab^N:=(\beta^{1,N},\cdots,\beta^{N,N}).$ 
\begin{proposition} \label{prop:chaos_common}
    Under {\rm \Cref{assum:main1}}, one has
    \begin{align*}
        \Lim_{N \to \infty} \varphi^N[\alphab^N]
        =
        \mu\;\mbox{in}\;\Wc_p,\;\P\mbox{--a.e.}.
    \end{align*}
\end{proposition}

\begin{proof}
       We define
       \begin{align*}
           b(t,x,\pi)
           :=
           b(t,x,\delta_{\beta(t,y,\pi)}(\mathrm{d}u)\pi_t(\mathrm{d}y), \beta(t,x,\pi))\;\mbox{for each}\;(t,x,\pi) \in [0,T] \x \R^n \x \Cc^n_{\Wc}.
       \end{align*}
       As $\beta$ is Lipschitz in $(x,\pi)$ uniformly in $t,$ under \Cref{assum:main1}, we can apply \cite[Proposition 4.15]{djete2019general} (a classical propagation of chaos with common noise, see also \cite{KURTZ1999103}) and find the desired result. 
\end{proof}
Recall that $(\beta^{1,N},\cdots,\beta^{N,N})$ is defined in \Cref{eq:def_beta}.
\begin{proposition}\label{prop:closed_Nash}
    There exists a sequence $(\varepsilon_{N})_{N^*} \subset (0,\infty)$ s.t. $\Limsup_{N \to \infty} \varepsilon_{N} \in [0,\varepsilon]$ and $\alphab^N:=(\beta^{1,N},\cdots,\beta^{N,N})$ is an $\varepsilon_N$--$closed$--$loop$ Nash equilibria for each $N \in \N^*$.
\end{proposition}

\begin{proof}

    Let us define
    \begin{align*}
        c^{i,N}
        :=
        \sup_{\alpha' \in \Ac_N^c}
			J_i\big(( (\alphab^{N})^{-i},\alpha')\big)
			-
			J_i(\alphab^{N}).
    \end{align*}
    There exists a sequence of controls $(\kappa^{i,N})_{(i,N) \in \{1,\dots,N\} \x \N^*}$ satisfying $J_i\big(( (\alphab^{N})^{-i},\kappa^{i,N})\big)
	-J_i(\alphab^{N})\ge c^{i,N}-2^{-N},$ for each $i \in \{1,...,N\}.$
    We define 
$$
    \alphab^{i,N}:=(\beta^{1,N},\dots,\beta^{,i-1,N},\kappa^{i,N},\beta^{i+1,N},\dots,\beta^{N,N}).
$$
By using the same technique as in the proof of \Cref{lemm:convergenceNashEquilibrium}, one has for each $i \in \{1,\dots,N\},$
\begin{align} \label{eq:chaos_strongCN}
    \Lim_{N \to \infty} \E\big[\Wc_p(\varphi^N[\alphab^{i,N}],\mu) \big].
\end{align}
Indeed, let $\Ybb^i:=(Y^{i,1},\cdots,Y^{i,N}):=(X^{\alphab^{i,N},1},\dots,X^{\alphab^{i,N},N}),$ 
   and we introduce
	\begin{align*}
	    Z^i_t
	    :=
	    \exp\bigg\{ \int_0^t \phi^i_r \mathrm{d}W^i_r - \frac{1}{2} \int_0^t |\phi^i_r|^2 \mathrm{d}r \bigg\}\;\mbox{for all}\;t \in [0,T],\; \mbox{and}\;\frac{\mathrm{d}\Q^i}{\mathrm{d}\P}=Z^i_T
	\end{align*}
	with 
	\begin{align*}
	    \phi^i_t:=\sigma(t,X^{\alphab^N,i}_t)^{-1}\Big(b\big(t,X^{\alphab^N,i}_{t},\overline{\varphi}^{N}_{t}[\alphab^{i,N}] ,\kappa^{i,N}(t,\Xbb^{\alpha}) \big) - b\big(t,X^{\alphab^N,i}_{t},\overline{\varphi}^{N}_{t}[\alphab^N] ,\alpha^{\star,i,N}(t,\Xbb^{\alphab^N}) \big) \Big).
	\end{align*}
	 By uniqueness, we can check that $\Lc^{\Q^i} (\Xbb^{\alphab^N},B)=\Lc^{\P}(\Ybb^i,B)$ for each $i.$ Then,
	 \begin{align*}
	     \Lim_{N \to \infty} \E^{\P}\big[\Wc_p(\varphi^N[\alphab^{i,N}],\mu) \big]
	     =
	     \Lim_{N \to \infty} \E^{\P}\big[Z^i_T\Wc_p(\varphi^N[\alphab^{N}],\mu) \big]
	     \le 
	     \Lim_{N \to \infty}
	     \E^{\P}\big[|Z^i_T|^2 \big]^{1/2}\E^{\P}\big[\Wc_p(\varphi^N[\alphab^{N}],\mu)^2 \big]^{1/2}=0
	 \end{align*}
    where we used \Cref{prop:chaos_common}. A consequence of \eqref{eq:chaos_strongCN} is $\Lim_{N \to \infty} \E\big[\Wc_p(\overline{\varphi}^N_t[\alphab^{i,N}],\mub_t) \big]$ $\mathrm{d}\P \otimes \mathrm{d}t$--a.e. Next, let $\widetilde{Y}^{i,N}$ be the solution of
    \begin{align*}
        \mathrm{d}\Yt^{i,N}_t
        =
        b(t,\Yt^{i,N}_t,\mub_t,\kappa^{i,N}(t,\Ybb^i)) \mathrm{d}t
        +
        \sigma(t,\Yt^{i,N}_t) \mathrm{d}W_t
        +
        \sigma_0 \mathrm{d}B_t.
    \end{align*}
    With the previous results, one has $\Lim_{N \to \infty} \E^{\P}\bigg[\sup_{t \in [0,T]} |\Yt^{i,N}_t-Y^{i,N}_t| \bigg]=0.$

Therefore, using all previous observations and as $\mub$ is  an $\varepsilon$--strong markovian MFG equilibrium, so can be seen as an $\varepsilon$--$open$--$loop$ MFG equilibrium,
    \begin{align*}
        \varepsilon
        &\ge
        \Big( \Limsup_{N \to \infty} 
        \Big(\frac{1}{N} \sum_{i=1}^{N} \E \bigg[
        \int_0^T L\big(t,\Yt^{i,N}_t,\mub_t ,\kappa^i(t,\Ybb^{i}) \big) \mathrm{d}t 
        + 
        g \big(\Yt^{i,N}_T, \mu \big)
        \bigg]
        -
        \E\big[J\big(\mu,\mu,\delta_{\mub_t}(\mathrm{d}\nub)\mathrm{d}t, \delta_{\mub_t}(\mathrm{d}\nub)\mathrm{d}t\big)\big]
         \Big)
        \\
        &\ge 
        \Limsup_{N \to \infty} 
        \Big(\frac{1}{N} \sum_{i=1}^{N} J_i\big(( (\alphab^{N})^{-i},\kappa^{i,N})\big)
        -
        \frac{1}{N} \sum_{i=1}^{N} J_i(\alphab^{N}) \Big)
        \ge
        \Limsup_{N \to \infty}
        \frac{1}{N} \sum_{i=1}^{N} c^{i,N}.
    \end{align*}
    To obtain the result as formulated, observe that by symmetry $\varepsilon_N:=c^{i,N}=c^{1,N}.$ This is enough to conclude.

\end{proof}

\paragraph*{Proof of \Cref{thm:converse-limit} (measure--valued MFG via approximate Nash equilibria)}\label{proof_thmconv_2_MFG}

Let $\Pr \in \Pcb_V$ such that $\Pr \circ (\mu,\Lambda)^{-1} \in \Sc^\star.$ By \Cref{thm:converse-limit} (first part) (see also \Cref{prop:open_closed}), there exist $(\varepsilon_k)_{k \in \N^*} \subset (0,\infty)$ with $\Lim_{k \to \infty} \varepsilon_k=0$ and a sequence $(\beta^k)_{k \in \N^*}$ s.t. $\mub^k$ is an $\varepsilon_k$--strong Markovian MFG equilibrium associated to the control $\beta^k$ with $\beta^k$ Lipschitz in $(x,\pi)$ uniformly in $t,$ and 
\begin{align*}
    \Lim_{k \to \infty} \P \circ \big(\mu^k,\mu^k,\delta_{\mub^k_t}(\mathrm{d}\nub)\mathrm{d}t, \delta_{\mub^k_t}(\mathrm{d}\nub)\mathrm{d}t\big)^{-1}
    =
    \Pr\;\mbox{in}\;\Wc_p.
\end{align*}
To conclude, it is enough to approximate approximate strong Markovian MFG equilibrium by approximate Nash equilibria. This is done by \Cref{prop:closed_Nash}.

\subsection{McKean--Vlasov optimal control} \label{sec:proof_Mc-Vl}

\subsubsection{Proof of \Cref{thm:equivalence_mckv_cl}}\label{proof_thm1_MFC}

Thanks to \Cref{prop:open_closed}, we can approximate any $open$--$loop$ McKean--Vlasov process by a sequence of $closed$--$loop$ McKean--Vlasov processes. Also, by \Cref{prop:canonical_McKV}, we can easily verify that $V^c_S \le V^o_S.$ We therefore find our result and $V^c_S=V^o_S.$ We deduce that
\begin{align*}
        V^c_S
        =
        V^{o}_S
        =
        \sup_{\Pr \in \Vc} \E^{\Pr} \big[J (\mu',\mu,\Lambda,\Lambda') \big].
    \end{align*}

\subsubsection{Proof of \Cref{thm:limit_mckv_cl}}\label{proof_thm2_MFC}

Let $(\Prt^N)_{N \in \N^*} \subset \Pc(\Cc^n_{\Wc} \x \M \x \Cc^n)$ be the sequence defined by
\begin{align*}
    \Prt^N
    :=
    \P \circ \big( \varphi^N[\alphab^N], \delta_{\overline{\varphi}^N_t[\alphab^N]}(\mathrm{d}\nub)\mathrm{d}t, B\big)^{-1}.
\end{align*}
If we denote by $(\mut,\widetilde{\Lambda},\Bt)$ the canonical process of $\Cc^n_{\Wc} \x \M \x \Cc^n.$ Notice that $\Pr^N=\Prt^N \circ (\mut,\Lambda)^{-1},$ where $\Pr^N$ is the sequence given in \Cref{thm:limit_mckv_cl}. By \cite[Proposition 3.4]{MFD-2020},  $(\Prt^N)_{N \in \N^*}$ is relatively compact in $\Wc_p$, and for each limit point $\Prt^\infty,$ there exists a sequence of $(\sigma\{\xi,W_{t \wedge \cdot}, B_{t \wedge \cdot}\})_{t \in [0,T]}$--predictable processes $(\alpha^k)_{k \in \N^*}$ s.t. if $X^k$ is the solution of
\begin{align*} 
		\mathrm{d}X^k_t
		= 
		b \big(t, X^k_t, \mub^k_t, \alpha^k_t \big) \mathrm{d}t
		+
		\sigma(t, X^k_t ) \mathrm{d} W_t
		+
		\sigma_0 \mathrm{d}B_t\;\;\;\mbox{with}\;X_0=\xi,\;\mu^k_t:=\Lc(X^k_t|\Fc^B_t)\;\mbox{and}\;\mub^k_t:=\Lc(X^k_t, \alpha^k_t|\Fc^B_t),
	\end{align*}
	then
	\begin{align*}
	    \Prt^\infty
	    =
	    \Lim_{k \to \infty} \P \circ (\mu^k,\delta_{\mub^k_t}(\mathrm{d}\nub)\mathrm{d}t, B)^{-1}\;\;\mbox{in}\;\Wc_p.
	\end{align*}
	We see that, for each $k \in \N^*,$ $X^k$ is an $open$--$loop$ McKean--Vlasov process associated to $\alpha^k.$ By \Cref{thm:equivalence_mckv_cl} (see also \Cref{prop:open_closed}), we can approximate this $open$--$loop$ McKean--Vlasov process by a sequence of $closed$--$loop$ McKean--Vlasov processes. To conclude the proof, it is enough to see that by \Cref{prop:chaos_common}, we can approach any $closed$--$loop$ McKean--Vlasov process by a sequence of interacting processes. \Cref{prop:chaos_common} we give exactly the sequence that we want for our \Cref{thm:limit_mckv_cl}.

\medskip
For proving $V_S^{o}=V_S^{c} =
        \Lim_{N \to \infty}\; \sup_{\alphab \in (\Ac^c_N)^N}\; \frac{1}{N} \sum_{1=1}^N J_i ( \alphab ),$ first of all, by the previous result, we observe that we have 
$$
    \Limsup_{N \to \infty}\sup_{\alphab \in (\Ac^c_N)^N} \frac{1}{N}\sum_{i=1}^N
        \E^{\P} \bigg[
        \int_0^T L\big(t,X^{\alphab,i}_{t},\overline{\varphi}^{N}_{t}[\alphab] ,\alpha^i(t,\Xbb^{\alphab}) \big) \mathrm{d}t 
        + 
        g \big( X^{\alphab,i}_{T}, \varphi^{N}[\alphab] \big)
        \bigg]
    \le 
    V^o_S.
$$

By \Cref{prop:chaos_common}, we can approximate any $closed$--$loop$ McKean--Vlasov process by a sequence of interacting processes.  Consequently, one gets 
$$
    V^c_S
    \le
    \Liminf_{N \to \infty}\sup_{\alphab \in (\Ac^c_N)^N} \frac{1}{N}\sum_{i=1}^N
        \E^{\P} \bigg[
        \int_0^T L\big(t,X^{\alphab,i}_{t},\overline{\varphi}^{N}_{t}[\alphab] ,\alpha^i(t,\Xbb^{\alphab}) \big) \mathrm{d}t 
        + 
        g \big( X^{\alphab,i}_{T}, \varphi^{N}[\alphab] \big)
        \bigg].
$$
As by \Cref{thm:equivalence_mckv_cl}, $V^o_S=V^c_S,$ we can conclude.





\bibliographystyle{plain}

\bibliography{Extended_MFG-C-loop_Arxiv_08-23}

\begin{appendix}
\section{Technical results}

\subsection{Some compactness results}

This first part is devoted to providing results of compactness and convergence related to the Fokker--Planck equation. We start by showing estimates for the density of the Fokker--Planck equation. A key tool here is \cite[Theorem 4]{AronsonSerrin67}.

\medskip
    Let $(\Phi^k)_{k \in \in \N}$ be a sequence of Borel functions s.t. for each $ k\in \N^*,$ $\Phi^k: [0,T] \x \R^n \x C([0,T];\Pc(\R^n)) \ni (t,x,\pi) \to \Phi^k(t,x,\pi_{t \wedge \cdot}) \in \R^n$. Besides, there is $c>0$ satisfying: for each $k \in \N,$
    \begin{align}
        \label{assum:general_drift}
       |\Phi^k| \le c\;\;\mbox{and}\;\; \frac{1}{c}\;|\Phi^k(t,x,\pi^1)
        -
        \Phi^k(t,x,\pi^2)|
        \le\;\sup_{s \in [0,T]}\; \|\pi^1_s - \pi^2_s \|_{{\rm TV}},\;\mbox{for all}\;(t,x),
    \end{align}
    where $\|\pi^1_s - \pi^2_s \|_{{\rm TV}}:=\sup_{f \in C(\R^n;[-1,1])} |\langle f,\pi^1_s \rangle-\langle f,\pi^2_s \rangle|$ is the total variation distance.
    The map $\sigma$ always satisfies {\rm \Cref{assum:main1}}. For each $k \in \N,$ let $X^k$ be the weak solution of
    \begin{align*}
        \mathrm{d}X^k_t
        =
        \Phi^k(t,X^k_t,\mu^k) \mathrm{d}t 
        +
        \sigma(t,X^k_t)\mathrm{d}W_t\;\mbox{with}\;X^k_0=\xi\;\mbox{and}\;\mu^k_t=\Lc(X^k_t).
    \end{align*}
We denote by $f^k(t,x)$ the density of $\mu^k_t$ i.e. $\mu^k_t(\mathrm{d}x)=f^k(t,x)\mathrm{d}x$ which is well--defined for $t \in (0,T)$ (see for instance \cite[Theorem 6.3.1, Corollary 6.3.2, Remark 6.3.4]{FK-PL-equations}, and aslo \cite{stroock2007multidimensional}).
\begin{proposition} \label{prop:estimates_FP-density}
    For each compact $[s,t] \x Q \subset (0,T) \x \R^n,$ there exists $\alpha \in (0,1)$ s.t.
    \begin{align*}
        \sup_{k \in \N^*}\bigg[\sup_{ (x,r)\in [s,t] \x Q} |f^k(r,x)|
        +
        \sup_{(x,r) \neq (x',r'),\;(x,r) \x (x',r')\in ([s,t] \x Q)^2 } \frac{|f^k(r,x)-f^k(r',x')|}{|r-r'|^{\alpha/2} + |x-x'|^\alpha} \bigg]
        < \infty.
    \end{align*}
\end{proposition} 

\begin{proof}
       Although the presentation looks different, this proof is largely based on the proof of \cite[Theorem 6.6.4]{FK-PL-equations}.
       For simplification, we display the proof for $n=1.$ Let $\rho$ be a density probability function, symmetric about $0$ and belongs to $C^\infty(\R^n)$ with compact support.  
       In a first time, we consider $k$ as fixed. Let $\delta>0$ and $\rho_{\delta}(x):=\delta^{-1}\rho(\delta^{-1} x),$ 
       \begin{align*}
           \Phi^k_{\delta}(t,x)
           :=
           \int_{\R} \Phi^k(t,y,\mu^k) \rho_{\delta}(x-y) \mathrm{d}y\;\;\mbox{and}\;\;
           a_{\delta}(t,x)
           :=
           \int_{\R} \sigma \sigma^\top(t,y) \rho_{\delta}(x-y) \mathrm{d}y.
       \end{align*}
       It is well known that $\Lim_{\delta \to 0} (\Phi^k_\delta,a_\delta)=(\Phi^k,\sigma \sigma^\top).$
       We introduce $X^{k,\delta}$ the solution of 
       \begin{align*}
           \mathrm{d}X^{k,\delta}_t
            = \Phi^k_{\delta}(t,X^{k,\delta}_t,\mu^k) \mathrm{d}t 
            + (a_\delta)^{1/2}(t,X^{k,\delta}_t)\mathrm{d}W_t.
       \end{align*}
       Using a martingale problem or \cite[Theorem 11.1.4]{stroock2007multidimensional}, it is straightforward that $\Lim_{\delta \to 0}\Lc(X^{k,\delta})=\Lc(X^k)$ in $\Wc_p.$ 
       If we note $\mu^{k,\delta}_t:=\Lc(X^{k,\delta}_t),$ By applying It\^{o} formula and taking the expectation, for each $\varphi \in C_b^2(\R),$
       \begin{align*}
           \mathrm{d} \langle \varphi, \mu^{k,\delta}_t \rangle
           =
           \langle \varphi'(\cdot)\Phi^k_{\delta}(t,\cdot), \mu^{k,\delta}_t \rangle \mathrm{d}t
           +
           \frac{1}{2}\langle \varphi''(\cdot)a_{\delta}(t,\cdot), \mu^{k,\delta}_t \rangle \mathrm{d}t.
       \end{align*}
       
       By \cite[Theorem 6.6.1]{FK-PL-equations}, for each $\delta>0,$ we know that there is $f^k_\delta$ such that $\mu^{k,\delta}_t(\mathrm{d}x)=f^k_\delta(t,x)\mathrm{d}x$ for $t \in (0,T)$ where $f^k_\delta(t,\cdot) \in C^2_b(\R).$ Then, in a weak sense, we have 
       \begin{align*}
           \partial_t f^k_\delta(t,\cdot)
           =
           \partial_x \Big( a^k_\delta(t,\cdot)\partial_x f^k_\delta(t,\cdot) - \Phi^k_\delta(t,\cdot)f^k_\delta(t,\cdot)
           +
           \frac{1}{2} \partial_x a_\delta(t,\cdot)f^k_\delta(t,\cdot)
           \Big).
       \end{align*}
       By \cite[Theorem 4]{AronsonSerrin67} (see also \cite[Theorem 6.2.7]{FK-PL-equations}), for each compact $[s,t] \x Q \subset (0,T) \x \R^n,$ there exists $C>0$ and $\alpha \in (0,1)$ depending only on $(\|\Phi^k_\delta\|_{\ell,q},$ $\|a_\delta\|_{\ell,q}$,  $\|\partial_x a_\delta\|_{\ell,q})$ and $[s,t] \x Q$ where
       \begin{align*}
            \|\psi\|_{\ell,q}
            :=
            \bigg(\int_s^t \bigg( \int_Q |\psi(t,x)|^\ell \mathrm{d}x \bigg)^{q/\ell} \mathrm{d}t \bigg)^{1/q}\;\;\mbox{with}\;\ell>2,\;\frac{1}{2 \ell} + \frac{1}{q} < 1/2
       \end{align*}
       s.t. $\mbox{for}\;(r,r',x,x') \in [s,t] \x [s,t] \x Q \x Q,$
       \begin{align} \label{eq:holder_estimates}
           |f^k_\delta(r,x)-f^k_\delta(r',x')| 
           \le
           C \Big( |r-r'|^{\alpha/2} + |x-x'|^{\alpha} \Big)\;\mbox{and}\;\sup_{(r,x) \in [s,t] \x Q}|f^k_\delta(r,x)| \le C. 
       \end{align}
       
       Notice that $\Phi^k$ is bounded uniformly in $k.$ Besides, $\sigma \sigma^\top$ is uniformly bounded and Lipschitz, so by Rademacher's theorem $\sigma \sigma^\top$ admits a weak derivative. then, $\partial_x a_\delta$ converges towards the weak derivative of $\sigma \sigma^\top$ when $\delta \to 0.$ Consequently,
        $$
            \sup_{(k,\delta) \in \N^* \x (0,\infty)}\|\Phi^k_\delta\|_{\ell,q} + \|a_\delta\|_{\ell,q} +
            \|\partial_x a_\delta\|_{\ell,q}< \infty.
        $$
       Consequently, we can take $C$ and $\alpha$ independent of $k$ and $\delta.$
       By Arzelà–-Ascoli Theorem, for each $M \in \N^*,$ there exists $u^{k,M} \in C([0,T] \x \R^n)$ and a sub--sequence $(\delta^M_l)_{l \in \N^*}$ such that
        \begin{align*}
            \Lim_{l \to \infty}\sup_{(r,x) \in [T/M,T(1-1/M)] \x B(M)} \big|f^{k}_{\delta^M_l}(r,x)-u^{k,M}(r,x) \big|=0.
        \end{align*}
        where $B(M):=\{x \in \R^n: |x| \le M\}.$ By Cantor's diagonal argument, we can find $u^k \in C([0,T] \x \R^n)$ and a sub--sequence $(\delta_l)_{l \in \N^*}$ such that $f^k_{\delta_l}$ converges uniformly to $u^k$ on each compact set of type $[s,t] \x \subset (0,T) \x \R^n.$ As we know that $\Lim_{\delta \to 0} f^k_{\delta}(t,x)=f^k(t,x),$  we can deduce that $u^k=f^k$. Next, as $C$ and $\alpha$ are independent of $k$ and $\delta,$ by passing to the limit over $\delta$ in \Cref{eq:holder_estimates}, we find our result.

\end{proof}
    
    Now, we will check that when the sequence of functions $(\Phi^k)_{k \in \N}$ satisfies a condition (see below), the sequence of densities $(f^k)_{k \in \N}$ converges to an identifiable limit.

\medskip    
    Let $(\pi^k)_{k \in \N} \subset \Cc^n_{\Wc}$ be a sequence s.t. $\Lim_{k \to \infty}\|\pi^k_t-\pi^0_t\|_{\rm TV}=0,$ for any $t \in (0,T).$ For any sequence of this type, we assume that:
    \begin{align} \label{cond:convergence}
        \Lim_{k \ge 1,\;k \to \infty}\int_0^T \int_{\R^n} \varphi(t,x) \Phi^k(t,x,\pi^k) \pi^k_t(\mathrm{d}x) \mathrm{d}t
        =
        \int_0^T \int_{\R^n} \varphi(t,x) \Phi^0(t,x,\pi^0) \pi^0_t(\mathrm{d}x) \mathrm{d}t,
    \end{align}
    for any continuous function $\varphi$ with compact support.
    
    \begin{corollary} \label{cor:converg_density}
        With the previous considerations, one has for each compact $[s,t] \x Q \subset (0,T) \x \R^n,$
        \begin{align*}
            \Lim_{k \ge 1,\;k \to \infty}
             \sup_{(x,r) \x (x',r')\in ([s,t] \x Q)^2 } |f^k(r,x)-f^0(r',x')|
            =
            0.
        \end{align*}
    \end{corollary}
    
    \begin{proof}
            By \Cref{prop:estimates_FP-density} and  by Arzelà–Ascoli Theorem, for each $M \in \N^*,$ there exists $u^M \in C([0,T] \x \R^n)$ and a sub--sequence $(k^M_l)_{l \in \N^*}$ such that
        \begin{align*}
            \Lim_{l \to \infty}\sup_{(r,x) \in [T/M,T(1-1/M)] \x B(M)} \big|f^{k^M_l}(r,x)-u^M(r,x) \big|=0.
        \end{align*}
        By Cantor's diagonal argument, we can find $u \in C([0,T] \x \R^n)$ and a sub--sequence $(k_l)_{l \in \N^*}$ such that $f^{k_l}$ converges uniformly to $u$ on each compact set of type $[s,t] \x \subset (0,T) \x \R^n.$ Now, let us show that $u=f^0$. Notice that, for each $k \in \N^*,,$ $\mu^k$ (or $f^k$) satisfies the Fokker--Planck equation
        \begin{align*}
            \mathrm{d}\langle \varphi,\mu^k_t\rangle
            =
            \langle \varphi'(\cdot) \Phi^k(t,\cdot,\mu^k),\mu^k_t \rangle \mathrm{d}t + \frac{1}{2}\langle {\rm Tr} \big[ \varphi''(\cdot) \sigma\sigma^\top(t,\cdot) \big],\mu^k_t \rangle \mathrm{d}t.
        \end{align*}
        Let us define $\nu_0:=\mu^k_0=\Lc(\xi),$ and $\nu_t(\mathrm{d}y):=u(t,y)\mathrm{d}y$ for $t \in (0,T].$ As $f^{k_l}$ converges uniformly to $u$ on each compact set of type $[s,t] \x Q \subset (0,T) \x \R^n,$ it is easy to check that $\Lim_{k \to \infty}\|\mu^k_t-\nu_t\|_{\rm TV}=0,$ for any $t \in (0,T).$ Using our assumptions and passing to the limit in the Fokker--Planck equation, one finds that
        \begin{align*}
            \mathrm{d}\langle \varphi,\nu_t\rangle
            =
            \langle \varphi'(\cdot) \Phi^0(t,\cdot,\nu),\nu_t \rangle \mathrm{d}t + \frac{1}{2}\langle {\rm Tr} \big[ \varphi''(\cdot) \sigma\sigma^\top(t,\cdot) \big],\nu_t \rangle \mathrm{d}t.
        \end{align*}
        As $\Phi^0$ satisfies \eqref{assum:general_drift}, this Fokker--Planck equation has a unique solution (see for instance \cite[Theorem 2.3]{lacker-strong-PC-2018}). Therefore $\nu=\mu^0.$ This is true for any sub--sequence, then the all sequence converges. This is enough to conclude.
    \end{proof}

In the next part, {\color{black} in a more general framework than that considered in the article, we give an approximation of the controls}. The result proved here is quite useful for the approximation through a sequence of $closed$--$loop$ controls or Markovian controls.

\medskip    
    We say that $(m,\widehat{q}) \in C([0,T];\Pc(\R^n)) \x \M((\Pc^n_U)^2)$ satisfies a generalized Fokker--Planck equation if: $m_0:=\Lc(\xi),$ and 
    \begin{align} \label{eq:FK-PL-closed-loop}
	    \mathrm{d}\langle f,m_t \rangle
	    ~=~
	    \int_{(\Pc^n_U)^2} \int_{\R^n \x U}  \Ac_t f(x,\nub^{x'}(\mathrm{d}u') m_t(\mathrm{d}x'),\nub^\star,u) \nub^x(\mathrm{d}u)m_t(\mathrm{d}x)\widehat{q}_t(\mathrm{d}\nub,\mathrm{d}\nub^\star)\mathrm{d}t,
	\end{align}
    for all $f\in C^{2}_b(\R^n),$ where $(\nub^x)_{x \in \R^n}$ is the disintegration of $\nub$ in $\R^n$ i.e $\nub(\mathrm{d}x,\mathrm{d}u)=\nub^x(\mathrm{d}u)\nub(\mathrm{d}x,U),$ and the generator $\Ac$ is defined by
    \begin{align}
        \Ac_t\varphi(x,\pi,\nub,\nub^\star,u) 
        &:= 
        \frac{1}{2}  \text{Tr}\big[ \sigma  \sigma^\top(t,x) \nabla^2 \varphi(x) 
        \big] 
        + \widehat{b}(t,x,\nub,\nub^\star,u)^\top \nabla \varphi(x),
    \end{align}
    with $\widehat{b}:[0,T] \x \R^n \x \Pc^n_U \x \Pc^n_U \x U \ni (t,x,\nub,\nub^\star,u) \to \widehat{b}(t,x,\nub,\nub^\star,u) \in \R^n$ a Borel map continuous in $(x,\nub,\nub^\star,u)$ and Lipschitz in $\nub$ uniformly in $(t,x,\nub^\star,u).$ Recall that $\sigma$ always satisfies {\rm \Cref{assum:main1}}. Notice that, under this assumption, given $\widehat{q},$ the process $(m_t)_{t \in [0,T]}$ is uniquely defined (\cite[Theorem 2.3]{lacker-strong-PC-2018}).
    For each density of probability $u$ on $\R^n,$ and $\widehat{\beta} \in \M((\Pc^n_U)^2),$ we define $\widehat{\beta} \in \M((\Pc^n_U)^2)$ by
    \begin{align*}
        \widehat{\beta}_t[u](\mathrm{d}\nub,\mathrm{d}\nub^\star)\mathrm{d}t
        :=
        \int_{\Pc^n_U} \delta_{\big(e^x(\mathrm{d}u')u(x)\mathrm{d}x \big)} (\mathrm{d}\nub) \widehat{\beta}_t(\mathrm{d}e,\mathrm{d}\nub^\star)\mathrm{d}t.
    \end{align*}
    
    Let $(m^k,\widehat{q}^k) \subset C([0,T];\Pc(\R^n)) \x \M((\Pc^n_U)^2)$ be a sequence satisfying: for each $k \in \N^*,$ $(m^k,\widehat{q}^k)$ satisfies a generalized Fokker--Planck equation. 
    Let $G$ be a density of probability continuous on $\R^n$ with $G >0.$ Let us assume that: 
    \begin{align} \label{eq:assum-weakly-lemma}
        \Lim_{k \to \infty}
        \widehat{q}^k_t[G](\mathrm{d}\nub,\mathrm{d}\nub^\star)\mathrm{d}t
        =
        \widehat{q}^\infty_t[G](\mathrm{d}\nub,\mathrm{d}\nub^\star)\mathrm{d}t,\;\mbox{in weakly sense,}
    \end{align}
    for some  $\widehat{q}^\infty \in \M((\Pc^n_U)^2).$ Then, we have the following result.
    \begin{proposition} \label{lemma:approx-control-closed-loop}
        Let $f^k(t,x)$ be the density of $m^k_t$ i.e. $m^k_t(\mathrm{d}x)=f^k(t,x)\mathrm{d}x,$   one has for each compact $[s,t] \x Q \subset (0,T) \x \R^n,$
    \begin{align*}
             \Lim_{\;k \to \infty}
             \sup_{(x,r) \x (x',r')\in ([s,t] \x Q)^2 } |f^k(r,x)-f^\infty(r',x')|
            =
            0
    \end{align*}
    $\mbox{and}$
    \begin{align*}
        \Lim_{k \to \infty}
        \widehat{q}^k_t[f^k(t,\cdot)](\mathrm{d}\nub,\mathrm{d}\nub^\star)\mathrm{d}t
        =
        \widehat{q}^\infty_t[f(t,\cdot)](\mathrm{d}\nub,\mathrm{d}\nub^\star)\mathrm{d}t\;\mbox{in the weak sense},
    \end{align*}
        where $f^\infty$ is the density of $m^\infty$ with $(m^\infty,\widehat{q}^\infty)$ satisfies a generalized Fokker--Planck equation.
    \end{proposition}
    
    \begin{proof}
        Using the proof of \Cref{cor:converg_density}, there exists a sub--sequence $(k_\ell)_{\ell \in \N^*}$ such that $f^{k_\ell}$ converges uniformly to $u$ on each compact set of type $[s,t] \x \subset (0,T) \x \R^n.$ Let $q \in \N^*,$ and $(h^1,\cdots,h^q) \subset C(\R^n \x U;\R)$ $q$--bounded continuous function with compact support
        \begin{align*}
            &\Lim_{\ell \to \infty} \int_0^T \int_{(\Pc^n_U)^2}  \prod_{e=1}^q  \int_{U \x \R^n} h^e(x,u) \nub^x(\mathrm{d}u) m^{k_\ell}_t(\mathrm{d}x) \widehat{q}^{k_\ell}_t(\mathrm{d}\nub,\mathrm{d}\nub^\star)\mathrm{d}t
            \\
            &=\Lim_{\ell \to \infty} \int_0^T \int_{(\Pc^n_U)^2} \prod_{e=1}^q  \int_{U \x \R^n} h^e(x,u) \frac{1}{G(x)}f^{k_\ell}(t,x) \nub^x(\mathrm{d}u)  G(x) \mathrm{d}x\;\widehat{q}^{k_\ell}_t(\mathrm{d}\nub,\mathrm{d}\nub^\star)\mathrm{d}t
            \\
            &=\Lim_{\ell \to \infty} \int_0^T \int_{(\Pc^n_U)^2} \prod_{e=1}^q  \int_{U \x \R^n} h^e(x,u) \frac{1}{G(x)}u(t,x) \nub^x(\mathrm{d}u)  G(x) \mathrm{d}x\;\widehat{q}^{k_\ell}_t(\mathrm{d}\nub,\mathrm{d}\nub^\star)\mathrm{d}t
            \\
            &= \int_0^T \int_{(\Pc^n_U)^2} \prod_{e=1}^q  \int_{U \x \R^n} h^e(x,u) u(t,x) \nub^x(\mathrm{d}u) \mathrm{d}x\;\widehat{q}^\infty_t(\mathrm{d}\nub,\mathrm{d}\nub^\star)\mathrm{d}t,
        \end{align*}
        where we use the uniform convergence of $f^{k_\ell}$ to $u$ on each compact set, and for the last equality, the fact that $(t,x) \to h^e(t,x)u(t,x)\frac{1}{G(x)}$ is continuous and bounded because $h^e$ has a compact support. By similar arguments to \cite[Proposition A.3]{djete2019general}, this result allows to say that $\Lim_{\ell \to \infty}
        \widehat{q}^{k_\ell}_t[f^{k_\ell}(t,\cdot)](\mathrm{d}\nub,\mathrm{d}\nub^\star)\mathrm{d}t
        =
        \widehat{q}^\infty_t[u(t,\cdot)](\mathrm{d}\nub,\mathrm{d}\nub^\star)\mathrm{d}t\;\mbox{in the weak sense}.$ Now, for any continuous function $\varphi$ with compact support,
        \begin{align*}
            &\Lim_{\ell \to \infty} \int_0^T \int_{(\Pc^n_U)^2} \int_{U \x \R^n} \varphi(t,x) \widehat{b}(t,x,\nub^{x'}(\mathrm{d}u') m^{k_\ell}_t(\mathrm{d}x'),\nub^\star,u) \nub^x(\mathrm{d}u) m^{k_\ell}_t(\mathrm{d}x) \widehat{q}^{k_\ell}_t(\mathrm{d}\nub,\mathrm{d}\nub^\star)\mathrm{d}t
            \\
            &= \int_0^T \int_{(\Pc^n_U)^2} \int_{U \x \R^n} \varphi(t,x) \widehat{b}(t,x,\nub^{x'}(\mathrm{d}u') u(t,x')\mathrm{d}x',\nub^\star,u) u(t,x) \nub^x(\mathrm{d}u) \mathrm{d}x\;\widehat{q}^\infty_t(\mathrm{d}\nub,\mathrm{d}\nub^\star)\mathrm{d}t,
        \end{align*}
         By applying \Cref{cor:converg_density}, one has that $\Lim_{\ell \to \infty} f^{k_\ell}=u$ on each compact and $(m^\infty,\widehat{q}^\infty)$ satisfies a generalized Fokker--Planck equation where $m^\infty_0:=\Lc(\xi)$ and $m^\infty_t(\mathrm{d}x):=u(t,x)\mathrm{d}x$ for $t \in (0,T].$ Given $\widehat{q}^\infty,$ $m^\infty$ is uniquely defined. Therefore, any convergent sub--sequence of $(f^k)_{k \in \N^*}$ converges towards $u.$ We can deduce the convergence of the all sequence and then our result.
        

    \end{proof}

\subsection{Markovian approximation of controlled Fokker--Planck equation}

\medskip
Let $q \in \M$ and $m_t:=\Lc(X_{t})$ with $X$ solution of
\begin{align} \label{eq:SDE_McK_feeback}
        \mathrm{d}X_t
        =
        \int_{(\Pc^n_U)^2} \int_U
        \widehat{b}(t,X_t,\nub^{x'}(\mathrm{d}u')m_t(\mathrm{d}x'),\nub^\star,u) \nub^{X_{t}}(\mathrm{d}u) q_t(\mathrm{d}\nub) q^\star_t(\mathrm{d}\nub^\star)\mathrm{d}t
        +
        \sigma(t,X_t) \mathrm{d}W_t,\;X_0=\xi\;\;\mbox{a.e.}
    \end{align}
    where the map $\widehat{b}$ is s.t. $\widehat{b}(t,x,\nub,\nub^\star,u)=b^0(t,\nub)+b^1(t,\nub^\star)+b^2(t,x,u)$ and also bounded and Lipschitz in all variables. Recall that {\rm \Cref{assum:main1}} is still satisfied for $\sigma$ and for $t \in (0,T)$, $m_t$ has a density $f(t,x)$ i.e. $m_t=f(t,x) \mathrm{d}x.$
Let us assume that: there exists a sequence of Borel measurable functions $(\beta^k)_{k \in \N^*}$ satisfying: for each $k \ge 1$, $\beta^k:[0,T] \x \R^n \x \M \to U$ and
\begin{align*}
    \Lim_{k \to \infty} \delta_{H^k_t}(\mathrm{d}\nub)\mathrm{d}t
    =
    q_t[G](\mathrm{d}\nub)\mathrm{d}t\;\;\mbox{where}\;\;H^k_t:=\delta_{\beta^k(t,x,q_{t \wedge \cdot})}(\mathrm{d}u)G(x)\mathrm{d}x,
\end{align*}
$G$ is a density of probability continuous on $\R^n$ with $G>0.$

\begin{proposition}{$($Deterministic case$)$} \label{prop:approximation_weak}
    If we let $X^k$ be the unique strong solution of:
    \begin{align*}
        \mathrm{d}X^{k}_t
        =
        \int_{\Pc^n_U}
        \widehat{b}(t,X^{k}_t,\overline{m}^k_t,\nub^\star,\beta^k(t,X^{k}_t,q)) q^\star_t(\mathrm{d}\nub^\star) \mathrm{d}t
        +
        \sigma(t,X^k_t) \mathrm{d}W_t
        +
        \sigma_0 \mathrm{d}B_t
    \end{align*}
    with $m^k_t:=\Lc(X^{k}_t)$ and $\overline{m}^k_t:=\Lc\big(X^{k}_{t},\beta^k(t,X^{k}_t,q)\big)$  then
    \begin{align*}
        \Lim_{k \to \infty} \big(m^k, \delta_{\overline{m}^k_t}(\mathrm{d}\nub) \mathrm{d}t \big)= \big(m, q_t(\mathrm{d}\nub)\mathrm{d}t \big),\;\mbox{for the Wasserstein metric}\;\Wc_p.
    \end{align*}
    
\end{proposition}

\begin{proof}
This is a direct application of \Cref{lemma:approx-control-closed-loop}. Indeed,  let us define $q^k_t(\mathrm{d}\nub)\mathrm{d}t:=\delta_{\overline{m}^k_t}(\mathrm{d}\nub)\mathrm{d}t.$ Then, one has $\Lim_{k \to \infty} q^k_t[G](\mathrm{d}\nub')\mathrm{d}t
    =
    \Lim_{k \to \infty}\delta_{H^k_t}(\mathrm{d}\nub)\mathrm{d}t
    =
    q_t[G](\mathrm{d}\nub')\mathrm{d}t.$ Therefore, we can apply \Cref{lemma:approx-control-closed-loop}, for $f^k(t,x)$ the density of $m^k_t,$ one has
    \begin{align*}
             \Lim_{k \to \infty} m^k=m\;\mbox{and}\;\Lim_{k \to \infty}
        q^k_t[f^k(t,\cdot)](\mathrm{d}\nub)\mathrm{d}t
        =
        q_t[f(t,\cdot)](\mathrm{d}\nub)\mathrm{d}t\;\mbox{in}\;\Wc_p.
    \end{align*}
\end{proof}

\medskip
Now, we provide an approximation result close to the previous one when the Fokker--Planck equation is stochastic. More precisely, let $(\Lambda^\star_t)_{t \in [0,T]}$ be a $\Pc^n_U$--valued $\F$--predictable process s.t. $W,$ $X_0$ and $(\Lambda^\star,B)$ are independent. We denote by $\G$ the natural filtration of $(\Lambda^\star,B)$ i.e. $\Gc_t:=\sigma \{\Lambda^\star_{t \wedge \cdot}, B_{t \wedge \cdot} \}.$  Let $\Lambda$ be $\G$--predictable process satisfying $\Lambda_t(\Z_{\mu_t})=1$ $\mathrm{d}\P \otimes \mathrm{d}t$--a.e. where $\mu_t:=\Lc(X_{t}|\Gc_t)$ with $X$ solution of
\begin{align} \label{eq:SDE_McK_feeback}
        \mathrm{d}X_t
        =
        \int_{(\Pc^n_U)^2} \int_U
        \widehat{b}(t,X_t,\nub,\nub^\star,u) \nub^{X_{t}}(\mathrm{d}u) \Lambda_t(\mathrm{d}\nub) \Lambda^\star_t(\mathrm{d}\nub^\star)\mathrm{d}t
        +
        \sigma(t,X_t) \mathrm{d}W_t
        +
        \sigma_0 \mathrm{d}B_t,\;X_0=\xi,\;\;\mbox{a.e.}
    \end{align}

\begin{proposition}{$($Stochastic case$)$} \label{prop:approximation_weak_2}
    There exists a sequence of functions $(\beta^k)_{k \in \N^*}$ satisfying: for each $k \in \N^*,$ $[0,T] \x \R^n \x \Cc^n \x \M \ni (t,x,b,q) \mapsto \beta^k(t,x,b,q) \in U$ is progressively measurable i.e. $\beta^k(t,x,b,v)=\beta^k(t,x,b_{t \wedge \cdot},v_{t \wedge \cdot})$, Lipschitz in $(x,b,q)$ uniformly in $t$, and if we let $X^k$ be the unique strong solution of:
    \begin{align*}
        \mathrm{d}X^{k}_t
        =
        \int_{\Pc^n_U}
        \widehat{b}(t,X^{k}_t,\mub^k_t,\nub^\star,\beta^k(t,X^{k}_t,B,\Lambda^\star)) \Lambda^\star_t(\mathrm{d}\nub^\star) \mathrm{d}t
        +
        \sigma(t,X^k_t) \mathrm{d}W_t
        +
        \sigma_0 \mathrm{d}B_t
    \end{align*}
    with $\mu^k_t:=\Lc(X^{k}_t|\Gc_t)$ and $\mub^k_t:=\Lc\big(X^{k}_{t},\beta^k(t,X^{k}_t, B,\Lambda^\star) | \Gc_t\big)$  then
    \begin{align*}
        \Lim_{k \to \infty} \big(\mu^k, \delta_{\mub^k_t}(\mathrm{d}\nub) \mathrm{d}t \big)= \big(\mu, \Lambda_t(\mathrm{d}\nub)\mathrm{d}t \big),\;\P\mbox{--a.e.},\;\mbox{for the Wasserstein metric}\;\Wc_p.
    \end{align*}

\end{proposition}

\begin{remark}
    When $\Lambda^\star$ is adapted to the canonical filtration of the Brownian motion $B$, $\beta^k$ which is a map of $(t,x,b,q)$ can be taken as a map of $(t,x,b)$ Lipschitz in $(x,b)$ uniformly in $t$ $($see {\rm \Cref{prop:approx_regularize}} $)$.
\end{remark}

\begin{proof}

We start by introducing some $``$shifted$"$ measures. Let us define, for all $(t,\mathsf{b},\pi,m) \in [0,T] \x \Cc^n \x \Cc^n_{\Wc} \x \Pc^n_U,$
\begin{align} \label{eq:shift-proba-initial}
    \pi_t[\mathsf{b}](\mathrm{d}y):= \int_{\R^n} \delta_{\big(y'+\sigma_0 \mathsf{b}(t) \big)}(\mathrm{d}y) \pi_t(\mathrm{d}y'),\;\; m[\mathsf{b}(t)](\mathrm{d}u,\mathrm{d}y):=\int_{\R^n \x U} \delta_{(y'+\sigma_0 \mathsf{b}(t))}(\mathrm{d}y)m(\mathrm{d}u,\mathrm{d}y')
\end{align}
and any $q \in \M,$
\begin{align} \label{eq:shift-proba-M}
        q_t[\mathsf{b}](\mathrm{d}m)\mathrm{d}t:=\int_{\Pc^n_U} \delta_{\big(\tilde m[\mathsf{b}(t)] \big)}(\mathrm{d}m) q_t(\mathrm{d}\widetilde m)\mathrm{d}t.
    \end{align}
In the same way, let us consider the {\color{black}``shifted''} generator $\widehat{\Lc},$
\begin{align} \label{eq:shift-generator}
        \widehat{\Lc}_t \varphi(y,\mathsf{b},\nub,\nub^\star,u):= \frac{1}{2}  \mathrm{Tr}\big[\sigma \sigma^\top(t,y+\sigma_0 \mathsf{b}(t)) \nabla^2 \varphi(y) \big] +\widehat{b}(t,y+\sigma_0\mathsf{b}(t),\nub[\mathsf{b}(t)],\nub^\star,u)^{\top} \nabla \varphi(y).
    \end{align}
Let us introduce $\vartheta_t:=\mu_t[-B]$ and $\Theta:=\Lambda[-B].$ The couple $(\vartheta,\Lambda)$ satisfies: $\P$--a.e. $\om \in \Om$,  $\Theta_t(\om)(\Z_{\vartheta_t(\om)})=1$ $\mathrm{d}t$--a.e. and
\begin{align*}
    \mathrm{d}\langle f,\vartheta_t(\om) \rangle
	~=~
	\int_{(\Pc^n_U)^2} \int_{\R^n \x U}  \widehat{\Lc}_t f(x,B(\om),\nub,\nub^\star,u) \nub(\mathrm{d}u,\mathrm{d}x)\Theta_t(\om)(\mathrm{d}\nub)\Lambda^\star_t(\om)(\mathrm{d}\nub)\mathrm{d}t.
\end{align*}

\medskip
By ??(see below), there is a sequence of maps $(\beta^k)_{k \ge 1}$ s.t. for each $k \ge 1$, $[0,T] \x \R^n \x \Cc^n \x \M \ni (t,x,b,q) \to \beta^k(t,x,b,q) \in U$ is Lipschitz in $(x,b,q)$ uniformly in $t$, $\beta^k(t,x,B,\Lambda^\star)=\beta^k(t,x,B_{t \wedge \cdot},\Lambda^\star_{t \wedge \cdot})$ and for $\P$--a.e. $\om \in \Om,$ we have, in weakly sense,
\begin{align*}
    \Lim_{k \to \infty} \delta_{H^k_t(\om)}(\mathrm{d}\nub)\mathrm{d}t
    =
    \Lambda_t(\om)[G](\mathrm{d}\nub)\mathrm{d}t\;\;\mbox{where}\;\;H^k_t(\om):=\delta_{\beta^k\big(t,x,\;B(\om),\;\Lambda^\star(\om)\big)}(\mathrm{d}u)G(x)\mathrm{d}x,
\end{align*}
$G$ is a density of probability continuous on $\R^n$ with $G>0.$ It is straightforward to check that $\P$--a.e. $\om \in \Om,$
\begin{align*}
    \Lim_{k \to \infty} \delta_{\tilde H^k_t(\om)}(\mathrm{d}\nub)\mathrm{d}t
    =
    \Theta_t(\om)[G](\mathrm{d}\nub)\mathrm{d}t\;\;\mbox{where}\;\;\widetilde{H}^k_t(\om):=\delta_{\beta^k\big(t,\;x+B_t(\om),\;B(\om),\;\Lambda^\star(\om)\big)}(\mathrm{d}u)G(x+\sigma_0B_t(\om))\mathrm{d}x.
\end{align*}

\medskip
By \Cref{prop:approximation_weak} (deterministic case), if we let $\Xt^{\om,k}:=\Xt^k$ be the unique strong solution of:
    \begin{align*}
        \mathrm{d}\Xt^{k}_t
        =
        \int_{\Pc^n_U}
        \widehat{b}\Big(t,\;\Xt^{k}_t+\sigma_0B_t(\om),\;\overline{m}^k_t(\om),\;\nub^\star,&\;{\beta}^k(t,\Xt^{k}_t+\sigma_0B_t(\om),\;\Lambda^\star(\om)) \Big) \Lambda^\star_t(\om)(\mathrm{d}\nub^\star) \mathrm{d}t
        \\
        &~~~~~~+
        \sigma(t,\Xt^k_t + \sigma_0 B_t(\om)) \mathrm{d}W_t
    \end{align*}
    with $m^k_t(\om):=\Lc(\Xt^{\om,k}_t)$ and $\overline{m}^k_t(\om):=\Lc\big(\Xt^{\om,k}_{t},{\beta}^k(t,\Xt^{\om,k}_t+\sigma_0B_t(\om),\Lambda^\star(\om))\big)$  then
    \begin{align*}
        \Lim_{k \to \infty} \big(m^k(\om), \delta_{\overline{m}^k_t(\om)}(\mathrm{d}\nub) \mathrm{d}t \big)= \big(\vartheta(\om), \Theta_t(\om)(\mathrm{d}\nub)\mathrm{d}t \big),\;\mbox{for the Wasserstein metric}\;\Wc_p.
    \end{align*}
    Now, let us introduce $X^k$ the solution of
    \begin{align*}
        \mathrm{d}X^{k}_t
        =
        \int_{\Pc^n_U}
        \widehat{b}(t,X^{k}_t,\mub^k_t,\nub^\star,\widetilde{\beta}^k(t,X^{k}_t,B,\Lambda^\star)) \Lambda^\star_t(\mathrm{d}\nub^\star) \mathrm{d}t
        +
        \sigma(t,X^k_t) \mathrm{d}W_t +\sigma_0 \mathrm{d}B_t\;\mbox{with}\;\mub^{k}_t:=\Lc\big(X^k_{t},\widetilde{\beta}^k(t,X^{k}_t,B,\Lambda^\star)| \Gc_t\big).
    \end{align*}
    By uniqueness (see \Cref{thm:unique_stochastic-FP} or \cite[Theorem 2.3]{lacker-strong-PC-2018}), it is straightforward to check that: $\P$--a.e. $\om,$
    $$
        \mu^k
        = m^k(\om)[B(\om)],\;\mbox{and}\;\delta_{\mub^k_t}(\mathrm{d}\nub) \mathrm{d}t
        = \delta_{\overline{m}^k_t(\om)[B(\om)]}(\mathrm{d}\nub) \mathrm{d}t,\;\;\P^{\Gc_T}_\om\;\mbox{--a.e}
    $$
    Since the function 
\begin{align*}
    (\pi,q,\mathsf{b}) \in \Cc^n_{\Wc} \x \M \x \Cc^n \to \big(\pi[\mathsf{b}],q_t[\mathsf{b}](\mathrm{d}m)\mathrm{d}t,\mathsf{b} \big) \in \Cc^n_{\Wc} \x \M \x \Cc^n
\end{align*}
is continuous, consequently, $\P$--a.e. $\om,$
\begin{align*}
    \Lim_{k \to \infty} \big(\mu^k(\om), \delta_{\mub^k_t(\om)}(\mathrm{d}\nub) \mathrm{d}t \big)
    =
    \Lim_{k \to \infty} \big(m^k(\om)[B(\om)], \delta_{\overline{m}^k_t(\om)[B(\om)]}(\mathrm{d}\nub) \mathrm{d}t \big)= \big(\vartheta'(\om)[B(\om)], \Theta_t(\om)[B(\om)](\mathrm{d}\nub)\mathrm{d}t \big)
    =
    (\mu,\Lambda).
\end{align*}
This is enough to conclude.
    
\end{proof}

{
\color{black}

\medskip

\begin{proposition} \label{prop:approx_regularize}
    There is a sequence of maps $(\beta^k)_{k \ge 1}$ s.t. for each $k \ge 1$, $[0,T] \x \R^n \x \Cc^n \x \M \ni (t,x,\bb,q) \to \beta^k(t,x,\bb,q) \in U$ is Lipschitz in $(x,\bb,q)$ uniformly in $t$, $\beta^k(t,x,B,\Lambda^\star)=\beta^k(t,x,B_{t \wedge \cdot},\Lambda^\star_{t \wedge \cdot})$ and for $\P$--a.e. $\om \in \Om,$ we have, in weakly sense,
\begin{align*}
    \Lim_{k \to \infty} \delta_{H^k_t(\om)}(\mathrm{d}\nub)\mathrm{d}t
    =
    \Lambda_t(\om)[G](\mathrm{d}\nub)\mathrm{d}t\;\;\mbox{where}\;\;H^k_t(\om):=\delta_{\beta^k\big(t,x,\;B(\om),\;\Lambda^\star(\om)\big)}(\mathrm{d}u)G(x)\mathrm{d}x,
\end{align*}
and $G$ is a density of probability continuous on $\R^n$ with $G>0.$ In addition, when $\Lambda^\star$ is adapted to the natural filtration of the Brownian motion $B$, $\beta^k$ can be taken as a map of $(t,x,\bb)$ Lipschitz in $(x,\bb)$ uniformly in $t$.
\end{proposition}

\begin{proof}
       $\underline{Step\;1:approximation\;of\;any\;elements\;of\;\M}$ Let $q \in \M$. Given $G,$ by an application of \cite[Lemma 3.1]{NEUFELD20143819}, there exists a Borel function $\Rc: \Pc^n_U \to \Pc^n_U$ satisfying: for all $\nub \in \Pc^n_U,$ $\Rc(\nub)=\nub^{x}(\mathrm{d}u)G(x)\mathrm{d}x.$
By \cite[Proposition C.1]{Lacker_carmona_delarue_CN} (see the construction in \cite[Theorem 2.2.3]{CastaingCharles2004YMoT}), there exists a sequence of Borel functions $(a^j)_{j \in \N^*}$ with for each $j \in \N^*,$ $a^j: [0,T] \x \R^n \x \M(\Pc^n_U) \to U$ s.t. for each $\nub \in \Pc^n_U,$
    \begin{align} \label{eq:first_app_control}
        \Lim_{j \to \infty} \delta_{a^j(t,x, \Rc(\nub) )}(\mathrm{d}u)G(x)\mathrm{d}x
        =
        \Rc(\nub)(\mathrm{d}x,\mathrm{d}u)
        =
       \nub^{x}(\mathrm{d}u)G(x)\mathrm{d}x,\;\mbox{in the weak sense.}
    \end{align}    
Also, there exists a sequence of Borel functions $(c^k)_{k \in \N^*}$ such that for each $k \in \N^*,$  $c^k:[0,T] \x \M(\Pc^n_U) \to \Pc^n_U$ is progressively measurable i.e. $c^k(t,q)=c^k(t,q_{t \wedge \cdot})$ and
    \begin{align} \label{eq:second_app_control}
        \Lim_{k \to \infty} \delta_{c^k (t, q_{t \wedge \cdot} )}(\mathrm{d}\nub)\mathrm{d}t
        =
       q_t(\mathrm{d}\nub)\mathrm{d}t\;\mbox{in weakly sense}.
    \end{align}
    \cite[Lemma 4.7]{el1987compactification} provides an example of construction guaranteeing the fact that $c^k$ is progressively measurable.
    Let us denote $h^k_t:=c^k (t, q_{t \wedge \cdot}).$ By combining \Cref{eq:first_app_control} and \Cref{eq:second_app_control}, if we set $\beta^{k,j}(t,x,q):=a^{j}(t,x,\Rc(c^k (t, q_{t \wedge \cdot}))),$ we find that $\beta^{k,j}$ is progressively measurable and
    \begin{align*}
        \Lim_{k \to \infty}
        \lim_{j \to \infty}\delta_{L^{k,j}_t}(\mathrm{d}\nub)\mathrm{d}t
        =
        q_t[G](\mathrm{d}\nu)\mathrm{d}t\;\;\mbox{where}\;\;L^{k,j}_t(\mathrm{d}x,\mathrm{d}u):=\delta_{\beta^{k,j}(t,x,q)}(\mathrm{d}u)G(x)\mathrm{d}x.
    \end{align*}
     
     $\underline{Step\;2:approximation\;of\;progressively\;measurable\;\beta}$ Let $\widetilde{\beta}:[0,T] \x \R^n \x \M \to U$ be a progressively measurable map. We know that $\Lambda$ is $(\Gc_t)_{t \in [0,T]}$--predictable process where $\Gc_t:=\sigma \{\Lambda^\star_{t \wedge \cdot}, B_{t \wedge \cdot} \}.$ Then, there is a progressively measurable map ${\beta}:[0,T] \x \R^n \x \Cc^n \x \M \to U$ satisfying $\widetilde{\beta}(t,x,\Lambda)=\beta(t,x,B,\Lambda^\star).$ Notice that when $\Lambda^\star$ is adapted to the filtration of $B$, the map $\beta$ can be chosen as a function of $(t,x,\bb).$ We will provide in the following the approximation for $\beta$ a map of $(t,x,\bb,q)$. The case $\beta$ a map of $(t,x,\bb)$ follows from this.

     \medskip
     Let $0=t_0^J<\cdots <t^J_J=T$ be a subdivision of $[0,T]$ s.t. $\Lim_{J \to \infty}\sup_{1 \le j \le J-1}|t^J_j-t^J_{j+1}|=0$. We introduce the notation for each $t \in (0,T]$, $[t]^J=t^J_j$ where $t^J_j < t \le t^J_{j+1}$ and $[0]^J=0$.
      For each $(\bb,q) \in \Cc^n \x \M$, and any $(\varepsilon,\delta) \in (0,1) \x (0,1)$, we set
      \begin{align*}
         \beta^{\varepsilon,\delta,J}(t,x,\bb,q)
          :=
          \frac{1}{\varepsilon}\int_{([t]^J-\varepsilon)\vee 0}^{[t]^J} \beta(t,y,\bb,q) G_{\delta}(x-y)\;\mathrm{d}y
      \end{align*}
      where $G_\delta$ is a convolution kernel on $\R^n$. It is classical to check that, for each $(\bb,q)$, 
      \begin{align*}
         \beta^{\varepsilon,\delta,J}(t,x,\bb,q)=\beta^{\varepsilon,\delta,J}(t,x,\bb_{t \wedge \cdot},q_{t \wedge \cdot})\;\;\mbox{and}\;\;\lim_{\varepsilon \to 0}\lim_{\delta \to 0}\lim_{J \to \infty}\beta^{\varepsilon,\delta,J}(t,x,\bb,q)=\beta(t,x,\bb,q)\;\mbox{a.e.}\;(t,x) \in [0,T] \x \R^n.
      \end{align*}
      
      Let us fix $\varepsilon>0$, $\delta >0$ and $J \in \N^*$.
      We observe that $\beta^{\varepsilon,\delta,J}(t,x,\bb,q)$ is piece--wise constant in $t$  i.e. $\beta^{\varepsilon,\delta,J}(t,x,\bb,q)=\beta^{\varepsilon,\delta,J}(t^J_j,x,\bb,q)$ for each $t^J_j < t \le t^J_{j+1}$, and Lipschitz in $x$ uniformly in $t$. We are going to approximate $\beta^{\varepsilon,\delta,J}$ by a sequence piece--wise constant in $x$. Let $(\kappa_i)_{i \ge 1}$ be a sequence of positive number s.t. $\lim_{i \to \infty} \kappa_i=0$. For each $i \ge 1$, we consider a sequence of countable disjoint balls $(U^i_\ell)_{\ell \ge 1}$ of radius $\kappa_i$, center in $x^i_\ell \in \R^n$, and $\cup_{\ell \ge 1} U^i_\ell=\R^n$. We define for each $i \ge 1$, $\beta^{\varepsilon,\delta,J,i}$
      \begin{align*}
          \beta^{\varepsilon,\delta,J,i}(t,x,\bb,q):=\beta^{\varepsilon,\delta,J}(t,x^i_\ell,\bb,q)\;\;\mbox{for}\;\;x \in U^i_\ell.
      \end{align*}
      Since $\beta^{\varepsilon,\delta,J}(t,x,\bb,q)$ is Lipschitz in $x$ uniformly in $t$  and $\lim_{i \to \infty} \kappa_i=0$, we get that, for each $(t,x,\bb,q) \in [0,T] \x \R^n \x \Cc^n \x \M$, 
      \begin{align*}
          \beta^{\varepsilon,\delta,J,i}(t,x,\bb,q)=\beta^{\varepsilon,\delta,J,i}(t,x,\bb_{t \wedge \cdot},q_{t \wedge \cdot})\;\;\mbox{and}\;\;\lim_{i \to \infty}\beta^{\varepsilon,\delta,J,i}(t,x,\bb,q):=\beta^{\varepsilon,\delta,J}(t,x,\bb,q).
      \end{align*}
      For each $i \ge 1$, $\beta^{\varepsilon,\delta,J,i}$ is piece--wise in $(t,x)$.
      
     \medskip 
     Since $B$ is a Brownian motion, for each $1 \le j \le J$, $\Lc(B_{t^J_t \wedge \cdot}, \Lambda^\star_{t^J_j \wedge \cdot})$ is non--atomic. We know that $\beta^{\varepsilon,\delta,J,i}(t^J_j,x^i_\ell,\bb,q)=\beta^{\varepsilon,\delta,J,i}(t^J_j,x^i_\ell,\bb_{t^J_j \wedge \cdot},q_{t^J_j \wedge \cdot})$. By combining the approximation of measurable function by continuous functions in \cite[Proposition C.1]{Lacker_carmona_delarue_CN} and the approximation of continuous function by Lipschitz functions in \cite[Theorem 6.4.1]{Cobza2019}, for each $1 \le j \le J$, each $i \ge 1$ and $\ell \ge 1$, there exists a sequence of Lipschitz maps $(H^{j,i,\ell,v})_{v \ge 1}$ s.t.  $H^{j,i,\ell,v}:\Cc^n \x \M \to U$ and 
      \begin{align*}
          \Lim_{v \to \infty}  H^{j,i,\ell,v}(B_{t^J_j \wedge \cdot},\Lambda^\star_{t^J_j \wedge \cdot})=\beta^{\varepsilon,\delta,J}(t^J_j,x^i_\ell,B_{t^J_j \wedge \cdot},\Lambda^\star_{t^J_j \wedge \cdot}),\;\;\P\mbox{--a.e.}
      \end{align*}
      There is a Borel set $A_{J,i}$ s.t. $\P(A_{J,i})=1$ and for each $\om \in A_{J,i}$, 
      \begin{align*}
          \Lim_{v \to \infty}  H^{j,i,\ell,v}(B_{t^J_j \wedge \cdot}(\om),\Lambda^\star_{t^J_j \wedge \cdot}(\om))=\beta^{\varepsilon,\delta,J}(t^J_j,x^i_\ell,B_{t^J_j \wedge \cdot}(\om),\Lambda^\star_{t^J_j \wedge \cdot}(\om))\;\mbox{for each }1 \le j \le J\;\mbox{and any}\;\ell \ge1.
      \end{align*}

      Notice that, for $u_0 \in U$, 
      \begin{align*}
          \lim_{K \to \infty}u_0 \mathbf{1}_{x \notin [-K,K]^n} 
          + \beta^{\varepsilon,\delta,J,i}(t,x,\bb,q) \mathbf{1}_{x \in [-K,K]^n}=\beta^{\varepsilon,\delta,J,i}(t,x,\bb,q).
      \end{align*}
      For each $K>0,$ we define $\Gamma^{\varepsilon,\delta,J,i,v,K}$ by: $\Gamma^{\varepsilon,\delta,J,i,v,K}(t,x,\bb,q)
          :=u_0\;\;\mbox{for}\;\;(t,x) \in [0,T] \x (\R^n \setminus [-K,K]^n)$ or $(t,x) \in \{0\} \x \R^n$,
      \begin{align*}
          \Gamma^{\varepsilon,\delta,J,i,v,K}(t,x,\bb,q)
          :=
          H^{j,i,\ell,v}(\bb_{t^J_j \wedge \cdot},q_{t^J_j \wedge \cdot})\;\mbox{for}\;(t,x)\in (t^J_j,t^J_{j+1}] \x (U^i_\ell \cap [-K,K]^n).
      \end{align*}
      Since the definition of $\Gamma^{\varepsilon,\delta,J,i,v,K}$ involves only a finite number of  $\{ H^{j,i,\ell,v},\;\ell \ge 1,\;1 \le j \le J\}$, the map $\Gamma^{\varepsilon,\delta,J,i,v}$ is Lipschitz in $(\bb,q)$ uniformly in $(t,x)$. $\Gamma^{\varepsilon,\delta,J,i,v,K}(t,x,\bb,q)=\Gamma^{\varepsilon,\delta,J,i,v,K}(t,x,\bb_{t \wedge \cdot},q_{t \wedge \cdot})$ and for each $\om \in A_{J,i}$, 
      \begin{align*}
          \Lim_{v \to \infty} \Lim_{K \to \infty}  \Gamma^{\varepsilon,\delta,J,i,v,K}(t,x,B(\om),\Lambda^\star(\om))=\beta^{\varepsilon,\delta,J,i}(t,x^i_\ell,B_{t^J_j \wedge \cdot}(\om),\Lambda^\star_{t^J_j \wedge \cdot}(\om))\;\mbox{for each }(t,x) \in (t^J_j,t^J_{j+1}] \x U^i_\ell.
      \end{align*}
      We set $A:=\cap_{J \ge 1,i \ge 1} A_{J,i}$. By combining all the result, we have $\P(A)=1$ and for each $\om \in A$, 
      \begin{align*}
          \Lim_{\varepsilon \to 0} \Lim_{\delta \to 0} \Lim_{J \to \infty} \Lim_{i \to \infty} \Lim_{v \to \infty} \Lim_{K \to \infty} \Gamma^{\varepsilon,\delta,J,i,v,K}(t,x,B(\om),\Lambda^\star(\om))=\beta(t,x,B(\om),\Lambda^\star(\om))\;\mbox{for a.e. }(t,x) \in [0,T] \x \R^n.
      \end{align*}
      To recover the Lipschitz continuity in $x$, it is enough to do another regularization by convolution. 
      Therefore, we find a sequence $(\beta^k)_{k \ge 1}$ of Lipschitz functions in $(x,\bb,q)$ unifformly in $t$ s.t. $\beta^k(t,x,B,\Lambda^\star)=\beta^k(t,x,B_{t \wedge \cdot},\Lambda^\star_{t \wedge \cdot})$, $\P(A)=1$, and for $\om \in A,$ we have
      \begin{align*}
          \lim_{k \to \infty} \beta^k(t,x,B(\om),\Lambda^\star(\om))=\beta(t,x,B(\om),\Lambda^\star(\om))\;\mbox{a.e. }(t,x) \in [0,T] \x \R^n.
      \end{align*}

    \medskip
    The proof of the proposition is just a combination of step $1$ and step $2$.
\end{proof}
\begin{remark}
    Notice that, as we saw in the previous proof, we can take $\beta$ Lipschitz in $t$ as well. We chose the regularity mentioned in the proposition because it is largely enough to establish our results. 
\end{remark}

}

\subsection{Uniqueness and equivalence of filtration}

In this section, we first provide the uniqueness property of the stochastic Fokker--Planck equation that we use in this paper. In a second time, we show an approximation result of the Brownian motion $B$ by a sequence of Fokker--Planck equations. 

\medskip
Recall that $(\Om,\F,\Fc,\P)$ is a given probability space. On this space, we have an $\R^{n \x n}$--valued $\F$--Brownian motion $(W,B)$ and a $\Fc_0$--random variable $X_0$ with $\Lc(X_0)=\nu.$ Let us set $\Gt:=(\Gct_t)_{t \in [0,T]} \subset \F$ a sub--filtration s.t. $B$ is $\Gt$ adapted and, $(W,X_0)$ and $\Gt$ are independent. Besides, $\Gt$ satisfies the $($H$)$--hypothesis with $\F$ i.e. for every $t \in [0,T],$ for each $A \in \Fc_t,$ one has $\E[\mathbf{1}_A|\Gct_t]=\E[\mathbf{1}_A|\Gct_T].$ We give ourselves the bounded Borel map 
$$
    \Phi: [0,T] \x \R^n \x C([0,T];\Pc_p(\R^n)) \x \Cc^\ell \ni (t,x,\pi,\mathsf{b}) \to \Phi(t,x,\pi_{t \wedge \cdot},\mathsf{b}(t \wedge \cdot)) \in \R^n.
$$

\medskip
Let $(X,\mu)$ be a solution of 
\begin{align} \label{eq:uniq-mc_vl}
    \mathrm{d}X_t
    =
    \Phi(t,X_t,\mu,B) \mathrm{d}t 
    +
    \sigma(t,X_t)\mathrm{d}W_t
    +
    \sigma_0 \mathrm{d}B_t\;\;\mbox{with}\;\;\mu_t=\Lc(X_t|\Gct_t)=\Lc(X_t|\Gct_T).
\end{align}

Recall that the sub--filtration $\Gt$ is fixed. We are first interested in the uniqueness of the conditional law of $X_t$ given the $\sigma$--field $\Gct_t$ i.e. $\Lc(X_t|\Gct_t).$ We say that there is uniqueness of the marginal law of \eqref{eq:uniq-mc_vl} if for $(X^1,\mu^1)$ and $(X^2,\mu^2)$ s.t. for $i=1,2,$ $(X^i,\mu^i)$ is solution of \eqref{eq:uniq-mc_vl} then $\mu^1=\mu^2$ $\P$--a.s.

\begin{proposition} \label{thm:unique_stochastic-FP}
    If $\Phi$ is uniformly Lipschitz in $\pi,$ the uniqueness of the marginal law of \eqref{eq:uniq-mc_vl} is true.
\end{proposition}

\begin{proof}
     
Let us define, for all $(t,x,\mathsf{b},\pi) \in [0,T] \x \R^n \x \Cc^n \x \Cc^n_{\Wc},$
\begin{align*} 
    \pi_t[\mathsf{b}](\mathrm{d}y):= \int_{\R^n} \delta_{\big(y'+\sigma_0 \mathsf{b}(t) \big)}(\mathrm{d}y) \pi_t(\mathrm{d}y')\;\mbox{and}\;\widetilde{\Phi}(t,x,\pi,\mathsf{b})
    :=
    \Phi(t,x+\sigma_0 \mathsf{b}(t),\pi[\mathsf{b}],\mathsf{b}).
\end{align*}
Then, for each $\mathsf{b} \in \Cc^n,$ by \cite[Theorem 2.3]{lacker-strong-PC-2018}, the equation
\begin{align} \label{eq:uniq-mc_vl_nocommonnoise}
    \mathrm{d}\Xt_t
    =
    \widetilde{\Phi}(t,\Xt_t,m,\mathsf{b}) \mathrm{d}t 
    +
    \sigma(t,\Xt_t + \sigma_0 \mathsf{b}(t))\mathrm{d}W_t
    \;\;\mbox{with}\;\;{\color{black}m_t=\Lc(X_t)}
\end{align}
is unique in law. Therefore, for $(X^1,\mu^1)$ and $(X^2,\mu^2)$ two solution of \eqref{eq:uniq-mc_vl}, one has, by uniqueness of \Cref{eq:uniq-mc_vl_nocommonnoise}, that $\Lc(X^1_t-\sigma_0 B_t|\Gct_T)=\Lc(X^2_t- \sigma_0 B_t|\Gct_T,)$ $\P$--a.e. for each $t \in [0,T].$ This allow us to deduce our uniqueness result.
\end{proof}

Notice that, this proof shows that $\mu$ the unique conditional distribution given in \Cref{eq:uniq-mc_vl_nocommonnoise} is in fact adapted to the filtration of $B.$

\medskip
For each $k \in \N^*,$ we denote by  $[t]^k=t^k_\ell$ for all $t \in (t^k_\ell,t^k_{\ell+1}]$ and $[0]^k=0$ where $0=t^k_0<\cdots<t^k_k=T$ is a subdivision such that $\Lim_{k \to \infty}\; \sup_{\ell \in \{1,\cdots,k \}}\;\big|t^k_\ell-t^{k-1}_{\ell}\big|=0.$
\begin{proposition} \label{prop:appr_general_closed}
    Let us assume that $\Phi$ is uniformly Lipschitz in $(x,\pi,\mathsf{b}).$ There exists a sequence of functions $(\phi^k)_{k \in \N^*}$ s.t. for each $k \in \N^*,$ $\phi^k: [0,T] \x \Cc^{n,p}_{\Wc} \ni (t,\pi) \to \phi^k(t,\pi_{t \wedge \cdot}) \in  \Cc^n$ is Lipschitz in $\pi$ uniformly in $t,$ and satisfies
    \begin{align*}
        B_t
        =
        \phi^k(t,\mu^k)
    \end{align*}
    where $\mu^k_t:=\Lc(X^k_t|\Fc^B_t)$ with $X^k$ satisfying
\begin{align*} 
    \mathrm{d}X^k_t
    =
    \Phi\big(t,X^k_t,\mu^k,\phi^k([t]^k,\mu^k) \big) \mathrm{d}t 
    +
    \sigma(t,X^k_t)\mathrm{d}W_t
    +
    \sigma_0 \mathrm{d}B_t
\end{align*}
 and
\begin{align*}
    \Lim_{k \to \infty} \E \bigg[\int_0^T \| \mu_t-\mu^k_t \|_{\rm TV}\; \mathrm{d}t\bigg]=0.
\end{align*}
\end{proposition}

\begin{remark} \label{rm:equal-filtration}
    Notice that, this result leads to the observation that: for each $k \in \N^*,$ the natural filtration of $(\mu^k_t)_{t \in [0,T]}$ is equal to the natural filtration of $(B_t)_{t \in [0,T]}.$
\end{remark}
\begin{proof}
       \underline{$Step\;1$: $piece$--$wise$ $case$} Let us assume first that $\Phi$ is s.t. $\Phi(t,x,\pi,\mathsf{b})=\Phi\big(t,x,\pi,\mathsf{b}([t]^k \wedge \cdot) \big)$ for some $k \in \N^*.$
       Let $\mu$ be the unique solution of 
\begin{align*} 
    \mathrm{d}X_t
    =
    \Phi(t,X_t,\mu,B) \mathrm{d}t 
    +
    \sigma(t,X_t)\mathrm{d}W_t
    +
    \sigma_0 \mathrm{d}B_t\;\;\mbox{with}\;\;\mu_t=\Lc(X_t|\Gct_t)=\Lc(X_t|\Gct_T).
\end{align*}
Notice that $\mu$ is adapted to the filtration of $B.$ Now, by recurrence, we construct the function $\phi^k.$ Let $\ell=0,$ for each $t \in (0,t^k_1],$ by taking the conditioning expectation
\begin{align*}
    \sigma_0B_t
    =
    \int_{\R^n} x \mu_t(\mathrm{d}x)
    -
    \int_{\R^n} x \mu_0(\mathrm{d}x)
    -
    \int_0^t \int_{\R^n}  \Phi(s,x,\mu_{s \wedge \cdot},0) \mu_s(\mathrm{d}x) \mathrm{d}s.
\end{align*}
Therefore, one has $B_t=\phi_0(t,\mu)$ where
\begin{align*}
    \phi_0(t,\pi)
    :=
    \sigma^{-1}_0 \bigg[\int_{\R^n} x \pi_t(\mathrm{d}x)
    -
    \int_{\R^n} x \pi_0(\mathrm{d}x)
    -
    \int_0^t \int_{\R^n}  \Phi(s,x,\pi_{s \wedge \cdot},0) \pi_s(\mathrm{d}x) \mathrm{d}s \bigg].
\end{align*}    
    Notice that, as $\Phi$ is uniformly Lipschitz in $(x,\pi,\mathsf{b}),$ $\phi_0:[0,T] \x \Cc^{n,p}_{\Wc} \to \Cc^n$ is Lipschitz in $\pi$  uniformly in $t.$ For $\ell \in \{0,\cdots,k-1\},$ let us assume that we construct the functions $(\phi_0,\cdots,\phi_\ell),$ a sequence of Lipschitz  functions in $\pi$  uniformly in $t.$ Now, for $t \in (t^k_\ell,t^k_{\ell+1}],$
\begin{align*}
    \sigma_0B_t
    =
    \int_{\R^n} x \mu_t(\mathrm{d}x)
    -
    \int_{\R^n} x \mu_{t^k_\ell}(\mathrm{d}x)
    -
    \int_{t^k_\ell}^t \int_{\R^n}  \Phi\big(s,x,\mu_{s \wedge \cdot},\phi_\ell(t^k_\ell,\mu) \big) \mu_s(\mathrm{d}x) \mathrm{d}s.
\end{align*}    
Then, $B_t=\phi_{\ell+1}(t,\mu)$ where
\begin{align*}
    \phi_{\ell+1}(t,\pi)
    :=
    \sigma^{-1}_0 \bigg[\int_{\R^n} x \pi_t(\mathrm{d}x)
    -
    \int_{\R^n} x \pi_{t^k_\ell}(\mathrm{d}x)
    -
    \int_{t^k_\ell}^t \int_{\R^n}  \Phi\big(s,x,\pi_{s \wedge \cdot},\phi_\ell(t^k_\ell,\pi) \big) \pi_s(\mathrm{d}x) \mathrm{d}s \bigg].
\end{align*} 
We then construct by recurrence the functions $(\phi_0,\cdots,\phi_k).$ If we define $\phi^k(t,\pi):=\phi_\ell(t,\pi)$ for each $t \in [t^k_\ell,t^k_{\ell+1}]$ and all $\ell \in \{0,\cdots,k-1\}.$ We find the function we are looking for.

\medskip
\underline{$Step\;2$: $General$ $case$} By continuity of $\Phi$ in $\mathsf{b},$ we know that $\Lim_{k \to \infty} \Phi(t,x,\pi,\mathsf{b})=\Phi\big(t,x,\pi,\mathsf{b}([t]^k \wedge \cdot) \big).$ For each $k \in \N^*,$ let $\mu^k$ verify
\begin{align*} 
    \mathrm{d}X^k_t
    =
    \Phi\big(t,X^k_t,\mu^k,B_{[t]^k \wedge \cdot} \big) \mathrm{d}t 
    +
    \sigma(t,X^k_t)\mathrm{d}W_t
    +
    \sigma_0 \mathrm{d}B_t\;\;\mbox{with}\;\;\mu^k_t=\Lc(X^k_t|\Gct_t)=\Lc(X^k_t|\Gct_T).
\end{align*}
By the previous case, there exists $\phi^k$ Lipschitz in $\pi$ uniformly in $t$ s.t. $B_t=\phi^k(t,\mu^k).$ To finish, by (an obvious extension in the stochastic case of) \Cref{cor:converg_density} , we check that $\Lim_{k \to \infty} \E \bigg[\int_0^T \| \mu_t-\mu^k_t \|_{\rm TV}\; \mathrm{d}t\bigg]=0.$

\end{proof}

\end{appendix}

\end{document}